\documentclass[a4paper, 11pt]{amsart}

\usepackage{amsmath, amsthm, amssymb}
\usepackage{enumitem}
\usepackage{xcolor}
\usepackage{hyperref}
\hypersetup{
    colorlinks,
    linkcolor={red!50!black},
    citecolor={blue!50!black},
    urlcolor={blue!80!black}
}
\usepackage{geometry} 
\usepackage{tikz-cd}
\usepackage{stackengine}

\usepackage{overpic}
\usepackage[ansinew]{inputenc}
\usepackage{graphicx}
\usepackage[rightcaption]{sidecap}
\usepackage{caption}
\usepackage{subcaption}

\newtheorem{theorem}{Theorem}[section]
\newtheorem{corollary}[theorem]{Corollary}
\newtheorem{lemma}[theorem]{Lemma}
\newtheorem{proposition}[theorem]{Proposition}

\newtheorem{claim}[theorem]{Claim}

\newtheorem*{question}{Question}
\newtheorem{thmintro}{Theorem}

\theoremstyle{definition}
\newtheorem{definition}[theorem]{Definition}
\newtheorem{remark}[theorem]{Remark}

\newtheorem*{convention}{Convention}

\makeatletter
\newcommand{\setword}[2]{%
  \phantomsection
  #1\def\@currentlabel{\unexpanded{#1}}\label{#2}%
}
\makeatother

\newcommand{\eps}{\epsilon}

\newcommand{\wt}[1]{\widetilde{#1}}
\newcommand{\R}{\mathbb R}

\newcommand{\ZZ}{\mathbb Z}
\newcommand{\T}{\mathbb T}

\newcommand{\cT}{\mathcal{T}}
\newcommand{\cW}{\mathcal{W}}
\newcommand{\cF}{\mathcal{F}}
\newcommand{\cL}{\mathcal{L}}
\newcommand{\cC}{\mathcal{C}}
\newcommand{\cS}{\mathcal{S}}

\newcommand{\cB}{\mathcal{B}}

\newcommand{\mt}{\widetilde M}
\newcommand{\ft}{\widetilde f}

\newcommand{\fbs}{\cW^{cs}}
\newcommand{\fbu}{\cW^{cu}} 
\newcommand{\ftk}{\widetilde {f^k}}
\newcommand{\wfbs}{\widetilde \cW^{cs}}
\newcommand{\wfbu}{\widetilde \cW^{cu}}
\newcommand{\fes}{\cW^{cs}_{\epsilon}}
\newcommand{\feu}{\cW^{cu}_{\epsilon}} 
\newcommand{\wfes}{\widetilde \cW^{cs}_{\epsilon}}
\newcommand{\wfeu}{\widetilde \cW^{cu}_{\epsilon}}
\newcommand{\hs}{h^{cs}_{\epsilon}}
\newcommand{\hu}{h^{cu}_{\epsilon}}

\newcommand{\fol}{\cF}
\newcommand{\fole}{\cF_{\epsilon}}
\newcommand{\fn}{\widetilde \cF}

\newcommand{\lcse}{{\mathcal L}^{cs}_{\epsilon}}
\newcommand{\lcue}{{\mathcal L}^{cu}_{\epsilon}}
\newcommand{\lcsb}{{\mathcal L}^{cs}}
\newcommand{\lcub}{{\mathcal L}^{cu}}

\newcommand{\cG}{\mathcal G}

\newcommand{\wcs}{\widetilde \cW^{cs}}
\newcommand{\wcu}{\widetilde \cW^{cu}}

\newcommand{\pe}{\Phi_{\epsilon}}

\newcommand{\rquotient}[2]{{\left.\raisebox{.1em}{$#1$}\middle/\raisebox{-.1em}{$#2$}\right.}}

\DeclareMathOperator{\id}{Id}

\title[Partial hyperbolicity in 3 manifolds II]{{\small Partially hyperbolic diffeomorphisms homotopic to 
the identity in dimension 3} \\ {\small  Part II: Branching foliations}}
\author[T. Barthelm\'e]{Thomas Barthelm\'e}
\address{Queen's University, Kingston, ON}
\email{thomas.barthelme@queensu.edu}
\urladdr{sites.google.com/site/thomasbarthelme}

\author[S.R. Fenley]{Sergio R.\ Fenley} 
\address{Florida State University, Tallahassee, FL 32306}
\email{fenley@math.fsu.edu}

\author[S. Frankel]{Steven Frankel} 
\address{Washington University in St.~Louis, St.~Louis, Mo}
\email{steven.frankel@wustl.edu}

\author[R. Potrie]{Rafael Potrie} 
\address{Centro de Matem\'atica, Universidad de la Rep\'ublica, Uruguay}
\email{rpotrie@cmat.edu.uy}
\urladdr{http://www.cmat.edu.uy/~rpotrie/}

\begin{document}
 
 \begin{abstract}
 We study $3$-dimensional partially hyperbolic diffeomorphisms that are homotopic to the identity, focusing on the geometry and dynamics of Burago and Ivanov's center stable and center unstable \emph{branching} foliations. This extends our study of the true foliations that appear in the dynamically coherent case \cite{BFFP-prequel}. We complete the classification of such diffeomorphisms in Seifert fibered manifolds. In hyperbolic manifolds, we show that any such diffeomorphism is either dynamically coherent and has a power that is a discretized Anosov flow, or is of a new potential class called a \emph{double translation}.
 \bigskip
 
 
 \medskip
 
\end{abstract}

 \maketitle

\tableofcontents

\section{Introduction}\label{s.introduction}

%

A diffeomorphism $f$ of a $3$-manifold $M$ is \emph{partially hyperbolic} if it preserves a splitting of the tangent bundle $TM$ into three $1$-dimensional sub-bundles
\[ TM = E^s \oplus E^c \oplus E^u, \]
where the \emph{stable bundle} $E^s$ is eventually contracted, the \emph{unstable bundle} $E^u$ is eventually expanded, and the \emph{center bundle} $E^c$ is distorted less than the stable and unstable bundles at each point. That is, for some $n>0$ one has
\begin{align*}
\|Df^n|_{E^s(x)}\| &< 1, \\
\|Df^n|_{E^u(x)}\| &> 1, \text{ and}\\
\|Df^n|_{E^s(x)}\| < \|Df^n|_{E^c(x)}\| &< \|Df^n|_{E^u(x)}\|,
\end{align*}
at each $ x\in M$.

From a geometric perspective, one can think of partial hyperbolicity as a generalization of the discrete behavior of an Anosov flow. On a $3$-manifold $M$, such a flow $\Phi$ preserves a splitting of the unit tangent bundle $TM$ into three $1$-dimensional sub-bundles
\[ TM = E^s \oplus T\Phi \oplus E^u, \]
where $E^s$ is eventually exponentially contracted, $E^u$ is eventually exponentially expanded, and $T\Phi$ is the tangent direction to the flow. After flowing for a fixed time, an Anosov flow generates a partially hyperbolic diffeomorphism of a particularly simple type, where the stable and unstable bundles are contracted uniformly, and the center direction, which corresponds to $T\Phi$, is left undistorted. More generally, there are examples of partially hyperbolic diffeomorphisms of the form $f(x) = \Phi_{\tau(x)}(x)$ where $\Phi$ is a (topological) Anosov flow and $\tau\colon M \to \mathbb{R}_{>0}$ is a positive continuous function; the partially hyperbolic diffeomorphisms obtained in this way are called \emph{discretized Anosov flows}.

A partially hyperbolic diffeomorphism is said to be \emph{dynamically coherent} if there are invariant foliations tangent to the center stable and center unstable bundles $E^c \oplus E^s$ and $E^c \oplus E^u$. Discretized Anosov flows are dynamically coherent, since their center stable and center unstable bundles are uniquely integrable. On the other hand, we show in \cite{BFFP-prequel} that large classes of dynamically coherent partially hyperbolic diffeomorphisms must in fact be discretized Anosov flows:

\begin{theorem}[ {\cite[Theorem A]{BFFP-prequel}} ]\label{theorem:prequelA}
	Let $f\colon M \to M$ be a dynamically coherent partially hyperbolic diffeomorphism on a closed Seifert fibered $3$-manifold. If $f$ is homotopic to the identity, then some iterate is a discretized Anosov flow. 
\end{theorem}

\begin{theorem}[{\cite[Theorem B]{BFFP-prequel}} ]\label{theorem:prequelB}
	Let $f\colon M \to M$ be a dynamically coherent partially hyperbolic diffeomorphism on a closed hyperbolic $3$-manifold. Then some iterate is a discretized Anosov flow.
\end{theorem}

The assumption of dynamical coherence is natural from a geometric perspective: the way that an Anosov flow distorts its weak stable and weak unstable foliations is often seen as the defining property of such a flow. In this light, the preceding results say that on certain classes of manifolds, any diffeomorphism with a geometric structure reminiscent to that of an Anosov flow must in fact come from one. 

This assumption is much less satisfying from a dynamical perspective, however. Here the interest in partial hyperbolicity stems from its appearance as a generic consequence of dynamical conditions, such as stable ergodicity and robust transitivity (see \cite{BDV}), and one is not provided with any invariant foliations. Although dynamical coherence was once generally expected, a number of recent results (see, e.g., \cite{HHU-noncoherent,BGHP, BFFP_companion}) have shattered that belief. For instance, in the unit tangent bundle of a hyperbolic surface, we proved in \cite{BFFP_companion} that many partially hyperbolic diffeomorphisms are not dynamically coherent.

In our study of the dynamically coherent case \cite{BFFP-prequel}, the key to relating the inherently local property of partial hyperbolicity with the global structure of the ambient manifold lay in understanding the geometry and topology of the center stable and center unstable foliations, as well as their leafwise and transverse dynamics. The present article does away with the assumption of dynamical coherence. Instead of foliations we work with the center stable and center unstable ``branching foliations'' constructed by Burago and Ivanov \cite{BI} under certain orientability conditions. These are generalizations of foliations in which distinct leaves are allowed to merge together.

A large part of the present paper is concerned with carrying over our understanding of the geometry of foliations to branching foliations. We find that much of the familiar structure still holds in this more general context -- sometimes by direct analogy, and sometimes with considerably more work. At the same time, there are important points at which branching foliations allow for more varied behavior than true foliations. A particularly important example of this appears in Figure~\ref{figureCancelation}, where the possibility of merging leaves thwarts one's ability to use the qualitative transverse and tangent behavior of a dynamical system to draw conclusions about its Lefschetz index. We hope that our work will entice those interested in the theory of foliations to consider the possible uses for branching foliations.

The following two theorems, which generalize the preceding theorems from \cite{BFFP-prequel}, summarize the major consequences of the present article.

\begin{thmintro}\label{thmintro:Seifert}
	Let $f\colon M \to M$ be a partially hyperbolic diffeomorphism on a closed Seifert fibered $3$-manifold. If $f$ is homotopic to the identity, then it is dynamically coherent, and some iterate is a discretized Anosov flow. 
\end{thmintro}

This is a stronger version of Theorem \ref{theorem:prequelA}, without the \emph{a priori} assumption of dynamical coherence. The following corresponds to  Theorem \ref{theorem:prequelB}.

\begin{thmintro}\label{thmintro:Hyperbolic}
	Let $f\colon M \to M$ be a partially hyperbolic diffeomorphism on a closed hyperbolic $3$-manifold. Then either
	\begin{enumerate}[label=\rm{(\roman*)}]
		\item $f$ is dynamically coherent, some iterate is a discretized Anosov flow; or
		\item\label{it.DT_in_hyperbolic} $f$ is not dynamically coherent, and after taking a finite cover\footnote{This is only needed to get the existence of $f$-invariant branching foliations.} and iterate, it has center stable and center unstable branching foliations which are $\R$-covered and uniform, and a lift of $f$ acts as a nontrivial translation on both of the corresponding leaf spaces.
	\end{enumerate}
\end{thmintro}

The existence or non-existence of examples of type \ref{it.DT_in_hyperbolic} is one of the major questions coming out of this article. See \S\ref{ss.DT}. 

Let us also mention a dynamical consequence of our analysis (Corollary \ref{c.contractible}).

\begin{theorem}\label{thm-nocontractible}
	Let $f \colon M \to M$ be a partially hyperbolic diffeomorphism of a closed 3-manifold $M$ that is homotopic to the identity. If either $M$ is hyperbolic or Seifert fibered, or the center stable or center unstable branching foliation is $f$-minimal, then $f$ has no contractible periodic points (see Definition \ref{def.contractible_periodic_point}).
\end{theorem}

\subsection{Acknowledgments}
We thank C. Bonatti, A. Gogolev and A. Hammerlindl for interesting discussions. We also thank the referee for their careful reading and thoughtful suggestions which led to significant improvements, especially in \S\ref{sss.minimal}.

T.~Barthelm\'e was partially supported by the NSERC (Funding reference number RGPIN-2017-04592).

S.~Fenley was partially supported by Simons Foundation grant number 280429.

S.~Frankel was partially supported by National Science Foundation grant number~DMS-1611768.

R.~Potrie was partially supported by CSIC~618 and ANII--FCE--135352.

\section{Outline and discussion} \label{s.plan_of_proof}


After recalling some definitions, we outline the more detailed results that lie behind our main theorems.

Let $f\colon M \to M$ be a partially hyperbolic diffeomorphism that is homotopic to the identity on a closed $3$-manifold $M$.

\vspace{.1cm}
\fbox{\begin{minipage}{.95\textwidth}
		\textbf{Convention:}  Throughout this paper we will assume that $\pi_1(M)$ is not virtually solvable.
\end{minipage}}
\vspace{.1cm}

Although this assumption is not always necessary, it will simplify certain parts of the exposition. It does not result in loss of generality, since partially hyperbolic diffeomorphisms have been completely classified in manifolds with solvable or virtually solvable fundamental group \cite{HP-Nil,HP-Sol}.

A foundational result of Burago and Ivanov (Theorem \ref{teo-BI}) implies that, after passing to an appropriate finite power and lift, we can assume that there is a pair of ``branching foliations'' $\fbs$ and $\fbu$ that are preserved by $f$ and tangent to the center stable and center unstable bundles $E^c \oplus E^s$ and $E^c \oplus E^u$.

We outline the theory of these branching foliations in \S\ref{sec-branching}, and construct corresponding \emph{leaf spaces} $\lcsb$ and $\lcub$. Like the leaf spaces of true foliations, these are simply-connected, possibly non-Hausdorff $1$-manifolds that capture the transverse structure of $\wfbs$ and $\wfbu$, the lifts of $\fbs$ and $\fbu$ to the universal cover. This is where a large part of our work takes place, studying the dynamics of the following important class of lifts of $f$.

\begin{definition}\label{d.goodlift}
	A lift of $f$ to the universal cover is called \emph{good} if it moves each point a uniformly bounded distance and commutes with every deck transformation.
\end{definition}

Since $f$ is homotopic to the identity, it has at least one good lift, obtained by lifting such a homotopy.

\begin{remark}
The diffeomorphisms we consider are in fact \emph{isotopic} to identity: Indeed, all the manifolds that appear in this article are irreducible and covered by $\R^3$. Then, the works of many authors (Waldhausen \cite{Wald} for Haken manifolds, Boileau--Otal \cite{BoileauOtal} for Seifert manifolds and Gabai--Meyerhoff--Thurston \cite{GMT} for hyperbolic manifolds) give that homotopy implies isotopy. We will however not use this fact in the sequel, as the existence of a good lift is all that we use.
\end{remark}

\subsubsection{Dynamics on leaf spaces}\label{sss-introLeafSpaces}
In \S\ref{sec-branching_goodlift}, we study the way that good lifts of $f$ permute the leaves of the lifted center stable and center unstable branching foliations, and the implications for the structure of their leaf spaces. This extends \cite[\S~3]{BFFP-prequel}.

The picture is particularly simple when $\fbs$ is $f$-minimal, which means that the only closed, non empty, $f$-invariant set which is a union of leaves is $M$ itself.
\begin{tabular}{cl}
	\setword{$(\star)$}{star_dichotomy}
	&%
	\begin{minipage}{.85\textwidth}
		\vspace{.1cm}
		If $\fbs$ is $f$-minimal, then:
		\begin{itemize}
			\item Each good lift $\ft$ fixes either every leaf or no leaf of $\wfbs$.
			
			\item If some good lift $\ft$ fixes no leaf, then $\fbs$ is $\R$-covered and uniform, and $\ft$ acts as a translation its leaf space.
		\end{itemize}
	\end{minipage}
\end{tabular}
\vspace{.1cm}

The same holds for $\wfbu$. In particular, if both $\fbs$ and $\fbu$ are $f$-minimal, then one of the following holds for each good lift $\ft$ of $f$:

\begin{enumerate}[label=(\arabic*)]
	\item\label{it.DI_dinv} {\bf double invariance:} $\ft$ fixes every leaf of both $\wfbs$ and $\wfbu$;
	\item\label{it.DI_nomix} {\bf mixed behavior:} $\ft$ fixes every leaf of either $\wfbs$ or $\wfbu$, and acts as a translation on the leaf space of the other, or
	\item\label{it.DI_dtran} {\bf double translation:} $\ft$ acts as a translation on the leaf spaces of both $\wfbs$ and $\wfbu$. 
\end{enumerate}

This trichotomy applies whenever $f$ is transitive or volume-preserving, where the associated branching foliations are always $f$-minimal \cite{BW}.

When $f$ is a discretized Anosov flow, there is a natural homotopy from the identity to $f$ that moves points along the orbits of the underlying flow. The good lift $\ft$ that comes from lifting this homotopy fixes every center leaf. In order to show that a given partially hyperbolic diffeomorphism is a discretized Anosov flow, we will need to find a good lift with this property. Here, one takes the center leaves to be the components of intersections between center stable and center unstable leaves. In particular, we will need find a good lift with doubly invariant behavior.

\subsubsection{Center dynamics in fixed leaves}
In \S\ref{sec-centerdynamics}, we study the dynamics of the center foliation within center stable  and center unstable leaves. We obtain the following crucial tool (See Definition~\ref{d.coarsely_contracting} and Proposition \ref{p.alternnonDC}):\\
\begin{tabular}{cl}
	\setword{$(\star\star)$}{starDI_contractingcenters}
	&%
	\begin{minipage}{.85\textwidth}
		\vspace{.1cm}
		Suppose that $\fbs$ is $f$-minimal, and that some good lift $\ft$ fixes every center stable leaf but no center leaf in $\mt$. Then every $f$-periodic center leaf in $M$ is coarsely contracted.
	\end{minipage}
\end{tabular}
\vspace{.1cm}

If one replaces $\fbs$ with $\fbu$ then one concludes that any $f$-periodic center leaf in $M$ is coarsely expanded. This is widely applicable since one can find a periodic center leaf on any center stable or center unstable leaf with non-trivial fundamental group (Proposition \ref{periodiccenter}).

\begin{remark}
	In the dynamically coherent case, \ref{starDI_contractingcenters} leads to a contradiction that yields a \emph{fixed} center leaf \cite[Proposition~4.4]{BFFP-prequel}. In \S\ref{ss.AbsolutePH} we show that this holds as well under the assumption of \emph{absolute} partial hyperbolicity.
\end{remark}

\subsubsection{Minimality in hyperbolic and Seifert fibered manifolds}
In \S\ref{sss.minimal}, we show the following, which means that the preceding trichotomy holds whenever the ambient manifold is hyperbolic or Seifert fibered.

\begin{tabular}{cl}
	\setword{$(\star')$}{star_dichotomyHypSeif}
	&%
	\begin{minipage}{.85\textwidth}
		\vspace{.1cm}
		If $M$ is hyperbolic or Seifert fibered, then: 
		\begin{itemize}
			\item Each good lift $\ft$ fixes either every leaf or no leaf of $\wfbs$.
			
			\item If some good lift $\ft$ fixes every leaf, then $\fbs$ is $f$-minimal.
			
			\item If some good lift $\ft$ fixes no leaf, then $\fbs$ is $\R$-covered and uniform, and $\ft$ acts as a translation on its leaf space.
		\end{itemize}
	\end{minipage}
\end{tabular}
\vspace{.1cm}

\subsubsection{Double invariance implies dynamical coherence}
In \S\ref{sec-doubleinvndc} we prove the following criterion for when a partially hyperbolic diffeomorphism is a discretized Anosov flow:
\begin{theorem}\label{thm.DIimpliesDC}
	Let $f\colon M \to M$ be a partially hyperbolic diffeomorphism that is homotopic to the identity. If $f$ admits $f$-minimal center stable and center unstable branching foliations, and some good lift $\ft$ has doubly invariant behavior, then $f$ is a discretized Anosov flow.
\end{theorem} 
The key is to show that such an $f$ is dynamically coherent. Then \cite[Theorem~6.1]{BFFP-prequel} implies that it is a discretized Anosov flow.

Until this point we have always assumed that the bundles $E^s$, $E^c$, and $E^u$ have orientations that are preserved by $f$ so that we can use the result of Burago-Ivanov to find center stable and center unstable branching foliations. In \S\ref{dcwithout}, we show that if a lift of an iterate of $f$ is dynamically coherent and has a good lift $\widetilde{g}$ with doubly invariant behavior, then $f$ is dynamically coherent. This is why Theorems \ref{thmintro:Seifert} and \ref{thmintro:Hyperbolic}(i) do not need the orientability conditions.

\subsubsection{Seifert fibered and hyperbolic manifolds}

We rule out mixed behavior in Seifert fibered manifolds in \S\ref{sec-ThmA}, and in hyperbolic manifolds in \S\ref{sec-transl}--\ref{sec-mix.hyp}. Together with Theorem~\ref{thm.DIimpliesDC}, this yields the following:
\begin{theorem}\label{teo-hypseif}
	Let $f \colon M \to M$ be a partially hyperbolic diffeomorphism homotopic to the identity on a closed hyperbolic or Seifert fibered $3$-manifold. Assume that there are center stable and center unstable branching foliations. Then each good lift of $f$ either
	\begin{enumerate}[label=\rm{(\roman*)}]
		\item\label{item.hypsef_DI} fixes every leaf of both $\wfbs$ and $\wfbu$, or
		\item\label{item.hypsef_DT} acts as a translation on both leaf spaces.
	\end{enumerate}
	
	If there is a good lift of type \ref{item.hypsef_DI}, then $f$ is a discretized Anosov flow.
\end{theorem}

As was already pointed out in \cite[Remark 7.3]{BFFP-prequel}, there are examples in Seifert fibered manifolds where every good lift acts as a double translation. However, we show in \S\ref{sec-ThmA} that one can always find a finite power of such diffeomorphisms with a good lift that has doubly invariant behavior. Together with the results of \S\ref{sec-doubleinvndc} this implies Theorem~\ref{thmintro:Seifert}. 

Since every diffeomorphism of a hyperbolic 3-manifold has an iterate homotopic to the identity one also deduces Theorem~\ref{thmintro:Hyperbolic}. 

\begin{remark}
	An analogue of Theorem~\ref{teo-hypseif} holds under the assumption of $f$-minimality together with \emph{absolute} partial hyperbolicity. See ~\S\ref{ss.AbsolutePH}.
	
	We believe that Theorem \ref{teo-hypseif} should hold, using the same strategy as here, under the assumption of $f$-minimality together with the existence of an atoroidal piece in the JSJ decomposition of $M$. We have not pursued this here as it would require proving results similar to \cite{Thurston2,Calegari00,Fen2002} in this setting.
\end{remark}

\subsubsection{Double translations}\label{ss.DT}
This leaves open one major question:
\begin{question}
	Is there a partially hyperbolic diffeomorphism on a closed hyperbolic $3$-manifold whose good lifts act as double translations?
\end{question}

As noted above, there are such examples on Seifert fibered manifolds, but by Theorem~\ref{thmintro:Seifert} these are all dynamically coherent and have iterates that are discretized Anosov flows.

The dynamics of a double translation on a hyperbolic manifold would have to be coarsely comparable to that of a pseudo-Anosov flow (see \S\ref{sec-transl}). The closest analogues from this perspective are the non dynamically coherent examples on Seifert manifolds, constructed in \cite{BGHP}, which act as pseudo-Anosov maps on the base.

\subsection{Remarks and references}
There are three major areas in which the general case differs significantly from the dynamically coherent case:
\begin{enumerate}
	\item Unlike the dynamically coherent case (see condition ($\star\star$) in \cite[\S 2]{BFFP-prequel}), there may be annular center stable leaves which do not contain a closed center leaf.  
	
	\item In hyperbolic manifolds, we cannot rule out the possibility of double translations from the general version of the existence of cores that ``shadow'' the periodic orbits of the transverse pseudo-Anosov flow (see condition ($\star \star \star$) in \cite[\S 2]{BFFP-prequel}). 
	
	\item In hyperbolic and Seifert manifolds, it is more difficult to eliminate the hypothesis of $f$-minimality. See Section \ref{sss.minimal}. 
\end{enumerate}

We refer to \cite{CHHU,HP-survey,PotrieICM} for surveys on the problem of classification of partially hyperbolic diffeomorphisms in dimension 3. There is earlier work towards classification that does not assume dynamical coherence, but these articles tend to have two simplifying characteristics: They work with manifolds on which taut foliations are well understood and amenable to classification, and on which known partially hyperbolic models are available for comparison. Typically, dynamical coherence is established under the assumption of non-existence of invariant tori by using the fact that coarse dynamics separates leaves of the branching foliations. Neither of these features hold for the classes of manifolds considered in this article, and dynamical incoherence may appear in several different ways.

For instance, we obtain dynamical coherence in Section \ref{sec-doubleinvndc} when the lift of the partially hyperbolic diffeomorphism fixes each leaf of the lifted branching foliations. We also learn more about the structure of the branching foliations in the non dynamically coherent case, leading, in particular, to case \ref{it.DT_in_hyperbolic} of Theorem \ref{thmintro:Hyperbolic}. This structure also allows us to better understand the dynamical properties of the system, even when the manifold is not hyperbolic or Seifert fibered, as can be seen in Theorem \ref{thm-nocontractible}.

More generally, the framework that we develop for the study of non dynamically coherent partially hyperbolic diffeomorphism is useful outside of the homotopy class of the identity.

Below are several tools developed in this article that we wish to emphasize:

\begin{enumerate}
	
	\item In \S\ref{sec-branching} and \ref{sec-branching_goodlift}, we develop some of the basic theory necessary for the topological study of branching foliations and the diffeomorphisms that preserve them, including the structure of their leaf spaces.
	
	\item\label{itemintro1} In \S \ref{ss.fixed_center_or_coarse_contraction} we introduce the notion of coarsely contracting and coarsely repelling periodic rays. This should be useful for the study of all partially hyperbolic diffeomorphisms in 3-manifolds, i.e., including those not homotopic to the identity,  
	
	\item In \S\ref{sss.minimal} we study the way that certain special lifts of a partially hyperbolic diffeomorphism act within a fixed center stable leaf, and find conditions that guarantee the non-existence of fixed points. This involves understanding the behavior of strong stable manifolds through fixed points under iteration, which may find applications in other contexts.
	
	\item In \S\ref{sec-doubleinvndc} we prove uniqueness of (branching) foliations under certain conditions. This is a key to finding results that do not require taking finite lifts and finite powers. As such, it may also be relevant for the study of topological obstructions for partially hyperbolic diffeomorphisms --  note that the topological obstructions for the existence of Anosov flows can depend on taking finite lifts (see, e.g., \cite{Calegari:book}).
	
	There is other work that shows the uniqueness of branching foliations, but always in a setting where there is an understood model partially hyperbolic diffeomorphism for comparison.
	
	\item\label{itemintro4} In \S\ref{sec-transl}, \ref{sec-mix.hyp} we develop some tools to analyze the transverse geometry of branching foliations. This combines ideas from the theory of Lefschetz index, hyperbolic geometry, and the notion of coarsely expanding and contracting rays in item (\ref{itemintro1}). 
\end{enumerate}

The tools in (\ref{itemintro4}) are used in \cite{BFFP_companion} to prove that a large class of partially hyperbolic diffeomorphisms in Seifert manifolds
are dynamically incoherent. In addition (\ref{itemintro1}) and (\ref{itemintro4}) are used in \cite{FenleyPotrie} to obtain fine dynamical consequences of partial hyperbolicity in 3-manifolds.

\section{Branching foliations and leaf spaces}\label{sec-branching}


In this section we review the existence of center stable and center unstable branching foliations, and construct corresponding leaf spaces that capture their transverse topology. We will also construct a ``center foliation'' and leaf space.

\begin{definition} \label{def.branching}
 A \emph{branching foliation} of a $3$-manifold $M$ is a collection $\cF$ of $C^1$-immersed surfaces, called \emph{leaves}, each complete in its induced metric, such that:
 \begin{enumerate}[label=(\roman*)]
  \item Each $x \in M$ is contained in at least one leaf;
  \item \label{item.no_self_crossing} No leaf crosses itself;
  \item \label{item.no_crossing_of_leaves} Different leaves do not cross each other;
  \item \label{item.convergence_of_leaves} If $L_n$ are leaves, and $x_n \in L_n$ converges to a point $x \in M$, then some subsequence of the $L_n$ converges to a leaf $L$ with $x \in L$. \footnote{Here, convergence should be understood in the pointed compact-open topology, i.e., given a compact set $K$ in $L$ containing $x$, there is a sequence of compact subsets $K_n$ of $L_n$ containing $x_n$ such that $K_n$ converges to $K$ in the Hausdorff topology.}
 \end{enumerate}
\end{definition}

Here, ``crossing'' is meant in a topological sense -- see \cite{BI} or \cite{HP-survey}.

\begin{remark}
	In this context, ``branching'' refers to the fact that leaves may merge. This should not be confused with the typical use of ``branching'' in the theory of codimension-$1$ foliations, where it refers to non-Hausdorff behavior in the leaf space.
\end{remark}

Since a branching foliation has $C^1$ leaves that do not cross, it has a well-defined tangent distribution.

As with foliations, there is a sense in which branching foliations are ``locally product (branched) foliated'': around each point one can find a neighborhood $U$ with a smooth product structure $U \simeq \mathbb{D}^2 \times [0,1]$ such that each leaf of $\cF$ that intersects $U$ does so in a collection of discs that are transverse to the $[0,1]$-fibration and meet every $[0,1]$-fiber.
This follows readily from the fact that branching foliations are tangent to $C^1$ distributions.

On a compact manifold there is a uniform scale $\epsilon_0$, called  the \emph{local product structure size}, such that every open set of diameter less than $\epsilon_0$ is contained in a product chart as above.

\begin{definition}\label{def.wellapproximated}
	A branching foliation $\cF$ is \emph{well-approximated by foliations} if there is, for a set of $\epsilon > 0$ accumulating on $0$, a family of foliations $\{\cF_{\eps}\}$ with $C^1$ leaves, and a family of continuous maps $\{ h_{\eps}\colon M \to M \}$, that have the following properties (with respect to some fixed Riemannian metric):
	\begin{enumerate}[label=(\roman*)]
		\setcounter{enumi}{4}
		\item The angles between leaves of $\cF$ and $\cF_{\eps}$ are less than $\eps$;
		\item \label{item.epsilon_close_identity} The $C^0$-distance between $h_{\eps}$ and the identity is less than $\eps$;
		\item \label{item.leaves_to_leaves} On each leaf of $\cF_{\eps}$, the map $h_{\eps}$ restrict to a local diffeomorphism to a leaf of $\cF$;
		\item \label{item.map_surjective} For each leaf $L$ of $\cF$ there is a leaf $L_{\eps}$ of $\cF_{\eps}$ with $h_{\eps}(L_{\eps}) = L$.
	\end{enumerate}
\end{definition}

\begin{remark}
	Note that while the maps $h_{\eps}$ restrict to local diffeomorphisms on leaves, they will fail to be global diffeomorphisms on leaves of $\cF_{\eps}$ that map to self-merging leaves of $\cF$. In addition, the $h_{\eps}$ will not be local diffeomorphisms on $M$ unless $\cF$ is actually a true foliation.
\end{remark}

\begin{definition}\label{def:oriented}
	A partially hyperbolic diffeomorphism $f: M \to M$ is said to be \emph{orientable} if the bundles $E^s$, $E^u$ and $E^c$ admit orientations that are preserved by $f$.
\end{definition}

The following is the foundational existence result of Burago-Ivanov:
\begin{theorem}[Burago-Ivanov \cite{BI}]\label{teo-BI} 
Let $f$ be an orientable partially hyperbolic diffeomorphism of a $3$-manifold $M$. Then there are $f$-invariant branching foliations $\fbs$ and $\fbu$ tangent to $E^c \oplus E^s$ and $E^c \oplus E^u$ that are well-approximated by foliations. 
\end{theorem}

Here, a branching foliation is said to be \emph{$f$-invariant} if the image of any leaf under $f$ is again a leaf.

Note that there is no \emph{a priori} uniqueness for the center stable and center unstable branching foliations $\fbs$ and $\fbu$ related to a partially hyperbolic diffeomorphism. Nevertheless, we will typically fix \emph{some} pair of such branching foliations and call them ``the'' branching foliations for our diffeomorphism. In addition, we will fix families of approximating foliations $\fes$ and $\feu$, with associated maps denoted by $h^{cs}_{\eps}$ and $h^{cu}_{\eps}$. 

On the other hand, since the stable bundle $E^s$ \emph{is} uniquely integrable, a stable leaf $s$ that intersects a center stable leaf $L$ must be contained entirely in $L$. Consequently, the intersection of any two center stable leaves is saturated by stable leaves.

Once we have fixed ``the'' center stable and center unstable branching foliations $\fbs$ and $\fbu$, the corresponding lifted foliations on $\mt$ will be denoted by $\wfbs$ and $\wfbu$. We may then define center leaves as follows:

\begin{definition}\label{d.center_leaf}
A \emph{center leaf} of a partially hyperbolic diffeomorphism is the projection to $M$ of a connected component of the intersection between a leaf of $\wfbs$ and a leaf of $\wfbu$.
\end{definition}

Although the collection of center leaves is not a foliation, it is a kind of codimension-$2$ branching foliation. We will abuse terminology and call the collection of center leaves the \emph{center foliation}.

\begin{remark}\label{r.center_leaf}
Each center leaf is tangent to the central direction $E^c$, but a complete curve that is tangent to the central direction may not be a center leaf. Indeed, even when the diffeomorphism is dynamically coherent, the central direction may not be uniquely integrable. See \cite{HHU-noncoherent} for an example.
\end{remark}

\begin{figure}[h]
\centering
\begin{subfigure}[t]{0.45\textwidth}
\centering
\includegraphics{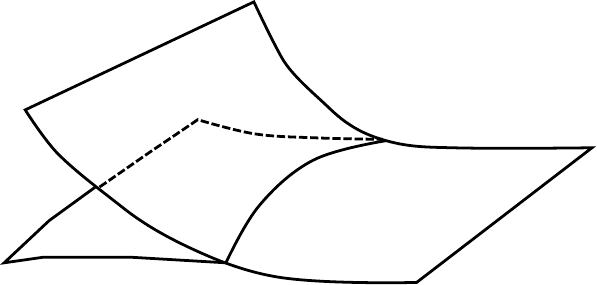} 
\subcaption{Two center stable leaves sharing a region}
\end{subfigure}%
\hspace{.05\textwidth}
\begin{subfigure}[t]{0.45\textwidth}
 \centering
 \includegraphics{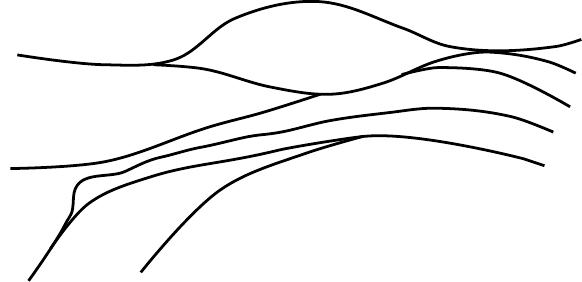} 
\subcaption{Distinct center leaves inside a center stable leaf}
\end{subfigure}
\caption{The branching of center and center-stable leaves.}\label{f.1M}
\end{figure}

\subsection{Tautness} 
In this article, the approximating foliations $\fes$ and $\feu$ have no compact leaves. 

Indeed, suppose that one has a compact leaf $L \in \fes$. Then $K := h^{cs}_{\eps}(L)$ is a compact leaf of $\fbs$. Since the stable bundle $E^s$ is uniquely integrable, this compact surface has a foliation without compact leaves, so it is a torus. According to \cite[Theorem~1.4]{HHU-nocompactleaves}, there are only a few classes of manifolds that admit partially hyperbolic diffeomorphisms with tori tangent to $E^s \oplus E^c$, all mapping tori of $\mathbb{T}^2$.

Since we assume that $\pi_1(M)$ is not virtually solvable, it follows that the approximating foliations have no compact leaves, which implies that they are taut.

\subsection{Center stable and center unstable leaf spaces}\label{ss.leafspace}

Given a foliation $\fol$ on a manifold $M$, the set of leaves of the lifted foliation $\fn$ on $\mt$ has a natural topology -- the quotient obtained from $\mt$ by collapsing each leaf to a point -- and the resulting space is called the \emph{leaf space} of $M$. 

In this section we will define a notion of leaf space for our branching foliations, where it would not make sense to take the quotient topology. We will see, in fact, that the leaf spaces of our branching foliations are homeomorphic to those of the approximating foliations for small enough $\eps$.

Much of this section would apply to any codimension-$1$ branching foliation, of any dimension, as long as the leaves in the universal cover are properly embedded $\mathbb{R}^{n-1}$'s in $\mathbb{R}^n$. For convenience, however, we will mostly restrict attention to the branching foliations that we are interested in. This allows for some shortcuts. For example, in Proposition \ref{prop.leafspace} we use the approximating foliations and maps to see that the leaf space is a $1$-manifold as desired, though this could also be done directly.

\subsubsection{Complementary regions and sides}\label{ss.sides}
Since $M$ is not finitely covered by $S^2 \times S^1$ (as $\pi_1(M)$ is not virtually solvable), and our branching foliations are well-approximated by taut foliations, it follows that the universal cover is homeomorphic to $\R^3$, and the lifted leaves are properly embedded planes \cite{CandelConlonI}.

The \emph{complementary regions} of a leaf $L$ are the two connected components of $\mt \smallsetminus L$. For each complementary region $U$ of a leaf $L$, the closure $\overline{U} = U \cup L$ is called a \emph{side} of $L$.

A coorientation of the branching foliation (which may be thought of as a coorientation of its tangent distribution) determines, for each leaf $L$, a \emph{positive} and a \emph{negative} complementary region which we denote by $L^{\oplus}$ and $L^{\ominus}$. The corresponding sides are denoted by $L^+ = L^{\oplus} \cup L$ and $L^-  = L^{\ominus} \cup L$. We will fix such a coorientation throughout.

\subsubsection{Leaf spaces.}\label{sss.ls-centerstable}
Let us now construct the \emph{center stable leaf space} $\lcsb$. This is the set of leaves of $\wfbs$ with the topology defined below. The \emph{center unstable leaf space} $\lcub$ is constructed similarly.

In the case of a true codimension-$1$ foliation, each transverse arc in the universal cover maps homeomorphically to an arc in the leaf space. We will use a similar idea for branching foliations, and use transverse arcs to construct the topology. In a true foliation each point in a transverse arc intersects a single leaf; for our branching foliations we need to ``blow up'' at some points using the following definition:

\begin{definition}
	Given $x \in \mt$, let $\lcsb(x) \subset \lcsb$ denote the set of leaves that contain $x$.
	
	Given distinct leaves $L \neq E$ in $\lcsb(x)$, we will write $L <_x E$ whenever $L^+ \supset E$.
\end{definition}

\begin{claim}\label{claim.closed_intervals_merge}
	For each $x \in \mt$, $<_x$ defines a linear order, with which $\lcsb(x)$ is order-isomorphic to a closed interval (possibly a single point). 
\end{claim}
\begin{proof}
	Assume that $\lcsb(x)$ is not a singleton.
	
	That $<_x$ defines a linear order on $\lcsb(x)$ follows from the fact that leaves do not cross (property \ref{item.no_crossing_of_leaves} of Definition \ref{def.branching}). From property \ref{item.convergence_of_leaves}, it follows that this order is complete.
	
	To see that $\lcsb(x)$ is order-isomorphic to a closed interval, it suffices to check that there are no gaps in the order. That is, given $L, E \in \lcsb(x)$ such that $L <_x E$, we must find some $L' \in \lcsb(x)$ with $L <_x L' <_x E$.
	
	Given such $L, E$, let $y$ be a boundary point of the connected component of $L \cap E$ that contains $x$. Consider a neighborhood $B$ of $y$ with diameter less than $\eps_0$, the local product structure size of $\fbs$. Since $\wfbs$ is product branched foliated in $B$, each leaf that intersects $B \cap (L^+ \cap E^-)$ must intersect $y$, and since leaves do not cross, any such leaf must intersect $x$. Any such leaf $L'$ will have $L <_x L' <_x E$.
\end{proof}

Combined with the linear ordering of points in a transversal, this gives a linear ordering on the set of leaves that intersect a transversal:

\begin{definition}
	Given a transverse arc $\tau$, let $\lcsb(\tau) \subset \lcsb$ denote the set of leaves that intersect $\tau$.
	
	Orient $\tau$ so that it agrees with the coorientation on $\wfbs$. Given distinct leaves $K \neq L$ in $\lcsb(\tau)$, we will write $K <_\tau L$ whenever either
	\begin{itemize}
		\item $K \cap \tau$ lies forward of $L \cap \tau$ with respect to the orientation on $\tau$, or
		\item $K$ and $L$ intersect $\tau$ at the same point $x$ and $K <_x L$.
	\end{itemize} 
\end{definition}

The following properties of these orderings may be found in \cite[\S7]{BI}.

\begin{claim}\label{claim.ordersFromArcs}
	\begin{enumerate}
		\item For each \emph{open} transverse arc $\tau$, $<_\tau$ is a linear order, with respect to which $\lcsb(\tau)$ is order-isomorphic to an open interval.
		
		\item \label{item.differentarcs} $\sigma$ and $\tau$ are open transverse arcs, then $<_\sigma$ and $<_\tau$ define the same linear order on $\lcsb(\sigma) \cap \lcsb(\tau)$, which is order-isomorphic to an open interval (possibly empty). 
	\end{enumerate}
\end{claim}

\begin{definition}[topology of $\lcsb$]
	The \emph{center stable leaf space} is $\lcsb$, with the topology 
$\cT$ generated by all open intervals in $\lcsb(\tau) \subset \lcsb$, over all open transverse arcs $\tau$.
\end{definition}

From Claim~\ref{claim.ordersFromArcs}(\ref{item.differentarcs}), it suffices to take any collection of open transverse arcs that intersect every leaf of $\wfbs$. Since $M$ is compact, one can take a finite collection of open transverse arcs in $M$ and consider all of their lifts to $\mt$. This implies in particular that $\lcsb$ is second countable.

\begin{proposition}
	The center stable leaf space $\lcsb$ is a simply-connected, possibly non-Hausdorff $1$-manifold.
\end{proposition}

The same applies to $\lcub$. This is not difficult to prove directly, and it applies more generally to any codimension-$1$ branching foliation of a closed $n$-manifold, as long as the lifted foliation is by properly embedded $\mathbb{R}^{n-1}$'s in $\mt \simeq \mathbb{R}^n$. In the present case, it follows as well from Proposition~\ref{prop.leafspace} below.

\subsubsection{Leaf spaces and approximating foliations}
Let $\lcse$ and $\lcue$ denote the leaf spaces of the approximating foliations $\fes$ and $\feu$. The maps $h^{cs}_\epsilon$ and $h^{cu}_\epsilon$ induce functions
\[
g_{\epsilon,s}\colon \lcse \rightarrow \lcsb \text{ and } g_{\epsilon,u}\colon \lcue \rightarrow \lcub.
\]
between the corresponding leaf spaces, which are surjective whenever $\eps$ is sufficiently small (cf. Definition~\ref{def.wellapproximated}).

Since $\fes$ is a true foliation, its leaf space $\lcse$ is a simply-connected, possibly non-Hausdorff $1$-manifold (cf. \cite[Appendix B]{BFFP-prequel}).

\begin{remark}
It is possible to modify the proof of \cite[Theorem 7.2]{BI}, where the foliations $\fes$ and maps $h^{cs}_\epsilon$ are constructed, so that the $g_{\epsilon,s}$ are injective in addition to surjective. With this in hand, one could define the topology on $\lcsb$ to be the one induced by this bijection.

Instead of redoing the entire proof of \cite[Theorem 7.2]{BI}, we will use a simpler fact that can easily be extracted from that proof: The maps $h^{cs}_{\epsilon}$ are ``monotone'' in the sense that they preserves the natural linear order on plaques in local charts. 
\end{remark}

\begin{proposition}\label{prop.leafspace}
When $\eps$ is sufficiently small,
\begin{enumerate}
	\item the preimage of each point in $\lcsb$ under $g_{\epsilon,s}$ is a closed interval,
	
	\item $g_{\epsilon,s}\colon \lcse \rightarrow \lcsb$ is continuous, and
	
	\item the topology $\cT$ on $\lcsb$ is equivalent to the quotient topology $\cT_{\eps}$ induced by $g_{\epsilon,s}$.
\end{enumerate}

The same applies for the center unstable foliations.
\end{proposition}
\begin{proof}
Let $\eps_0$ be the local product sizes of $\fbs$, and let $\eps < \eps_0/2$. 
Let $\cT_{\eps}$ be the quotient topology induced
by $g_{\eps,s}$ on $\lcsb$.

\begin{enumerate}
	\item Let $I \subset \lcse$ be the preimage of a leaf $L \in \lcsb$, and suppose that $I$ contains two leaves $\hat L_1$ and $\hat L_2$. We want to show that $\widetilde h^{cs}_{\epsilon}$ takes every leaf between $\hat L_1$ and $\hat L_2$ to $L$.
	
	From property \ref{item.epsilon_close_identity} of Definition 
\ref{def.wellapproximated}, the Hausdorff distance between $\hat L_1$ and $\hat L_2$ is less than $2\eps$. Since $2\eps$ was chosen to be less than the local product structure size, it follows that the region between $\hat L_1$ and $\hat L_2$ has leaf space which is a closed interval. By the local monotonicity of $\widetilde h^{cs}_{\epsilon}$, it follows that $g_{\epsilon,s}$ maps the entire region between $\hat L_1$ and $\hat L_2$ to $L$. This implies that the preimage of $L$ is an interval, which is closed because $\widetilde h_\epsilon^{cs}$ is continuous.
	
	\item Let $U \subset \lcsb$ be open. Around each point in $U$ one can find an open interval $J \subset U$ that is the set of leaves intersecting a small open transversal $\beta$. We want to show that $g_{\eps,s}^{-1}(J)$ is open in $\lcse$.
	
	Let $\hat L_1$ be a leaf in $g_{\eps,s}^{-1}(J)$. Then $\hat L_1$ intersects $\beta$ (or a slightly bigger transversal), so all the leaves of $\wfes$ close enough to $\hat L_1\cap \beta$ intersect $\beta$. Thus an open neighborhood of $\hat L_1$ is contained in $g_{\eps,s}^{-1}(J)$, and $g_{\eps,s}$ is continuous.
	
	\item From (2) it follows that $\cT \subset \cT_{\eps}$. Let us prove the other inclusion.
	
	Suppose $W \in \lcsb$ is an open set in $\cT_{\eps}$, and let $y \in W$. Then $U = (g_{\eps,s})^{-1}(W)$ is an open set containing the closed interval $I = (g_{\eps,s})^{-1}(y)$. Let $L$ and $E$ be the boundary leaves of $I$. Then one can find half-open intervals $I_L, I_E \subset U$ such that $I_L \cap I = L$ and $I_E \cap I = E$. Then $I_L \cup I \cup I_E$ projects to a set in $\lcsb$ which contains an open interval around $y$ in $\lcsb$. Since this applies for every $y \in W$ it follows that $W$ is open in $\cT$.
\end{enumerate}
\end{proof}

This suffices to show that $\lcsb$ is a $1$-manifold. It is possible to modify $g_{\epsilon,s}\colon \lcse \rightarrow \lcsb$ to be a homeomorphism when $\epsilon$ is sufficiently small, but we will not need this fact.

In the sequel, we fix $\epsilon$ small enough so that the previous proposition applies for both the center stable and center unstable foliations.

\subsection{Center ``foliations''}

\subsubsection{The center foliation within a center stable/unstable leaf}\label{sss.ls-center.in.centerstable}

Fix a center stable leaf $L$ of $\wfbs$. We will describe the topology of the center leaf space, $\cL^c_L$, restricted to $L$. The center leaf within a center unstable leaf is defined in the same manner.

\begin{remark}\label{r.center}
Recall from Definition \ref{d.center_leaf} that a center leaf in $\mt$ is defined as a connected component of the intersection between a leaf of $\wfbs$ and a leaf of $\wfbu$. 
Now, the following situation may arise (see Figure \ref{f.2M}): Two leaves $U_1, U_2$ of $\wfbu$ and a leaf $L$ of $\wfbs$ such that the triple intersection $U_1\cap L \cap U_2$ contains a connected component of $c_1$ of $U_1\cap L$ as well as a connected component $c_2$ of $U_2\cap L$. That is, the center leaves $c_1$ and $c_2$ represents the same set in $\mt$. In this case, we also consider $c_1$ and $c_2$ as the \emph{same} leaf of the center foliation in $L$.
\end{remark}

\begin{figure}[ht]
\begin{center}
\includegraphics[scale=0.64]{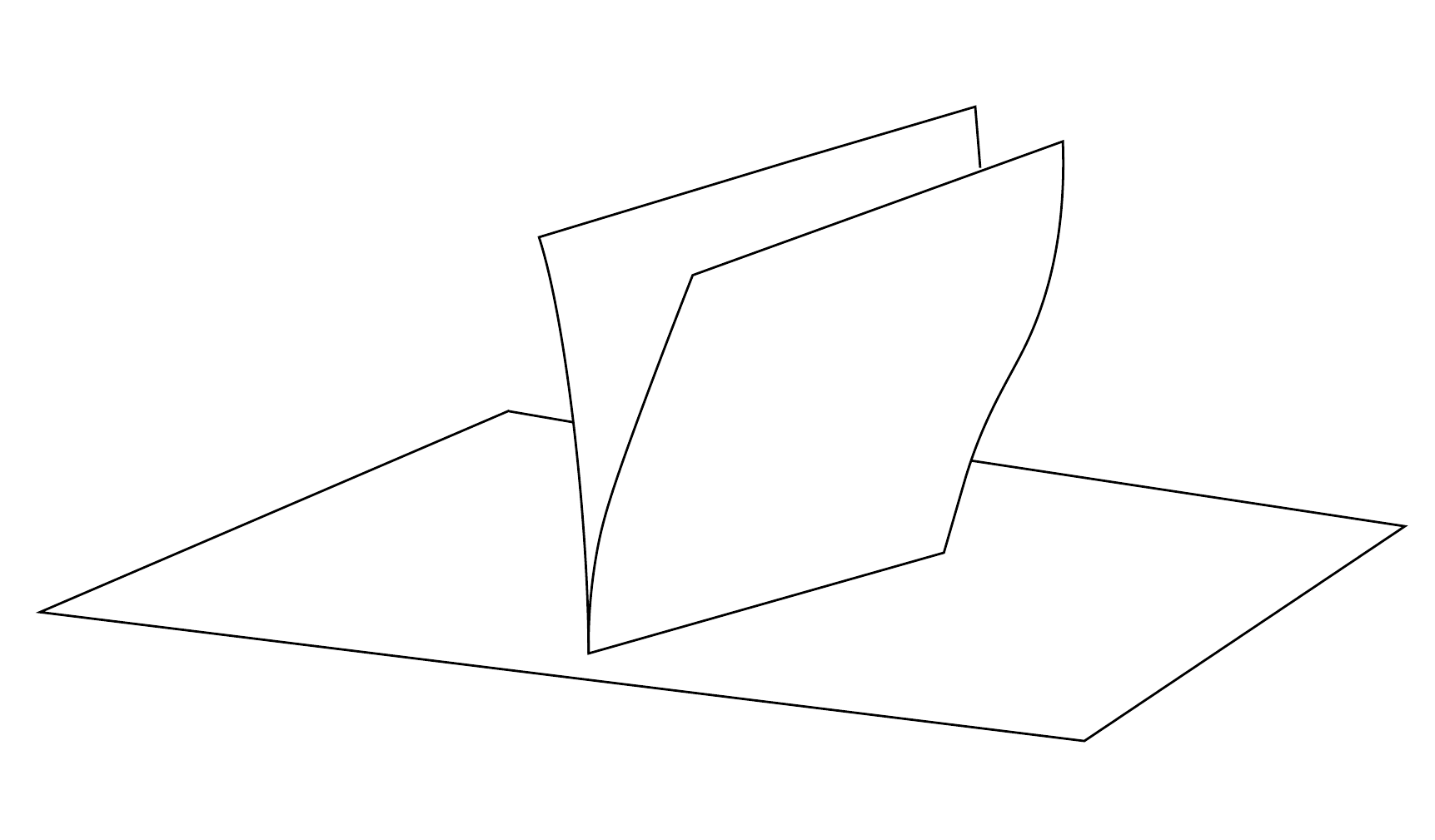}
\begin{picture}(0,0)
\put(-50,30){$L \in \wfbs$}
\put(-80,135){$U_2 \in \wfbu$}
\put(-212,150){$U_1 \in \wfbu$}
\put(-145,48){$c_1=c_2$}
\end{picture}
\end{center}
\vspace{-0.5cm}
\caption{Different center unstable leaves may intersect a given center stable leaf in the same center leaf.}\label{f.2M}
\end{figure}
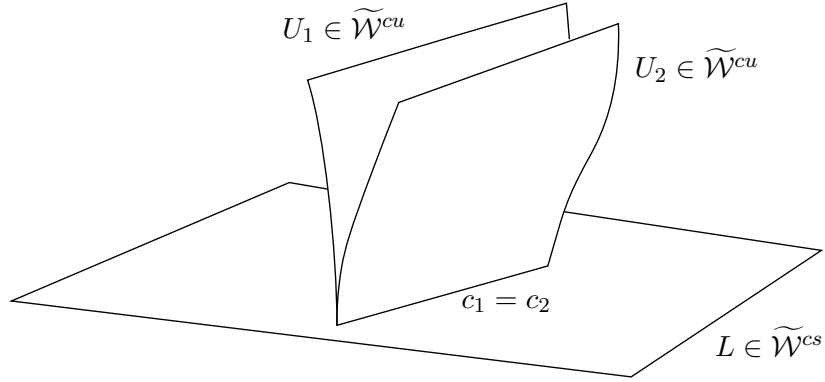

\begin{definition}[topology ${\mathcal A}$ in $\cL^c_L$]\label{def.center_leaf_space_in_L}
Consider a countable set of open transversals $\tau_i$ 
which are perpendicular to the center bundle in $L$, and
so that the union intersects every center leaf in $L$.
Put the order topology in the set $I_i$ of center leaves
intersecting $\tau_i$. This induces the topology ${\mathcal A}$ 
in $\cL^c_L$.
\end{definition}

Let $L$ be a fixed leaf of $\wfbs$. We again fix an $\epsilon>0$ and consider the approximating foliation $\wfeu$. Since $\wfbu$ is transverse to $L$, so is $\wfeu$ (for $\epsilon$ small enough). Thus, $\wfeu$ induces a $1$-dimensional (non branching) foliation $\fol_{\epsilon}$ on $L$, and hence its leaf space $\cL^c_{L, \epsilon}$ is a $1$-dimensional, not necessarily Hausdorff, simply
connected manifold. 

The behavior described in Remark \ref{r.center} above leads to the following issue: the unique center leaf $c_1=c_2$ is approximated by two distinct leaves of $\fol_{\epsilon}$. Thus, the leaf space, $\cL^c_L$, of the center foliation on $L$ is not in bijection with $\cL^c_{\epsilon}$. However, we still have a surjective, but not necessarily injective, projection $\mathrm{pr_{\eps}}\colon \cL^c_{L, \epsilon} \to \cL^c_{L}$ as in the previous 
subsection.
Let ${\mathcal A}_{\eps}$ be the quotient topology from the map 
$\mathrm{pr_{\eps}}$.

Just as in Proposition \ref{prop.leafspace} one can prove the
following:

\begin{lemma}\label{l.center_leaf_space_in_L}
The set of center leaves in $L$ through a point $x$ is a closed
interval.
 Let $c_0$ be a center leaf in $L$. Let $I = \mathrm{pr}^{-1}(c_0) \subset \cL^c_{\epsilon}$. The set $I$ is a closed interval.
If $\eps < \eps_0$, then the topologies ${\mathcal A}$ and
${\mathcal A}_{\eps}$ are the same.
\end{lemma}

\subsubsection{Center foliation in $\mt$}\label{sss.ls-center}

Finally, we have to put a topology on the leaf space $\cL^c$ of the center foliation in $\mt$. 

Pick an $0<\eps < \eps_0$ 
so that $\wfes$ and $\wfeu$ are transverse to each other. Call $\fol_{\epsilon}$ the $1$-dimensional foliation obtained as the intersection of $\wfes$ and $\wfeu$. The leaf space $\cL^c_{\epsilon}$ of $\fol_{\epsilon}$ is now a simply connected, possibly non Hausdorff, $2$-dimensional manifold. But as before, there is only a surjective, and not injective, projection $g_{\epsilon}\colon \cL^c_{\epsilon} \to \cL^c$.

The map $g_{\epsilon}$ is defined in the following way: If $\bar c$ is a leaf of $\fol_{\epsilon}$, then it is the intersection of a leaf $\bar U$ of $\wfeu$ and a leaf $\bar S$ of $\wfes$. Then, there exists a unique connected component $c$ of $g_{\epsilon,u}(\bar U)\cap g_{\epsilon,s}(\bar S)$ that is at distance less than $2\epsilon$ from $\bar c$. We define $g_{\epsilon}(\bar c) = c$.

Once again, the topology 
${\mathcal B}_{\eps}$
we put on $\cL^c$ is obtained by identifying elements of $\cL^c_{\epsilon}$ that project to the same element of $\cL^c$ and taking the quotient topology.

As done is the previous two subsections \ref{sss.ls-centerstable} and
\ref{sss.ls-center.in.centerstable}, in order to prove that the topology that we put on $\cL^c$ makes it a simply connected (not necessarily Hausdorff) $2$-manifold, it is enough to show that the preimages of points by $g_{\epsilon}$ are closed, simply connected sets contained in a local chart of $\cL^c_{\epsilon}$. In order to do that, first notice that $\cL^c_{\epsilon}$ is locally homeomorphic to $\lcse \times \lcue$. Indeed, any $\bar c_0 \in \cL^c_{\epsilon}$ is a connected component of $\bar U_0 \cap \bar S_0$, with $\bar U_0\in \lcue$ and $\bar S_0\in \lcse$. Now, if $V_u$ is a small enough open interval in $\lcue$ and $V_s$ is a small enough open interval in $\lcse$, then for any $\bar U \in V_u$ and $\bar S\in V_s$, the intersection $\bar U \cap \bar S$ contains a unique connected component close to $c_0$.
Using this local homeomorphism, the following lemma will imply that the topology $\cL^c$ is as we claimed.

\begin{lemma}\label{rectangle}
 Let $c_0$ be in $\cL^c$. The set $R = g_{\epsilon}^{-1}(c_0)$ is homeomorphic to a closed rectangle in $\lcse \times \lcue$.
\end{lemma}

\begin{proof}
 Let $\bar c_1,\bar c_2 \in R$. Let $\bar U_1$ be the leaf in $\lcue$ containing $\bar c_1$ and let $\bar S_2$ be the  the leaf in $\lcse$ containing $\bar c_2$. Let $U_1 = g_{\epsilon, u}(\bar U_1)$ and $S_2 = g_{\epsilon, s}(\bar S_2)$. Since $\bar c_1,\bar c_2 \in R$, the center leaf $c_0$ is a connected component of $U_1\cap S_2$. Thus $\bar U_1$ and $\bar S_2$ must intersect and the intersection
contains a unique connected component $\bar c_3$ at distance at most $2\epsilon$ from $c_0$.
 
 Now, the proof of Lemma \ref{l.center_leaf_space_in_L} shows that $\bar c_1$ and $\bar c_3$ are two ends of an interval in the leaf space of $\fol_{\epsilon}$ restricted to $\bar U_1$ that is entirely contained in $R$. Similarly, for $\bar c_2$ and $\bar c_3$ considered as elements of the leaf space of $\fol_{\epsilon}$ restricted to $\bar S_2$. In turns, the arguments of the proof of Lemma \ref{l.center_leaf_space_in_L} imply that the set $R$ projects to a closed interval in both $\lcse$ and $\lcue$, i.e., it is a closed rectangle in $\lcse \times \lcue$.
\end{proof}

Just as in the previous two sections we can also put a 
topology ${\mathcal B}$ on $\cL^c$ directly as follows:

\begin{definition}{(topology ${\mathcal B}$ on $\cL^c$)}
In $M$ pick a collection of very small open rectangles $R_i$ 
which are almost perpendicular to the center bundle,
and with boundary two arcs in a leaves  of $\lcsb$
and two arcs in leaves of $\lcub$. 
Consider all lifts $R$ of these to $\mt$. The set of 
center leaves intersecting $R$ is naturally bijective
to an open rectangle and put the topology making this
a local homeomorphism. The topology ${\mathcal B}$ is
generated by these rectangles.
\end{definition}

First we justify why the set of center leaves through $R$ is 
naturally an open rectangle. Let $L_1, L_2$ be the center
stable leaves containing the two arcs in the boundary of $R$,
and $U_1, U_2$ be the corresponding center unstable leaves.
The set of center stable leaves between $L_1, L_2$ (not
including $L_1, L_2$) is naturally
ordered isomorphic to an open interval. This was proved in
subsection \ref{sss.ls-centerstable}. The same for the center
unstable foliation. The product is an open rectangle. The set
of center leaves intersecting $R$ is a quotient of this.
The sets which are quotiented to a point  are compact subrectangles.
The proof is the same as the previous lemma. Hence the
quotient is naturally a rectangle. 
In addition if a collection of center leaves intersects
two such rectangles $R, R'$, then the identifications in 
$R$ also produce the same identifications in $R'$ and the
order of the center stable and center unstable foliations in the subsets are the same
whether in $R$ or $R'$. Hence in the identification, the
topologies agree. 

Just as in the previous sections one can prove:

\begin{lemma}{}
For $\epsilon < \eps_0$, the topologies ${\mathcal B}$ 
and ${\mathcal B}_{\eps}$ are the same.
\end{lemma}

The main property is to prove is exactly that of Lemma
\ref{rectangle}. The rest follows just as in the previous
subsections.

\subsection{From foliations to branching foliations}
Using the leaf space, one can carry over a number of concepts from foliations to branching foliations.

\subsubsection{Uniform and $\R$-covered branching foliations}
A branching foliation is said to be \emph{$\R$-covered} if its leaf space is homeomorphic to $\R$. It is \emph{uniform} if every two leaves in the universal cover are a finite Hausdorff distance apart.

By Proposition~\ref{prop.leafspace} a branching foliation is uniform or $R$-covered if and only if its approximating foliations are, for $\eps$ sufficiently small.

\subsubsection{Saturations and minimality}
A foliation that is preserved by a homeomorphism $f$ is said to be $f$-minimal if the only closed, saturated, $f$-invariant sets are the empty set and the whole manifold. We will define $f$-minimality identically for branching foliations, but we must be careful about what we mean by ``saturated'':

\begin{definition}\label{def.fsaturated}
	A set $C \subset M$ is \emph{$\fbs$-saturated} if, for every $x\in C$, there is a leaf of $\fbs$ that contains $x$ and is contained in $C$.
	
	A \emph{saturation} of a saturated set $C \subset M$ is a collection of leaves $X \subset \fbs$ whose union is $C$.
\end{definition}

Note that this is much weaker than asking that \emph{every} leaf intersecting $C$ is contained in $C$. In particular, our notion of saturation has the peculiar property that the complement of a saturated set need not be saturated (see Figure \ref{f.3M}).

In addition, a saturated set may have different saturations. However, a saturated set always has a unique \emph{maximal saturation}, consisting of all leaves that are contained in it.

\begin{definition}\label{def.fminimal}
	We say that $\fbs$ is \emph{$f$-minimal} if the only closed, $\fbs$-saturated, and $f$-invariant subsets of $M$ are $\emptyset$ and $M$.
\end{definition}
We emphasize that ``closed'' is meant as a subset of $M$, not $\lcsb$.
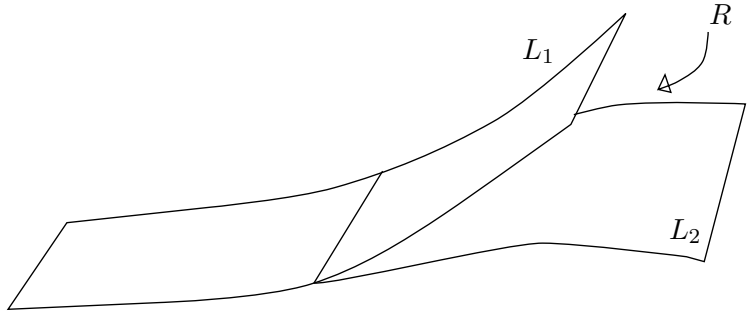
\begin{figure}[ht]
	\begin{center}
		\includegraphics[scale=0.64]{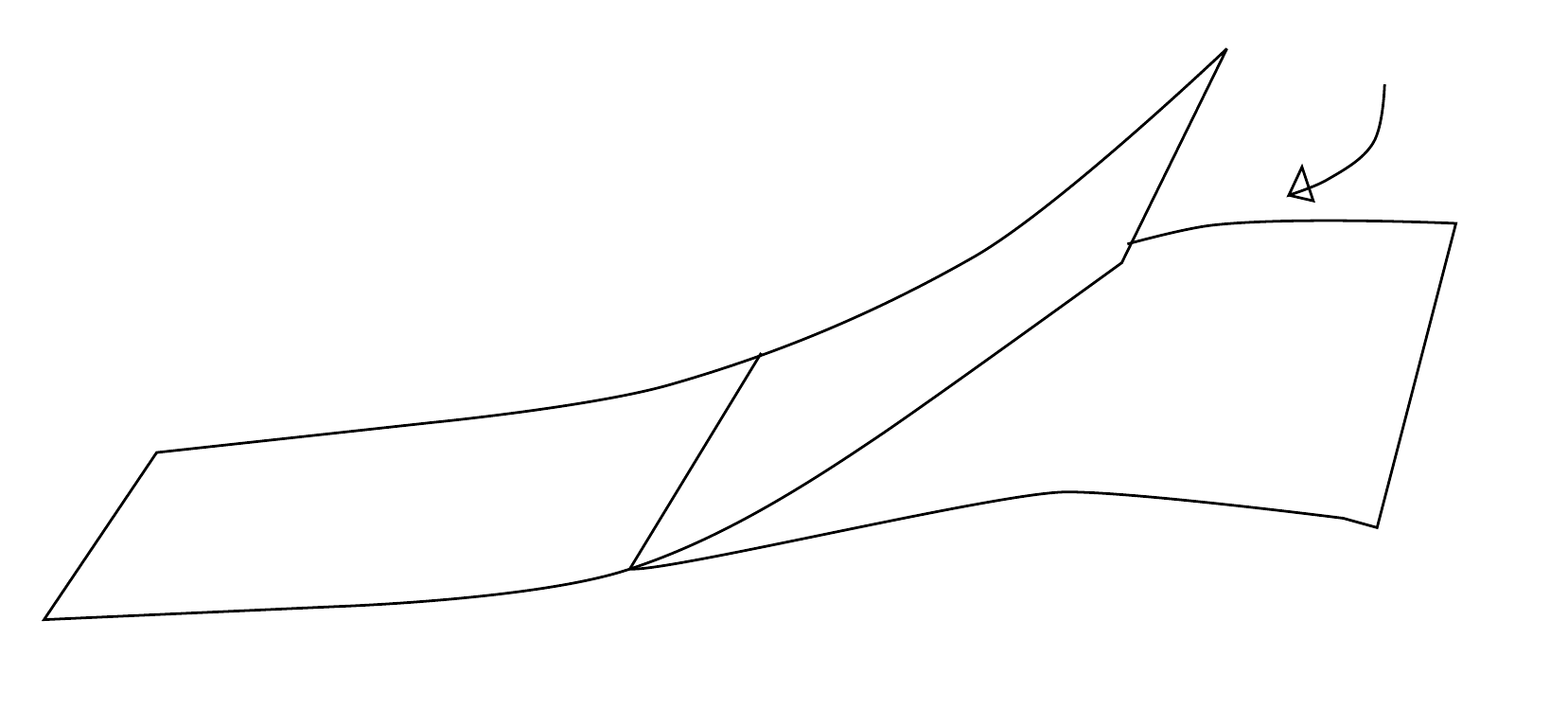}
		\begin{picture}(0,0)
		\put(-54,44){$L_2$}
		\put(-109,112){$L_1$}
		\put(-39,125){$R$}
		\end{picture}
	\end{center}
	\vspace{-0.5cm}
	\caption{$L_1$ and $L_2$ are two leaves in $C$, but the region $R$ is not in $C$. Then, in parts of $R$, all the center stable leaves intersect the branch locus between $L_1$ and $L_2$, so have parts in $C$ and parts not in $C$ (and therefore $M \setminus C$ is not saturated by center stable leaves).}\label{f.3M}
\end{figure}

Saturated sets and saturations are defined similarly in the universal cover. Here, a saturation can be naturally thought of as a subset of the leaf space $\lcsb$. However, the topology of a saturated set in $\mt$ does not necessarily agree with the topology of a saturation in $\lcsb$:
\begin{remark}\label{r.closedsetb}
	Let $C \subset \mt$ be $\wfbs$-saturated. It is possible for $C$ to be closed in $\mt$, but have a saturation $\mathcal{C} \subset \lcsb$ that is not closed in $\lcsb$. However, it is easy to see that $C$ is a closed in $\mt$ if and only if its \emph{maximal} saturation is closed in $\lcsb$.  
	
	It is true but less immediate that the only saturation of $\mt$ that is closed in $\lcsb$ is all of $\lcsb$ (Lemma \ref{closedsets}).
\end{remark}

\subsubsection{Perfect fits.}\label{ss.perfectfits} 
The notion of ``perfect fits'' from the theory of codimension-$1$ foliations (see \cite[\S4.1]{BFFP-prequel}) applies to branching foliations once it is modified appropriately. 

We will need the $2$-dimensional version of this concept, in \S\ref{sec-centerdynamics}, to understand the center and stable foliations within a center stable leaf. Given a center stable leaf $L$, let $\cC_L$ and $\cS_L$ be the center and stable foliations within $L$, and let $\cL^c_L$ and $\cL^s_L$ be the corresponding leaf spaces.

\begin{definition}\label{def.perfect_fits_branching}
	A leaf $c \in \cC_L$ and a leaf $s \in \cS_L$ make a $\cC\cS$-\emph{perfect fit} if they do not intersect, but there is a local transversal $\tau$ to $\cC_L$  through $c$ such that every leaf in $\cC_L(\tau)$ that lies sufficiently close to one side of $c$ (in the linear order $<_\tau$) intersects $s$. 
	
	They make a $\cS\cC$-\emph{perfect fit} if there is a local transversal $\tau'$ to $\cS_L$  through $s$ such that every leaf in $\cS_L(\tau)$ that lies sufficiently close to one side of $s$ intersects $c$.
	
	We say that $c$ and $s$ make a \emph{perfect fit} if they make both a $\cC\cS$- and $\cS\cC$-perfect fit.
\end{definition}

\begin{remark}
	Note that when defining  $\cC\cS$-perfect fits it is important to use the linear order $<_\tau$ on $\cC_L(\tau)$, defined in \S\ref{sss.ls-centerstable}, since there may be center leaves on the same side of $c$ as $s$ that merge with $c$.
	
	Since $\cS_L$ is a true foliation, the linear order $<_{\tau'}$ on $\cS_L(\tau')$ comes directly from the transversal $\tau'$, so the notion of a $\cS\cC$-perfect fit is exactly as in \cite[\S4.1]{BFFP-prequel}.
	
	One may equivalently define $\cC\cS$-perfect fits as follows: Given a stable leaf $s$ in $L$, let $I_s \subset \cL^c_L$ be the set of center leaves that intersect $s$. Then $c$ and $s$ makes a $\cC\cS$-perfect fit if and only if $c \in \partial I_s$.
\end{remark}

\begin{lemma}
	Let $c$ and $s$ be center and stable leaves in a center stable leaf $L$ that make a $\cC\cS$-perfect fit. Then there is a stable leaf $s'$ such that $c$ and $s'$ make a perfect fit.
	
	The symmetric statement holds for $\cS\cC$-perfect fits.
\end{lemma}
\begin{proof}
	This is \cite[Lemma 4.2]{BFFP-prequel}, whose proof remains valid with the obvious modifications.
\end{proof}

\section{Branching foliations and good lifts}\label{sec-branching_goodlift}


Fix a closed $3$-manifold $M$ whose fundamental group is not virtually solvable, a partially hyperbolic diffeomorphism $f \colon M \to M$ homotopic to the identity, and a good lift $\ft$. We will assume that $f$ is orientable (Definition~\ref{def:oriented}) so that we have center stable and center unstable branching foliations $\fbs$ and $\fbu$ which are well-approximated by taut foliations (Theorem~\ref{teo-BI}). This can be achieved by taking an iterate of $f$ and lifting to a finite cover of $M$ -- we will deal with the effects of replacing $f$ and $M$ in \S\ref{sec-doubleinvndc}.

In this section we will study the way that a good lift $\ft$ acts on the lifted branching foliations $\wfbs, \wfbs$ in the universal cover $\mt$.

\subsection{Translation-like behavior}\label{ss.translation-like}
In this section, we will see that the action of $\ft$ on the center stable leaf space must look locally like a translation. Identical statements hold for the center unstable foliation.

\begin{remark}
	In fact, the results in this subsection are not really particular to partially hyperbolic diffeomorphisms. They apply to any diffeomorphism that is homotopic to the identity that preserves a branching foliation well-approximated by taut foliations.
In addition in this subsection we also do not need
to assume that $\pi_1(M)$ is virtually solvable.
\end{remark}

The key to this section is the following fact:

\begin{lemma}[Big Half-Space Lemma]
	Let $L$ be a leaf of $\wfbs$. For any $R>0$, there exists a ball of radius $R$ contained in each complementary region of $L$.
\end{lemma}
\begin{proof}
	This lemma holds for true foliations -- see \cite[Lemma 3.3]{BFFP-prequel} -- so it suffices to consider a leaf corresponding to $L$ in the approximating foliation $\wfes$ for $\epsilon$ sufficiently small.
\end{proof}
\begin{remark}
 Note that the tautness of the foliation is essential for this result to hold. The branching foliation in the non-dynamically coherent example of \cite{HHU-noncoherent}, for instance, do not satisfy that lemma.
\end{remark}

\begin{definition}[Regions between leaves]
	Let $K, L \in \wfbs$ be distinct leaves. In the leaf space, $\lcsb \smallsetminus \{K, L\}$ consists of three open connected components. Only one of these components accumulates on both $K$ and $L$ --- we call this the \emph{open $\lcsb$-region between $K$ and $L$}.  Its closure in $\lcsb$, which is obtained by adjoining $K$ and $L$, is called the \emph{closed $\lcsb$-region between $K$ and $L$}.
\end{definition}

\begin{remark}
	Note that the subset of $\mt$ that corresponds to the open $\lcsb$-region between two leaves may not be open. However, the subset of $\mt$ that corresponds to the closed $\lcsb$-region between two leaves is closed. It is also connected, but its interior may not be. See Figure \ref{f.4M}.
	
	\begin{figure}[ht]
		\begin{center}
			\includegraphics[scale=0.64]{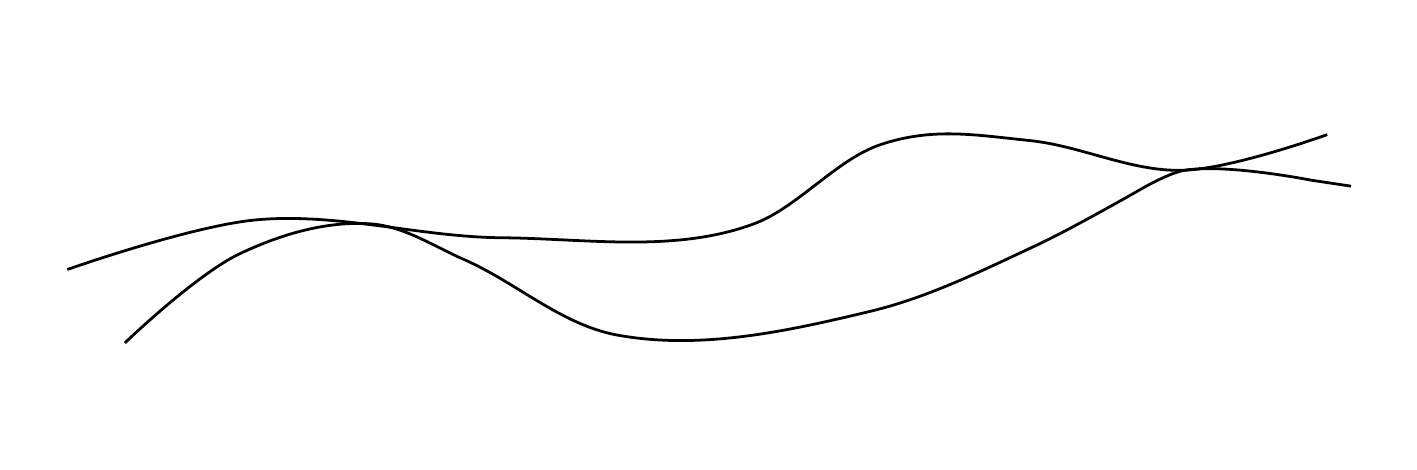}
			\begin{picture}(0,0)
			\put(-94,44){$V$}
			\put(-24,43){$L$}
			\put(-24,68){$K$}
			\put(-129,72){$U$}
			\put(-139,4){$W$}
			\end{picture}
		\end{center}
		\vspace{-0.5cm}
		\caption{The interior of the closed region between leaves may not be connected.}\label{f.4M}
	\end{figure}
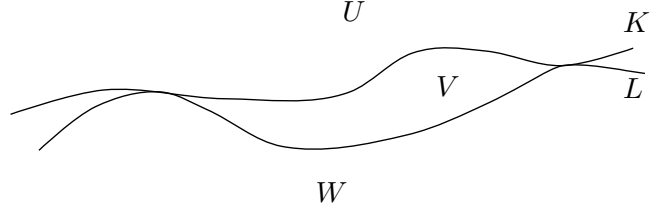
\end{remark}

The following is the equivalent of \cite[Proposition 3.5]{BFFP-prequel}. The same proof applies if one considers complementary regions and regions between leaves as subsets of $\mt$ and $\lcsb$ as appropriate.

\begin{proposition}\label{p.translatedleaves_nonDC}
	If $L \in \wfbs$ is not fixed by a good lift $\ft$, then
	\begin{enumerate}
		\item the closed $\lcsb$-region between $L$ and $\ft(L)$ is an interval,
		\item $\ft$ takes each coorientation at $L$ to the corresponding coorientation at $\ft(L)$, and
		\item the subset of $\mt$ corresponding to the closed $\lcsb$-region between $L$ and $\ft(L)$ is contained in the closed $2R$-neighborhood of $L$, where $R = \max_{y \in \mt} d(y,\ft(y))$.
	\end{enumerate}
\end{proposition}
\begin{remark}
 In the above proposition, we may a priori have that $L$ and $\ft(L)$ merge.
\end{remark}

Using Proposition \ref{p.translatedleaves_nonDC} we therefore also obtain the equivalent of \cite[Proposition 3.7]{BFFP-prequel}.

\begin{proposition}\label{p.dichotomy_branching}
	The set $\Lambda \subset \lcsb$ of leaves that are fixed by $\ft$ is closed and $\pi_1(M)$-invariant. Each connected component $I$ of $\lcsb \setminus \Lambda$ is acted on by $\ft$ as a translation, and every pair of leaves in $I$ are a finite Hausdorff distance apart.
\end{proposition}

In the above proposition, one has to be mindful again that ``open'' and ``closed'' refer to the topology on the leaf space $\lcsb$, and not the topology on $\mt$.

When $\fbs$ is $f$-minimal (Definition~\ref{def.fminimal}), we deduce the following dichotomy from Proposition \ref{p.dichotomy_branching}:
\begin{corollary}\label{c.minimalcase_branching_case}
	If $\fbs$ is $f$-minimal, then either
	\begin{enumerate}
		\item $\ft$ fixes every leaf of $\wfbs$, or
		\item $\fbs$ is $\R$-covered and uniform, and $\ft$ acts as a translation on the leaf space $\lcsb$.
	\end{enumerate}
\end{corollary}

\begin{proof}
	Although the proof is conceptually identical to that of the corresponding result in the dynamically coherent case, \cite[Corollary 3.10]{BFFP-prequel}, we will redo it since the distinction between the topology in $\lcsb$ and $\mt$ becomes important.
	
	Let $\Lambda$ be the set of leaves that are fixed by $\ft$. Since $\ft$ commutes with deck transformation, each deck transformation preserves $\Lambda$.
	In particular, if $I$ is a component of $\lcsb \setminus \Lambda$ and $g \in \pi_1(M)$, one has either $g(I) = I$ or $g(I) \cap I = \emptyset$.

	So $\Lambda$ is invariant under $\ft$ and deck transformations, saturated by $\wfbs$, and closed in $\lcsb$ (by Proposition \ref{p.dichotomy_branching}).
	
	Let $\widetilde B \subset \mt$ be the union of the points
in all leaves in $\widetilde \Lambda$, and let $B = \pi(\widetilde B) \subset M$.
	Since $\Lambda$ is closed in $\lcsb$, $\widetilde B$ is closed in $\mt$, and $B$ is closed in $M$. In addition. $B$ is $f$-invariant. Since $\fbs$ is $f$-minimal, $B$ is either $\emptyset$ or $M$.
	
	If $B$ is empty then $\Lambda$ is empty and Proposition \ref{p.dichotomy_branching} implies that we are in case (2).
	
	If $B = M$ then $\widetilde B = \mt$, and we have to prove that $\Lambda = \lcsb$. This follows from the more general Lemma \ref{closedsets}, but it also has the following more direct proof:
	
	Suppose $\Lambda \neq \lcsb$. Let $I$ be a connected component of $\lcsb \smallsetminus \Lambda$.  Let $J$ be the set of points of $\mt$ contained in a leaf in $I$. The set $I$ is open (in $\lcsb$) and $\ft$ translates leaves in $I$. It follows that the interior in $\mt$ of $J$ is non-empty. 
	These points in the interior of $J$ are not contained
	in $\widetilde B$. This contradicts $\wt B = \mt$.
	So $\Lambda = \lcsb$ and we are in case (1).
\end{proof}

This immediately implies the trichotomy in \S\ref{sss-introLeafSpaces}.

\subsection{Ruling out fixed points}\label{ss.nofixedpoints}
Let us now find conditions under which we show that our good lift $\ft$ has no fixed points in $\mt$. We will use the following lemma.

\begin{lemma}\label{l.nofixedpoints_branching_case}
Let $L \in \wfbs$ be a center stable leaf that is fixed by $\ft$. Suppose that for every $y \in L$ one can find a leaf $L' \in \wfbs$ that is fixed by $\ft$ and intersects the unstable leaf through $y$ in a point other than $y$. Then no nontrivial power of $\ft$ fixes a point in $L$.
\end{lemma}
\begin{proof}
Suppose that $\ft^n$ fixes a point $x \in L$ for some $n \neq 0$. One can assume after possibly switching signs that $n > 0$. Then expansion of the unstable leaf $u$ through $x$ implies that no leaf $L'$ that intersects $u$ at a point other than $x$ can be fixed.
\end{proof}

Compare this with the simpler statement in the dynamically coherent setting, \cite[Lemma~3.13]{BFFP-prequel}, where it suffices to assume $L$ is not isolated in the set of fixed leaves.

\begin{corollary}
	If $\ft$ fixes every center stable leaf, then it has no fixed or periodic points in $\mt$.
\end{corollary}

This follows immediately from the lemma. We will now exclude the existence of fixed or periodic points under the assumption of $f$-minimality.

\begin{theorem}\label{t.good_lifts_have_no_fixed_points}
 	If $\fbs$ or $\fbu$ is $f$-minimal, then $\ft$ does not have any fixed or periodic points in $\mt$.
\end{theorem}
\begin{proof}
Assume without loss of generality that $\fbs$ is $f$-minimal. By the dichotomy in Corollary \ref{c.minimalcase_branching_case}, $\ft$ either fixes every leaf of $\wfbs$, or acts as a translation on $\lcsb$.

If $\ft$ fixes every leaf of $\wfbs$ the result follows from Lemma~\ref{l.nofixedpoints_branching_case}. If $\ft$ acts as a translation on $\lcsb$, then for any leaf $L$ of $\wfbs$ one has $\ft^i(L) \cap L = \emptyset$ for $|i|$ sufficiently large.
\end{proof}

A noteworthy consequence is the non-existence of ``contractible periodic points'' under the assumption of $f$-minimality.

\begin{definition}\label{def.contractible_periodic_point}
Let $g$ be a homeomorphism of a manifold homotopic to the identity.
A point $p$ is a \emph{contractible periodic point} if $g^n(p) = p$ for some $n \neq 0$ and there is a homotopy $H\colon M \times [0,1] \to M$ from the identity to $g$, such that the concatenation of the paths $H(p,\cdot), H(g(p), \cdot), \dots, H(g^{n-1}(p),\cdot)$ is homotopically trivial.
\end{definition}
Notice that if $p$ is a contractible periodic point of $g$ of period $n$ then there exists a good lift $\wt g$ of $g$ and a lift $\wt p$ of $p$ such that $\wt g^n (\wt p) =  \wt p$. Thus, Theorem \ref{t.good_lifts_have_no_fixed_points} immediately yields:
\begin{corollary}\label{c.contractible}
If $f$ admits a $f$-minimal branching center stable or center unstable foliation, then $f$ has no contractible periodic points.
\end{corollary}

This completes the proof of Theorem \ref{thm-nocontractible} in the $f$-minimal case. The hyperbolic and Seifert fibered cases follow from Proposition \ref{p.hypSeifminimal_nonDC}.

\subsection{Fundamental groups of leaves}\label{ss.fundamental_group_branching}
The leaves of $\fbs$ and $\fbu$ are immersed surfaces which may not be injectively immersed. In the universal cover, however, the leaves of $\wfbs$ and $\wfbu$ are properly embedded planes (cf.~Section \ref{ss.leafspace}).

It follows that there may be a closed loop in a leaf with a corresponding element of $\pi_1(M)$ that fixes no lift of that leaf in the universal cover. These elements are not useful for our purposes, so we will remove them by convention:
\begin{convention}
	When working with a fixed lift $L$ of a leaf $C$ of $\fbs$ or $\fbu$, we will say that an element $\gamma \in \pi_1(M)$ is in the fundamental group of $C$ if it stabilizes $L$.
\end{convention}

There is another way of seeing this notion of fundamental group arise: Recall (Theorem \ref{teo-BI}) that the branching foliations are approximated by true foliations $\feu$ and $\fes$ and that there exists maps, $\hs$ and $\hu$ 
mapping leaves of $\fes$ (or $\feu$) to those of $\fbs$ (or $\fbu$).
Then, a loop is in the fundamental group of a leaf $C$ of $\fbs$ if and only if it is freely homotopic to a loop in a corresponding leaf $C_{\epsilon}$ of $\fes$, for every $\epsilon$ small enough. Notice that if there are several leaves that project to $C$, in the universal cover, take a lift $L$ and it follows from Proposition \ref{prop.leafspace} that the set of leaves that projects to $L$ is an interval in the leaf space of $\wfes$. It follows that $\hs$ lifts to a equivariant (with respect to the defined fundamental group of $C$) diffeomorphism from the boundary leaves of the closed interval to $L$. We call such a leaf $L_\epsilon$ and denote $C_\epsilon=\pi(L_{\epsilon})$. 

In other words, for us, the fundamental group of $C$ based at $y$ will be exactly $(\hs)_*(\pi_1(C_{\epsilon},y_0))$ where $\hs(y_0)=y$. 

In particular, since $\fes$ and $\feu$ are taut foliations without Reeb components, each leaf is $\pi_1$-injective in $M$. Thus, this second interpretation helps explain our convention: the closed loops in a leaf of $\fbs$ are either in the fundamental group as we defined it, or they are due to a self-intersection. In that case, they are not an essential feature of the leaf, as they stopped being closed when pulled-back to the approximating leaf.

Following our convention, we will then say that a leaf $C$ of the branching foliation is a plane, a cylinder, or a M\"obius band if its corresponding approximated leaf $C_{\epsilon}$ is, respectively, a plane, a cylinder, or a M\"obius band, for any small enough $\epsilon$.

Using these conventions, \cite[Proposition 3.14]{BFFP-prequel} holds for the leaves of the branching foliations whenever $\ft$ has no fixed points in the leaf (cf.~Lemma \ref{l.nofixedpoints_branching_case}). For ease of reference, we restate it here.

\begin{proposition}\label{p.planeannuliorfixedpoints_branching}  
Assume that $\ft$ fixes a leaf $L$ of $\wfbs$ then, $C = \pi(L)$ has cyclic fundamental group (thus it is either a plane, an annulus or a M\"{o}bius band), or $L$ has a point fixed by $\ft$.
\end{proposition}

\begin{remark}\label{r.distanceleaf}
Similarly, because of possible self-intersections, we need to be careful on how to define the path-metric on a leaf of $\fbs$ or $\fbu$.

If $C$ is a leaf of, say, $\fbs$, we define a \emph{path} on $C$ as a continuous curve $\eta$ that is the projection of a continuous curve $\wt \eta$ in a lift $L$ of $C$ to $\mt$. We then define the path-metric on $C$ as usual, but considering only the paths as defined before.

Notice that not every continuous curve $\eta$ on $C$ is a path in the above sense, as there might not exists any lift of $\eta$ that stays on only one lift of $C$.
\end{remark}

Still the analogue of \cite[Lemma 3.11]{BFFP-prequel} holds: 

\begin{lemma}\label{lema-boundedinfixcsleaf-nonDC} 
If $\ft$ fixes every leaf of $\wfbs$ (resp. $\wfbu$) then there is $K>0$ such that for every $L \in \wfbs$ (resp. $L \in \wfbu$) we have that $d_L(x, \ft(x))<K$. 
\end{lemma}

\subsection{Gromov hyperbolicity of leaves}\label{ss.gromov_hyperbolicity}

We now prove a version of \cite[Lemma 3.20]{BFFP-prequel} in
the non dynamically coherent setting.

\begin{lemma}\label{GromovhypnonDC}
If $\fbs$ is $f$-minimal, and $\ft$ fixes every leaf of $\wfbs$. Then each leaf of $\fbs$ is Gromov hyperbolic.
\end{lemma}

\begin{proof}
The foliation $\fes$ is taut. Thus, Candel's theorem \cite{Candel} asserts that either all the leaves of $\fes$ are Gromov hyperbolic or there is a holonomy invariant transverse measure (of
zero Euler characteristic).

Assume for a contradiction that $\mu$ is a holonomy invariant transverse measure.

Since $\fes$ is not $f$-invariant, we have to adjust the proof given in \cite{BFFP-prequel}.

The transverse measure $\mu$ lifts to a measure $\widetilde\mu$ transverse to $\wfes$. Thus, $\widetilde\mu$ defines a measure on $\lcse$, the leaf space of $\fes$. 

Let $g_{\epsilon,s}\colon \lcse \to \lcsb$ be the canonical projection between the leaf spaces of $\fes$ and $\fbs$ (see section \ref{sss.ls-centerstable}). Let $\widetilde\nu := \left(g_{\epsilon,s}\right)_{\ast} \widetilde\mu$ be the corresponding measure on $\lcsb$. Now $\widetilde\nu$ is $\ft$-invariant since $\ft$ is the identity on $\lcsb$, and it is also $\pi_1(M)$-invariant as $\widetilde\mu$ is. 
The support of $\widetilde\nu$ in $\lcsb$ is a 
closed set $Z$ in $\lcsb$ that is 
$\ft$-invariant and $\pi_1(M)$-invariant.

The measure $\widetilde\nu$ on $\lcsb$ can also be considered as a measure on the set of transversals to $\wfbs$ in $\widetilde M$: For any transversal $\tau$ to $\wfbs$ in $\widetilde M$, we define $\wt\nu(\tau)$ as the $\wt\nu$-measure of the set of leaves in $\lcsb$ that intersects $\tau$.
Notice that the measure of a point in $\mt$ (which can be thought of
as a degenerate transversal) can be positive if the image of that point in $\lcsb$ is an interval.

Note also that we refrained from calling $\wt\nu$ a transverse measure to $\wfbs$ because it is by no means holonomy invariant. In fact holonomy itself is not well defined for a branching foliation.
Still $\wt\nu$ satisfies the property that
if $\tau_1, \tau_2$ are transversals and every leaf 
intersecting $\tau_1$, also intersects $\tau_2$, then
$\widetilde \nu(\tau_1) \leq \widetilde \nu(\tau_2)$.

Projecting down to $M$,the measure $\widetilde \nu$ induces a measure $\nu$ on the set of transversals  
to $\fbs$ on $M$.

Let $\tau$ be any unstable segment in $M$.
Since $\ft$ fixes every leaf of $\wfbs$,
the measure of $f^{i}(\tau)$  ($= \nu(f^i(\tau))$) 
is equal to $\nu(\tau)$
for any integer $i$. We can choose $i$ very big and negative
so that the length of $f^{i}(\tau)$ is extremely small. 
Therefore it is contained in a small foliated box of $\fbs$, which
is the projection of a compact foliated box of $\fes$. 
It follows that $\nu(\tau)$ is uniformly bounded.
In particular this implies that the $\nu$-measure of any
unstable leaf in $M$ is bounded above.
In turns, it implies that for any $j > 0$ (assumed big enough),
there is an unstable segment $u_j$ of length $> j$
which has $\nu(u_j)$ measure $< 1/j$. Taking the midpoint of these 
segments and a converging subsequence, we obtain a full unstable leaf,
call it $\zeta$, so that $\zeta$ has $\nu(\zeta) = 0$ (since $\nu(\zeta) < 1/j$ for all big enough $j$).

Let $Y$ be the union of the leaves of $\fbs$ that do
not intersect $\zeta$ or any of its iterates by $f$.
Then $Y$ is a closed subset of $M$ and clearly $f$-invariant.
Let $L$ be a leaf in $\wfbs$ which is in $Z$, the support of $\widetilde\nu$.
Then by definition of support of $\widetilde\nu$, it follows
that  $\pi(L)$ cannot intersect
$\zeta$ or any of its iterates by $f$. Hence $\pi(L)$ is in $Y$.
In particular $Y$ is not empty.
This contradicts the fact that $\fbs$ is $f$-minimal,
and hence cannot happen.

This finishes the proof of the lemma.
\end{proof}


\section{Center dynamics in fixed leaves}\label{sec-centerdynamics}

This section deals with the dynamics of center leaves within center stable (and center unstable) leaves. It is one of the first places where we encounter significant difficulties compared with the dynamically coherent setting. 

In \cite[Proposition 4.4]{BFFP-prequel} we found a condition for the existence of center leaves that are fixed by a good lift, but the proof does not work without dynamical coherence \cite[Remark 4.8]{BFFP-prequel}.

Throughout this section we continue to assume that $f$ is orientable (Definition~\ref{def:oriented}).

\begin{definition} \label{d.coarsely_contracting}
	Let $c \subset M$ be a center leaf that is fixed by $f$. We say that $c$ is \emph{coarsely contracting} if it is homeomorphic to the line, and it contains an non-empty compact interval $I$ such that for each compact interval $J \subset c$ whose interior contains $I$ has the property that $f(J) \subset \mathring{J}$.
	
	We say that $c$ is called \emph{coarsely expanding} if it is coarsely contracting for $f^{-1}$.
	
	We also naturally extend the definition of coarse contraction/expansion to leaves that are periodic under $f$.
\end{definition}

The following is the main result of this section.

\begin{proposition}\label{p.alternnonDC} 
	Suppose that $\fbs$ is $f$-minimal, and there is a good lift $\ft$ that fixes every center stable leaf but no center leaf in $\mt$. Then every $f$-periodic center leaf in $M$ is coarsely contracting.
\end{proposition}

Note that a coarsely contracting periodic leaf must contain a periodic point.

If $\fbu$ is $f$-minimal, and there is a good lift $\ft$ that fixes every center unstable leaf in $\mt$ then one concludes that each periodic center leaf is coarsely expanding.

We will see in Proposition~\ref{periodiccenter} that one can always find $f$-periodic center leaves.

\subsection{Fixed centers or coarse contraction}\label{ss.fixed_center_or_coarse_contraction}
We begin with a preliminary result.
\begin{lemma}\label{l.iteratesdontfix}
Suppose that $\ft$ fixes every center stable leaf but no center leaf in $\mt$. Then the same holds for every iterate $\ft^n$ with $n \neq 0$.
\end{lemma} 

\begin{proof}
Suppose that $\ft^n$ fixes a center leaf $c_0$ for $n > 0$, and let $L$ be a center stable leaf that contains $c_0$ (which is fixed by $\ft$ by hypothesis). Since $f$ is orientable, $\ft$ preserves a transverse orientations to the center and stable foliations on $L$. 

Let $A^c$ be the axis for the action of $\ft$ on the center leaf space in $L$ (i.e., the set of center leaves $c$ such that $\ft(c)$ separates $c$ from $\ft^2(c)$, see \cite[Appendix~E]{BFFP-prequel}). Since $\ft^n(c_0) = c_0$, the leaf $c_0$ cannot be in $A^c$. If $c_0$ is not in $\partial A^c$ then we can replace it with the unique center leaf that separates $c_0$ from $A^c$. Thus we can assume that $c_0\in \partial A^c$.

Recall (see \cite[Proposition 2.15]{Bar98}) that the boundary of the axis of a homeomorphism on a $1$-manifold splits into three disjoint sets: the ``positive'' boundary, ``negative'' and ``middle'' boundary. That is, $\partial A^c$ contains three types of leaves, the center leaves $c$ such that $c$ and $\ft(c)$ are non separated on their positive side, the leaves $c$ such that $c$ and $\ft(c)$ are non separated on their negativeside, and the leaves $c$ that are non separated with a leaf $c'$ in $A^c$.

If $c_0$ was in the ``middle'' boundary, then we would have that there exists $c'\in A^c$ not separated with $c_0$. Thus $c'$ and $\ft^n(c')$ are separated, contradicting that $c_0 = \ft^n(c_0)$. So $c_0$ must be either in the positive or negative boundary. In particular, $c_0$ and $\ft(c_0)$ are non separated.

Now, consider the closure of the set of stable leaves intersect $c_0$. There exists a unique stable leaf $s_0$ in the boundary of that set that separates $c_0$ from $\ft(c_0)$. Therefore, $s_0$ must be fixed by $\ft^n$, and hence contains a fixed point $x$ of $\ft^n$.

In particular, we found a periodic point of $\wt f$, thus, by Brouwer Translation Theorem (see e.g. \cite{Brouwer}) $\wt f$ must also admit a fixed point, say $y$. Since the center leaves through $y$ form a closed interval (Lemma \ref{l.center_leaf_space_in_L}), there exists at least one closed center leaf through $y$, a contradiction.
\end{proof}

In order to obtain coarsely contracting center leaves we will use the following tool. 

\begin{proposition}\label{p.graphtransf2} 
Suppose that $\ft$ fixes every center stable leaf in $\mt$, and let $L$ be a center stable leaf that is also fixed by some $\gamma \in \pi_1(M) \smallsetminus \{\id\}$.

Assume that there exists a properly embedded $C^1$-curve, $\hat\eta \subset L$ that is transverse to the stable foliation and fixed by both $\gamma$ and $\ft$.

\begin{itemize}
	\item If $\ft$ does not act freely on $\cL^c_L$ then there is a center leaf in $L$ fixed by both $\ft$ and $\gamma$.
	
	\item If $\ft$ acts freely on $\cL^c_L$ then every $f$-periodic center leaf in $\pi(L)$ is coarsely contracting. 
\end{itemize} 
\end{proposition} 

Note that in the first case the center leaf projects to an $f$-invariant closed center leaf. 

Note also that hypothesis of Proposition \ref{p.graphtransf2} are implied by the conclusion of the Graph Transform Lemma \cite[Appendix H]{BFFP-prequel}. 

We will use the following result from \cite{BFFP-prequel}, whose proof works equally well in the non dynamically coherent case:

\begin{lemma}[Lemma 4.15 in \cite{BFFP-prequel}]\label{l.condition_for_coarse_contraction}
	Let $c$ be a center leaf in a center stable leaf $L \subset \mt$. Suppose that $L$ is Gromov-hyperbolic, and fixed by $\ft$ and some nontrivial $\gamma \in \pi_1(M)$. Moreover, assume that there exist two stable leaves $s_1, s_2$ on $L$ such that:
	\begin{enumerate}
		 \item The center leaf $c$ is in the region between $s_1$ and $s_2$;
		 \item The leaves $s_1$ and $s_2$ are a bounded Hausdorff distance apart;
		 \item The leaves $c$, $s_1$ and $s_2$ are all fixed by $h = \gamma^n \circ \ft^m$, $m\neq 0$.
	\end{enumerate}
	Then there is a compact segment $I\subset c$, such that $h$ (if $m > 0$) or $h^{-1}$ (if $m < 0$) acts as a contraction on $c \smallsetminus \mathring{I}$.
\end{lemma}

\begin{proof}[Proof of Proposition \ref{p.graphtransf2}]
Since $\ft$ fixes every leaf of $\fbs$, Lemma \ref{l.nofixedpoints_branching_case} implies that it fixes no point in $\mt$, and hence fixes no stable leaf.

Let $S$ be the stable saturation of the curve $\hat \eta$. Let $\alpha = \pi(\hat \eta)$. The curve $\alpha$ is closed, $f$-invariant, and tangent to the center bundle.

\vskip .07in
\noindent
{\bf {Case 1 -}} We start by assuming that $\ft$ fixes a center leaf $c$ in $L$.

Suppose that $c$ and $\hat \eta$ do not intersect
a common stable leaf. Then $c$ does not intersect the set
$S$ and there is a unique stable leaf $s$ contained in the boundary
of $S$ such that $s$ separates $S$ from $c$. Since both $S$
and $c$ are $\ft$-invariant, so is $s$. But then $\ft$ must admit a fixed point in $s$, contradiction\footnote{Note the distinction of $c$ being fixed by $\ft$ as
opposed to $\pi(c)$ periodic under $f$. It is the first 
property which creates a fixed point of $\ft$ and a contradiction.}.

Therefore there is a stable leaf $s$ intersecting $c$ in $y$ and
$\hat \eta$ in $x$. 
Iterating forward by $\ft$, we deduce that $d(\ft^n(y),\ft^n(x))$
converges to zero as $y$ and $x$ are in the same stable leaf.
Since both $c$ and $\hat\eta$ are $\ft$-invariant, it implies that $c$ and $\hat\eta$ are asymptotic (note that $c$ and $\hat\eta$ may or may not intersect). Calling $\alpha=\pi(\hat\eta)$ the projection to $M$, we deduce that $\pi(c)$ accumulates onto $\alpha$. But, as $\alpha$ is closed and $\pi(c)$ is a center leaf, we deduce that $\alpha$ is also a center leaf. Hence $\hat \eta$ is the required center leaf of the
first option of the proposition.

\vskip .07in
\noindent
{\bf {Case 2 -}} Assume now that $\ft$ acts freely on the center leaf space of $L$.

According to Lemma \ref{l.iteratesdontfix}, $\ft^n$ also acts freely on the center leaf space
of $L$ for any $n \not = 0$. 

We need to prove now that every center leaf in $\pi(L)$ that is periodic must be coarsely contracting.

Let then $c$ be a center leaf in $L$ such that $\pi(c) = e$ is periodic under $f$, say of period $m$. Then, for some $\gamma_1 \in \pi_1(M) \smallsetminus \{\id\}$, we have $c = \gamma_1 \ft^m(c)$. (Note that one can show under our current assumptions that $\pi(L)$ projects to an annulus, so $\gamma$ and $\gamma_1$ are both powers of a particular deck transformation, but we do not need that fact for the proof).
Let 
\[
 h:= \gamma_1 \circ \ft^m.
\]

We now want to show that either $c$ intersects $\hat\eta$, or there exists another center leaf, also fixed by $h$, that does.

Suppose first that $c$ intersects $S$, i.e., there exists a stable leaf intersecting both $c$ and $\hat\eta$. 
Since the stable distance is contracted by $h$, which fixes both $c$ and $\hat\eta$, we obtain that either $c$ and $\hat\eta$ are asymptotic, or they intersect. If $c$ and $\hat\eta$ are asymptotic, then, as in case 1, we deduce that $\hat\eta$ must be a center leaf, contradicting the fact that $\ft$ acts freely on the center leaf space. Thus we must have that $c$ intersects $\hat\eta$.

Suppose now that $c$ does not intersect $\hat\eta$, and thus does not intersect $S$. Then there is a unique stable leaf $s$ in $\partial S$ that separates $\hat\eta$ from $c$. That leaf $s$ must then be invariant by $h$, so admits a fixed point for $h$. Then at least one center leaf, say $c_1$, through that fixed point must be fixed by $h$. Since $c_1$ intersects $S$ and is invariant by $h$, it must intersect $\hat\eta$.

Thus in any case, we have a center leaf $c_1$ that intersects $\hat\eta$, is invariant by $h$, and, by the above argument has both ends that escapes compacts sets of $L$.

Let $I$ be the projection of $c_1$ onto $\hat\eta$ along stable leaves.

Suppose first that $I$ is unbounded. Then, considering iterates by $f^m$, we deduce that $\pi(c_1)$ must be asymptotic to $\pi(\hat\eta)$, so $\hat\eta$ must be a center leaf, which is not allowed, since $\ft$ is assumed to act freely on center leaves.

So $I$ is bounded in $\hat\eta$. Let $s_1$ and $s_2$ be the stable leaves through the two endpoints of the interval $I$. Since $I$ is fixed by $h$, so are $s_1$ and $s_2$.
Moreover, the center leaf $c_1$, as well as $c$ if it is different from $c_1$, is in between $s_1$ and $s_2$.

Now, $\ft$ acts as a translation on $\hat\eta$, so there exists $k\in\mathbb{Z}$ such that $s_2$ separates $s_1$ from $\ft^k(s_1)$. By Lemma \ref{lema-boundedinfixcsleaf-nonDC}, $s_1$ and $\ft^k(s_1)$ are a bounded Hausdorff distance apart. Thus $s_1$ and $s_2$ are a bounded Hausdorff distance apart. So $c$ satisfies all the conditions for 
Lemma \ref{l.condition_for_coarse_contraction} to hold, thus it is coarsely expanding. 

 This finishes the proof of Proposition \ref{p.graphtransf2}.
\end{proof}

We are now ready to prove the main result of this section.

\begin{proof}[Proof of Proposition \ref{p.alternnonDC}] Let $e \subset M$ be an $f$-periodic center leaf, and let $c \subset \mt$ be a lift of $e$. If $m > 0$ is the period of $e$, then 
$c$ and $\ft^m(c)$ both project to $e$, so there is an element $\gamma' \in \pi_1(M)$ with $\gamma'( \ft^m (c)) = c$.

Choose a leaf $L \in \wfbs$ that contains $c$. Then $\gamma'$ is in the stabilizer of $L$, because $\ft$ leaves invariant every leaf of $\wfbs$. Since $\ft^m$ acts freely on the center leaf space (cf.~Lemma \ref{l.iteratesdontfix}), $\gamma'$ is not the identity.

Since $\ft$ does not have any fixed points, Proposition \ref{p.planeannuliorfixedpoints_branching} implies that the stabilizer of $L$ in $\mt$ is infinite cyclic. Thus, there exists $\gamma \in \pi_1(M) \setminus \{\mathrm{id} \}$ such that 
$\gamma^n \circ \ft^m (c) = c$ for some $n\in\ZZ$, $n\neq 0$, and such that 
$\gamma$ generates the stabilizer of $L$.  
Let
\[
h := \gamma^n \circ \tilde f^m.
\]
Notice that $h$ is still a partially hyperbolic diffeomorphism and has bounded derivatives.

Since $\ft$ acts freely on $\cL^c_L$, it must also act freely on $\cL_L^s$. Let $A^s$ be the axis for the action of $\ft$ on the stable leaf space $\cL^s_L$ (see \cite[Appendix E]{BFFP-prequel}). No stable leaf in $M$ can be closed, so $\gamma$ must also act freely on $\cL^s_L$. Since $\gamma$ and $\ft$ commute, $A^s$ is also the axis for the action of $\gamma$ on $\cL^s_L$. The axis $A^s$ can be a line or a countable union of intervals.

\vskip .1in
Suppose first that $A^s$ is a line. 
Let $s$ be a stable leaf in $A^s$
and $p$ in $s$. Then $p$ and $\gamma p$ can be connected by a
transversal to the stable foliation, chosen so that the
projection to $\pi(L)$ is a smooth simple closed curve.
Let $\eta$ be the union of the $\gamma$ iterates of
this segment. Applying the Graph Transform Lemma \cite[Lemma H.1]{BFFP-prequel} to $\eta$ we obtain a curve $\hat \eta$ which satisfies the properties in the
hypothesis of Proposition \ref{p.graphtransf2} as desired.

\vskip .1in

Now suppose that $A^s$ is a countable union of intervals

\[
 A^s = \bigcup_{i \in {\ZZ}} [s^-_i,s^+_i] = \bigcup_{i \in {\ZZ}} T_i.
\]

Our first claim is that there exists $s\in A^s$, fixed by $h$, such that the center leaf $c$ is between $\gamma^{-1} s$ and $\gamma s$.

Suppose that $c$ intersects some stable leaf $s'$ in $A^s$, then $s'$ is in
a unique $T_i$ for some $i$ (the center leaf $c$ cannot intersect two different intervals otherwise $c$ would intersect two non-separated leaves, which is impossible). Then, since $h$ fixes $c$, it also fixes the axis $A^s$ and preserves the transverse orientation. It follows that $h(T_j) = T_j$ for all $j$. In this case we set $s = s^+_i$.  The leaf $s$ is fixed by $h$ and there exists $k\neq 0$ such that $\gamma^{\pm 1} T_i = T_{i\pm k}$. Thus $T_i$ is in between $\gamma^{-1}  s$ and $\gamma s$ and hence, so is $c$. Recall here that $h$ preserves orientation.

Now, suppose instead that $c$ does not intersect $A^s$. 
Hence, there is a unique $i$ such that $s^+_{i-1} \cup s^-_i$ separates $c$ from all other stable leaves in $A^s$. We again set $s := s^+_i$. As before, since $h$ fixes both $c$ and $A^s$, and preserves the transverse orientation, it must fix $s$ also. The same argument as above also shows that $c$ is between $\gamma^{-1} s$ and $\gamma s$.

In either case we have found a stable leaf $s$ (chosen as a positive endpoint of one of the closed intervals $T_i$) that is fixed by $h$, such that $c$ lies between $\gamma^{-1} s$ and $\gamma s$. Notice that both $\gamma s$ and $\gamma^{-1} s$ are fixed by $h$. 

The leaf $\gamma^{-1} s$ is between $\gamma s$ and $\ft^{2m}(\gamma s) = \gamma^{-2n+1}s$ (assuming $n\geq 1$, otherwise between $\gamma s$ and $f^{-2m}(\gamma s)$). Hence the Hausdorff distance between $\gamma^{-1} s$ and $\gamma s$ is bounded above by a uniform constant $C>0$, depending only on $f$ and $m$.

Thus the center leaf $c$, fixed by $h$, lies between the stable leaves $\gamma s$ and $\gamma^{-1} s$, also fixed by $h$, which are a bounded Hausdorff distance apart. Moreover, the leaves of $\fbs$ are Gromov-hyperbolic by Lemma \ref{GromovhypnonDC}. These are all the conditions needed to apply Lemma \ref{l.condition_for_coarse_contraction}, so $c$ is coarsely contracting for $h$.
\end{proof}

\subsection{Existence of periodic center leaves}
In order to apply Propositions \ref{p.alternnonDC} and \ref{p.graphtransf2} we will need some way to find periodic center leaves.

\begin{proposition}\label{periodiccenter}
Let $f \colon M \rightarrow M$ be a partially hyperbolic diffeomorphism homotopic to the identity.

Suppose that some good lift $\ft$ fixes every center stable leaf in $\mt$. If $L$ is a center stable leaf fixed by some $\gamma \in \pi_1(M) \smallsetminus \{\id\}$, then there is an $f$-periodic center leaf in $\pi(L)$.
\end{proposition}

\begin{proof}
First notice that if one can prove the above result for a finite cover of $M$ and a finite power of $f$, then the same result directly follows for the original map and manifold. Thus, we may assume that $M$ is orientable, $f$ is orientation-preserving, and the branching foliations are both transversely orientable. 

Given these assumptions, $L$ projects to an annulus in $M$ (Proposition~\ref{p.planeannuliorfixedpoints_branching}). Let $\gamma$ be a generator of the stabilizer of $L$.

If $\ft$ fixes a center leaf in $L$, then it would project to a center leaf fixed by $f$, proving the claim. So we assume that $\ft$ acts freely on the center leaf space in $L$. This implies that $\ft$ also acts freely on the stable leaf space in $L$, and we can thus consider the stable axis $A \subset \cL^s_L$ of $\ft$. Since $\gamma$ also acts freely on the stable leaves, and commutes with $\ft$, it has the same set $A$ as its axis. This axis is either a line or a countable union of intervals.

If the axis is a countable union of intervals, there must be integers $n, m$ such that $h := \gamma^n \ft^m$ fixes one of the intervals, and hence a stable leaf $s$. One cannot have $m = 0$, since this would mean that $\gamma^n$ would fix a stable leaf, which is impossible. So $m \not = 0$, and $s$
projects to a periodic stable leaf $\pi(s)$ in $M$. This must contain a periodic point, and at least one center leaf through that point is periodic as desired.

If the axis is a line, then one can use the Graph Transform Lemma \cite[Appendix H]{BFFP-prequel} to see that there is a properly embedded curve in $L$ which is invariant under $\ft$ and $\gamma$. Then \cite[Lemma H.3]{BFFP-prequel} provides a periodic center leaf as desired. 
\end{proof}

\subsection{Additional result}

The intermediate results in this section also provide a proof
of the following result which will be needed later in this article.

\begin{proposition}\label{p.alternnonDC2}
Suppose that $\ft$ fixes every center stable leaf in $\mt$, and let $L$ be a center stable leaf that is also fixed by some $\gamma \in \pi_1(M) \smallsetminus \{\id\}$. Assume moreover that there is no center leaf  in $L$
fixed by $\ft$. Then, there is a center leaf $c$ in $L$ fixed by $h= \gamma^n \circ \ft^m$ for some $n, m$,  with
$m \neq 0$ and two stable leaves $s_1, s_2$ on $L$ such that:
	\begin{enumerate}
		 \item The center leaf $c$ separates $s_1$ from $s_2$ in $L$;
		 \item The leaves $s_1$ and $s_2$ are a bounded Hausdorff distance apart;
		 \item The leaves $c$, $s_1$ and $s_2$ are all fixed by $h = \gamma^n \circ \ft^m$, $m\neq 0$.
	\end{enumerate}
\end{proposition}
\begin{proof}
The conditions imply that $\pi(L)$ is an annulus.
Proposition \ref{periodiccenter} implies that there is a periodic
center in $\pi(L)$.

To prove Proposition \ref{p.alternnonDC2}  we revisit the proof of Proposition \ref{p.alternnonDC}. Since there is no center fixed by 
$\ft$ in $L$, then
as in the proof of Proposition \ref{p.alternnonDC} the map $\ft$ acts
freely on the stable leaf space.
As  in that proposition
we separate into whether the stable axis is a line or when it is a $\mathbb{Z}$-union of intervals. 

In the first case, as in Proposition \ref{p.alternnonDC} we
produce a curve $\hat \eta$ in $L$ which is invariant
under $\hat f$ and $\gamma$. We will use Proposition \ref{p.graphtransf2},
and the existence of such a curve $\hat \eta$ is necessary for that.
The analysis of Proposition \ref{p.graphtransf2} has cases
depending on the action of $\ft$ on the center leaf space $-$
as opposed to the action on the stable leaf space $A^s$.
However in this proposition we are assuming that the action
on the center leaf space in $L$ is free, so this is Case 2
of Proposition \ref{p.graphtransf2}, where the proof
showed the existence of a center leaf $c$ and
stable leaves $s_1, s_2$ satisfying the
conditions stated in this proposition, except perhaps that
$c$ separates $s_1$ from $s_2$. 

We now show that such a center leaf exists with this additional property. 
Suppose that this does not happen for $c$. This can only occur if
both ends of $\pi(c)$ are in the same end of the annulus
$\pi(L)$, or in other words, if $\pi(c)$ separates
$\pi(L)$.
Since the action of $\ft$ on
the center leaf space in $L$ is free it has an axis
denoted by $A^c$. The leaf $c$ is not in this axis.
If the axis $A^c$ is a real line then there is a unique
center leaf $c'$ in the axis $A^c$ which is either non
separated from $c$ or is non separated from a leaf which
separates $c$ from the axis. In either case it also follows
that $h$ fixes $c'$. We can then redo the analysis with
$c'$ instead of $c$. It will produce stable leaves $s_1, s_2$
fixed by $h$, with $c'$ between them, and now $c'$
separates $s_1$ from $s_2$.
If the center axis $A^c$ is a countable union
of intervals, there is a unique consecutive pair of intervals
so that $c$ is ``between" them. Then the boundary leaves
of these intervals are fixed by $h$. Choose $c'$ to be one
of these boundary leaves, and redo the proof with $c'$ instead
of $c$ to obtain the conclusion of the proposition.

The other case of this proposition is when the stable axis is a 
$\mathbb{Z}$-union of intervals.
Here we use the notation of the proof of Proposition \ref{p.alternnonDC},
where $A^s = \bigcup_{i \in \mathbb Z} [s^-_i,s^+_i] = 
\bigcup_{i \in \mathbb Z} T_i$.
Consider $s^+_0$, which is non separated in the stable leaf space
from $s^-_1$. There are $n, m$, $m \not = 0$ so that $h = \gamma^n \circ \ft^m$
fixes all $T_i$ and their boundary leaves.
Since $s^+_0, s^-_1$ are non separated consider a nearby stable
leaf $s$
which intersects transversals to both of them.
Choose $c_0$ center intersecting $s, s^+_0$, and choose $c_1$ center 
intersecting $s, s^-_1$.
Starting from $c_0$ 
and considering  the centers intersecting $s$ between $c_0 \cap s$
and $c_1 \cap s$ there is a first
center leaf, denoted by $c$ which does not intersect 
$s^+_0$. This center is fixed by $h$. 
Let $s_1 = s^+_0, \  s_2 = s^-_1$. They are both fixed by $h$.
In addition $c$ separates $s_1$ from $s_2$. Finally $s_1, s_2$ 
are a finite Hausdorff distance from each other in $L$.

This completes the proof of the proposition. 
\end{proof}

\section{Minimality for Seifert and hyperbolic manifolds}\label{sss.minimal} 

In this section we will show that when $M$ is hyperbolic or Seifert, the existence of a single fixed center stable leaf implies that every center stable leaf is fixed. This is considerably easier in the dynamically coherent case \cite[Proposition 3.15]{BFFP-prequel}.

We continue to assume that $f$ is orientable.

\begin{proposition}\label{p.hypSeifminimal_nonDC}
Suppose that $M$ is hyperbolic or Seifert fibered, and a good lift $\ft$ fixes some leaf of $\wfbs$. Then $\ft$ fixes every leaf of $\wfbs$, $\fbs$ is $f$-minimal, and every leaf of $\fes$ and $\fbs$ is either a plane or an annulus.
The same statement holds for $\fbu$.
\end{proposition} 

The main issue to extend the proof of \cite{BFFP-prequel} to the non dynamically coherent context is that here we cannot ensure the non-existence of fixed points of $\ft$ since Lemma \ref{l.nofixedpoints_branching_case} does not exclude fixed points when the branching foliation is not $f$-minimal. So, we will need to exclude the existence of fixed points for good lifts. We cannot exclude their existence in general, but we are able to show that they cannot exist in minimal sub-laminations of $\fbs$ or $\fbu$.

\subsection{No fixed points for good lifts} 

Note that the definition of $f$-minimality for the whole foliation can be applied to subsets: An $\fbs$-saturated subset of $M$ is $f$\emph{-minimal} if it is closed, non-empty, and $f$-invariant, and no proper saturated subset satisfies these conditions.

The goal of this subsection is to prove the following proposition, which does not assume that $M$ is hyperbolic or Seifert.
\begin{proposition}\label{l.analogL53}
Let $\ft$ be a good lift of $f$ to $\mt$.
Suppose that $\Lambda$ is a non empty $f$-minimal set 
of $\fbs$, such that every leaf $L$ of $\widetilde \Lambda = \pi^{-1}(\Lambda)$ 
is fixed
by $\ft$. 
Then $\ft$ has no fixed points in $\widetilde \Lambda$. 
\end{proposition}

We will prove this proposition by contradiction. So from now on, we assume that there is a fixed point $p$ of $\ft$ in a leaf $L$ contained in $\widetilde \Lambda$. This point projects to a fixed point $\pi(p)$ in $M$. Note that any leaf $L'$ of $\widetilde \Lambda$ that intersects the unstable leaf $u(p)$ through $p$ must have $L' \cap u(p) = p = L\cap u(p)$. This is because $L$ and $L'$ are both fixed, and unstable leaves are expanded.

\subsubsection{Many fixed points} 

The following property uses crucially the fact that $\Lambda$ is an $f$-minimal sublamination. 

\begin{lemma}\label{claim.fixed_points_close_by}
	There exists $b>0$ such that any point in a leaf of $\widetilde \Lambda$ is at distance at most $b$ (for the path metric on the leaf) from a fixed point of $\ft$.
\end{lemma}

\begin{proof}
	Otherwise, one can find a sequence of discs $D_i$ in leaves of $\widetilde \Lambda$ that contain no fixed points and whose radius goes to $\infty$. Up to deck transformations and subsequences, these disks converge to a full leaf $L_1$ of $\wfbs$ that is contained in $\widetilde \Lambda$. Here, the convergence is with respect to the topology of the center stable leaf space, which also implies convergence as a set of $\mt$. The leaf $L_1$ does not contain any fixed point of $\ft$. Indeed, the unstable leaf through a fixed point $q$ in $L_1$ would eventually intersect one of the discs $D_i$. Since $\ft$ fixes the leaves of $\widetilde \Lambda$, this would imply that the leaf through $D_i$ merges with $L_1$ and that $D_i$ contains the fixed point $q$, a contradiction.
	
	Since $L_1$ contains no fixed points, it does not contain the $\ft$-fixed point $p$, and $A = \pi(L_1)$ does not contain the $f$-fixed point $\pi(p)$. But the closure of $A = \pi(L_1)$ in $M$ is $\Lambda$ by minimality, so $A$ must accumulate on $\pi(p)$. But this means that $A$ intersects $u(\pi(p))$, which implies that $A$ contains $\pi(p)$ as explained above. This is a contradiction.
\end{proof}

\subsubsection{A topological lemma}
Let $L$ be a metrically complete, non compact, simply connected, Riemannian surface.

For a compact subset $X \subset L$ we denote $\mathrm{Fill}(X)$ to be the complement of the unique unbounded connected component of $L \setminus X$. Note that $\mathrm{Fill}(X)$ is always compact as a neighborhood of $\infty$ in the compactification of $L$ is disjoint from $X$. Notice further that, by definition, $\mathrm{Fill}(X)$ is a compact connected set which does not separate the plane. 

We will use the following simple properties of $\mathrm{Fill}(X)$:

\begin{itemize}
\item If $X\subset Y$ are compact sets then $\mathrm{Fill}(X) \subset \mathrm{Fill}(Y)$.
\item If $g: L \to L$ is a homeomorphism and $X \subset L$, then $g(\mathrm{Fill}(X)) = \mathrm{Fill}(g(X))$. 
\end{itemize}

The following lemma will be used in the next section (see Figure~\ref{f.6M}).

\begin{lemma}\label{topologicallema}
Let $L$ be as above, then for every $b>0$ and $\delta>0$ there exists $R>0$ and $n_0>0$ with the following property. Let $A,B$ be topological disks, and let $\ell_1, \ldots, \ell_n$ be disjoint curves, with $n\geq n_0$ that join $A$ and $B$. Suppose that
\begin{enumerate}[label=(\roman*)]
\item $d(A,B)> 2R$, and
\item the $\delta$-neighborhoods of the curves $\ell_i$ are pairwise disjoint.
\end{enumerate}
Then the region $\mathrm{Fill}(A \cup B \cup \ell_1\cup \ldots \cup \ell_n) \setminus (A\cup B)$ contains a disk $D$ of radius $> 4b$.
Moreover, $D$ can be chosen so that $d(D,A)$ and $d(D,B)$ are larger than $\frac{d(A,B)}{10}$. 
\end{lemma}

 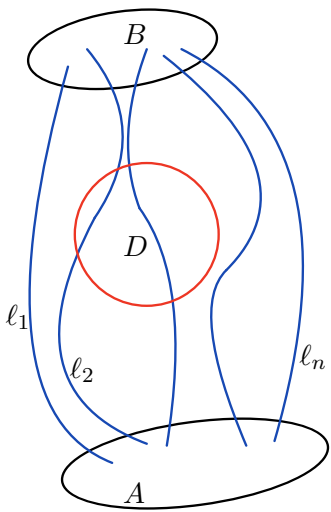
\begin{figure}[ht]
 \begin{center}
 \includegraphics[scale=0.68]{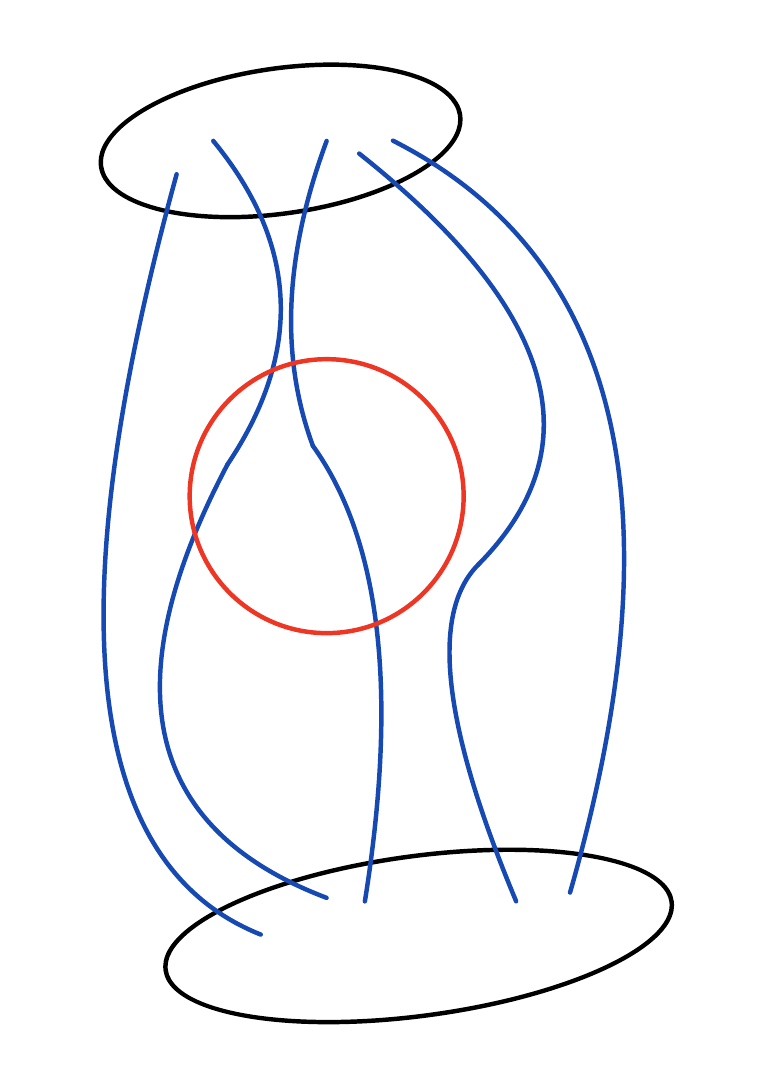}
 \begin{picture}(0,0)
\put(-100,12){$A$}
\put(-100,185){$B$}
\put(-100,105){$D$}
\put(-144,79){$\ell_1$}
\put(-120,59){$\ell_2$}
\put(-34,64){$\ell_n$}
 \end{picture}
 \end{center}
 \vspace{-0.5cm}
 \caption{A depiction of Lemma \ref{topologicallema}.}\label{f.6M}
 \end{figure}

\begin{proof}
Using the Jordan Curve Theorem we can reorder the curves so that:

\begin{itemize}
\item $\mathrm{Fill}(A \cup B \cup \ell_1\cup \ldots \cup \ell_n) = \mathrm{Fill}(A \cup B \cup \ell_1\cup \ell_n)$, 
\item for $1<i<n$ we have that $\ell_i \subset \mathrm{Fill}(A \cup B \cup \ell_{i-1} \cup \ell_{i+1})$. 
\end{itemize}

Take $R > 100 b$ and $n_0 > \frac{100 b}{\delta}$. Without loss of generality we can assume that $n$ is even. This way, we can choose a point $x \in \ell_{n/2}$ such that $d(x,A)>\frac{d(A,B)}{4}$ and $d(x,B)> \frac{d(A,B)}{4}$. We claim that $B(x,4b) \subset \mathrm{Fill}(A \cup B \cup \ell_1\cup \ell_n)$. By our choice of $x$ it will follow that $B(x,4b)$ is at distance larger than $\frac{d(A,B)}{10}$ from $A$ and $B$. 

To see this, consider $r$ a geodesic ray starting from $x$ and $y$ the first point of intersection of $r$ with $\partial \mathrm{Fill}(A \cup B \cup \ell_1\cup \ldots \cup \ell_n) \setminus (A\cup B)$.  By our ordering, there are two possibilities: 

\begin{itemize}
\item either $y$ belongs to $\partial A \cup \partial B$,
\item or $y$ belongs to $\ell_1 \cup \ell_n$. 
\end{itemize}

By our assumptions, if $y \in \partial A \cup \partial B$ then the distance $d(x,y)> R/4 > 4b$. On the other hand, if $y \in \ell_1$ then by our choice of reordering we deduce that $r$ must intersect $\ell_i$ for all $1\leq i \leq n/2$. Since the points of intersection are at distance pairwise larger than $\delta$ we deduce that $d(x,y)>4b$. Similarly, if $y \in \ell_n$ we also get $d(x,y)>4b$. This completes the proof. 
\end{proof}

\subsubsection{Proof of Proposition \ref{l.analogL53}} 
We will use the fact that $\wt f^{-1}$ expands stable length repeatedly. To simplify notation we set $g:= \wt f^{-1}$. The rest of this subsection is devoted to the proof of Proposition \ref{l.analogL53}.

According\footnote{It is not hard to see that the proof applies to the fixed sublamination.} to Lemma \ref{lema-boundedinfixcsleaf-nonDC} there is a constant $K>0$ such that, for any $z\in L$, we have

\[ d_L(z,g(z)) \leq K, \]

\noindent where $d_L$ denotes the path-metric on $L$. From now on, during this subsection we will always work in $L$ so we will drop the subscript and write $d:=d_L$. 

To finish the proof, our aim will be to show that the fact that $\ft$ moves points a bounded distance in $L$, together with the exponential contraction of length along the stable leaf $s(p)$ under iteration by $\ft$ will force an arbitrarily large amount of ``bunching'' of $s(p)$, which is impossible for leaves of planar foliations.

Indeed, since $s(p)$ is a leaf of a foliation of the plane, there exists some constants $\delta, \eta>0$ such that if $I, J \subset s(p)$ are closed segments which are at distance larger than $\eta$ in the $s(p)$ metric, then, their $\delta$-neighborhoods are disjoint in $L$. 
Now, this implies in particular that the volume of the $\delta$-neighborhood of a segment of $s(p)$ must grow to infinity with its length (and therefore, the diameter grows to infinity with the length).

Without loss of generality, we can assume that $\delta, \eta<1$ and $K>1$. 

To prove Proposition \ref{l.analogL53} we will fix $b>0$ as given by Lemma \ref{claim.fixed_points_close_by} and $\delta>0$ by the considerations above. Lemma \ref{topologicallema} then gives us values of $R>0$ and $n_0>0$ associated to $b$ and $\delta$ so that its statement holds. We will fix

$$ n > \max \left\{ \frac{10 R}{K} , \frac{10 b}{\delta}, n_0 \right\}. $$

\smallskip

We introduce the following notation: given any $a,b\in s(p)$, we write $[a,b]^s$ to indicate the closed segment along the stable leaf $s(p)$ between $a$ and $b$ oriented from $a$ to $b$. 

\smallskip

We will then pick points in $y, z \in s(p)$ with the following properties: 

\begin{itemize}
\item $d(y,z) > 200Kn $
\item $g([y,z]^s) \cap [y,z]^s = \emptyset$ (equivalently, $z \in [y,g(y)]^s$). 
\end{itemize}

The existence of points like this follows from the fact that if $y_0$ is any point in $s(p)$ then the length of $g^k([y_0,g(y_0)]^s)$ grows to infinity, and thus its diameter grows too. See Figure~\ref{f.7M}.

 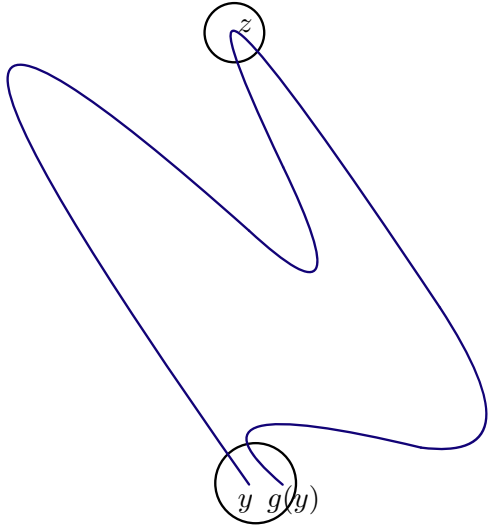
\begin{figure}[ht]
 \begin{center}
 \includegraphics[scale=0.68]{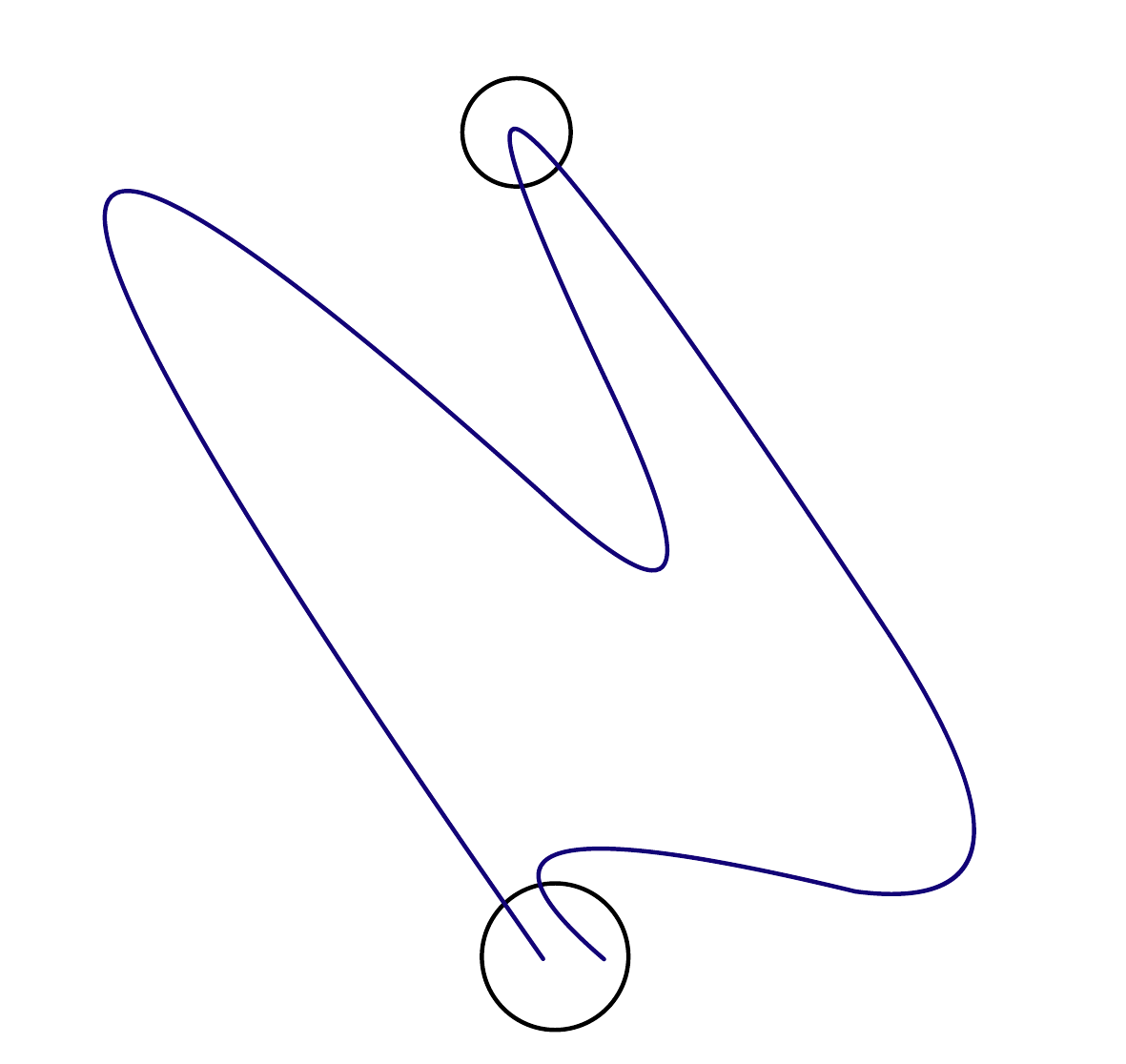}
 \begin{picture}(0,0)
\put(-128,13){$y$}
\put(-128,192){$z$}
\put(-117,13){$g(y)$}
 \end{picture}
 \end{center}
 \vspace{-0.5cm}
 \caption{Choosing points $y$ and $z$ in $s(p)$.}\label{f.7M}
 \end{figure}

We will pick $A_i = B(y,Ki)$ and $B_i= B(z,Ki)$ the neighborhoods of radius $Ki$ of the points $y,z$. Given our choices, notice that $g(y)\in A_1$, $g(z)\in B_1$, and, for any $i$, $g(A_{i}) \subset A_{i+1}$ as well as $g(B_{i}) \subset B_{i+1}$. 

The following holds: 

\begin{lemma}\label{l.fundamentaldomaincurve}
Every arc $J \subset [y, g^n(z)]^s$ which is disjoint from $A_n \cup B_n$ is completely contained in a fundamental domain of $s(p)$ for the action of $g$. More precisely, there exists $\ell$ such that $J \subset [g^\ell(y),g^\ell(z)]^s$ or $J \subset [g^\ell(z),g^{\ell+1}(y)]^s$. 
\end{lemma}
\begin{proof}
This is because $[y,z]^s$ intersects $A_1$ and $B_1$ so every fundamental domain as above intersects both $A_n$ and $B_n$. 
\end{proof}

We can apply Lemma \ref{topologicallema} to get:

\begin{lemma}\label{l.bigdiskinfill} 
We have that $\mathrm{Fill}(A_n \cup B_n \cup [y,g^n(z)]^s) \setminus (A_{10n} \cup B_{10n})$ contains a disk of radius $4b$. 
\end{lemma} 
\begin{proof}
Note that $[g^{\ell}(y),g^{\ell+1}(y)]^s$ contains at least two segments joining $A_{10n}$ to $B_{10n}$ if $0\leq \ell <n$ (see Figure~\ref{f.7M}). Thus there are at least $2n$ such curves, which since they are segments separated in $s(p)$ must have pairwise disjoint $\delta$-neighborhoods. Thus, by our choice of constants $b, \delta, K$ and $n$ above, we can apply Lemma \ref{topologicallema} to find a disk of radius $\geq 4b$ inside $\mathrm{Fill}(A_{n} \cup B_{n} \cup [y,g^{n}(z)]) \setminus (A_{n} \cup B_{n})$ which is at distance larger than $\frac{d(A_n,B_n)}{10}$ from $A_n$ and $B_n$. Thus, the disk is contained in $\mathrm{Fill}(A_n \cup B_n \cup [y,g^n(z)]^s) \setminus (A_{10n} \cup B_{10n})$ as required. 
\end{proof}

Using Lemma \ref{claim.fixed_points_close_by} we can find a fixed point $q \in \mathrm{Fix}(g)$ such that $B(q,2b) \subset \mathrm{Fill}(A_n \cup B_n \cup [y,g^n(z)]^s) \setminus (A_{10n} \cup B_{10n})$. 

We can show the following:

\begin{lemma}\label{l.jordanregion} 
There exists an arc $J \subset [y,g^n(z)]^s$ such that one of the following holds:
\begin{enumerate}
\item either $J$ intersects $A_n$ only at its endpoints and $q \in \mathrm{Fill}(A_n \cup J)$, or,
\item $J$ intersects $B_n$ only at its endpoints and $q \in \mathrm{Fill}(B_n \cup J)$. 
 \end{enumerate}
 Moreover, $J$ is contained in a fundamental domain: for some $0 \leq \ell \leq n$ we either have $J \subset [g^\ell(y), g^{\ell+1}(y)]^s$ or $J \subset [g^{\ell}(z), g^{\ell+1}(z)]^s$.
\end{lemma}

\begin{proof}
This follows from the fact that $\mathrm{Fill}(A_n \cup B_n \cup [y,g^n(z)]^s)$ is contained in a union of sets of this form. 

To see this note that $\mathrm{Fill}(A_n \cup B_n \cup [y,g^n(z)]^s) = \mathrm{Fill}(A_n \cup [y,g^n(z)]) \cup \mathrm{Fill}(B_n \cup [y, g^n(z)]^s)$ because $A_n$ and $B_n$ are disjoint topological disks and $[y,g^n(z)]^s$ is a topological interval. Indeed, by Jordan's theorem $\hat A=\mathrm{Fill}(A_n \cup [y,g^n(z)]^s)$ is a topological disk with an arc attached (i.e. the segment $[w,g^n(z)]^s$ where $w$ is the last point of intersection of $[y,g^n(z]]^s$ and similarly $\hat B=\mathrm{Fill}(B_n \cup [y,g^n(z)]^s)$ is a topological disk with an arc attached. One has that $\mathrm{Fill}(A_n\cup B_n \cup [y,g^n(z)]^s)= \mathrm{Fill}(\hat A \cup \hat B)$. Since the intersection of these sets is connected (because their intersection retracts to $[y,g^n(z)]^s$) we deduce\footnote{Here we are using the fact from plane topology that generalizes Jordan's curve theorem stating that if $X,Y$ are compact connected sets, then their union separates the plane if and only if their intersection is not connected.} that $\mathrm{Fill}(\hat A \cup \hat B)= \hat A \cup \hat B$. 

The fact that $J$ is contained in a fundamental domain is direct consequence of the fact that it intersects $A_n$ (or $B_n$) only in its boundaries, and thus Lemma \ref{l.fundamentaldomaincurve} can be applied. \end{proof}

Both cases are analogous, so we will assume from now on that the first option happens, namely, $q \in \mathrm{Fill}(A_n \cup J)$ for a curve $J \subset [y, g^n(z)]^s$ which intersects $A_n$ only at its endpoints and such that $J$ is contained in a fundamental domain of $s(p)$. 

To reach a contradiction, the idea will be to find fixed points $q_1,q_2$ which are sufficiently close, and such that one belongs to $\mathrm{Fill}(A_n \cup J)$ and the other does not. If we choose them appropriately, we will be able to see that $g^i(J)$ will intersect a geodesic joining $q_1$ and $q_2$ for several values of $i$ (before the set $g^i(A_n)$ becomes too big). This will produce some accumulation of the arcs $g^{i}(J)$ (which are segments of $s(p)$ far along $s(p)$), this is not possible and gives the desired contradiction. 

\begin{lemma}\label{l.fixedpointoutside}
There are fixed points $q_1, q_2\in \mathrm{Fix}(g)$ such that $d(q_1,q_2)< 3b$ and we have that $q_1 \in \mathrm{Fill}(A_n \cup J) \setminus A_{10n} $ while $q_2 \notin \mathrm{Fill}(A_n \cup J)$. 
\end{lemma}

\begin{proof}
We will use Lemma \ref{claim.fixed_points_close_by}. By the choice of the point $q$ we can consider an unbounded geodesic ray $r$ starting at $q$ which is at distance larger than $2b$ from $A_{10n}$. One can cover $r$ by balls of radius $b$, in each such ball there is a fixed point, and eventually, the fixed point is not in $\mathrm{Fill}(A_n \cup J)$ which is a bounded set. So, there is a pair of such points such that one belongs to $\mathrm{Fill}(A_n \cup J)$ and the other does not. Their distance is less than $3b$. 
\end{proof}

We are now ready to prove Proposition \ref{l.analogL53} by finding a contradiction, that will be produced using the following:

\begin{lemma}
For every $0\leq i \leq n$ we have that $g^i(J) \cap [q_1,q_2]_L \neq \emptyset$ where $[q_1,q_2]_L$ denotes a geodesic segment joining $q_1$ and $q_2$. 
\end{lemma}
\begin{proof}
Note first that since $d(q_1, q_2) < 3b$ and $q_1 \notin A_{10n}$ we know that the geodesic segment $[q_1,q_2]_L$ is disjoint from $A_{5n}$ (recall that $\delta<1$ and that $n > \frac{10b}{\delta}$). 

Since $q_1 \in \mathrm{Fill}(A_n\cup J)$ is fixed we get that $q_1 = g^i(q_1) \in g^i(\mathrm{Fill}(A_n \cup J)) = \mathrm{Fill}(g^i(A_n)\cup g^i(J))$. Similarly, we get that since $q_2 \notin \mathrm{Fill}(A_n\cup J)$ we have that $q_2 \notin \mathrm{Fill}(g^i(A_n)\cup g^i(J))$. 

This implies that $\partial \mathrm{Fill}(g^i(A_n) \cup g^i(J))$ must intersect $[q_1,q_2]_L$. Since $g^i(A_n) \subset A_{n+i}$ which is disjoint from $[q_1, q_2]_L$ we deduce that $g^{i}(J)$ must intersect $[q_1,q_2]_L$ as we wanted to show. 
\end{proof}

The contradiction is now the fact that $g^i(J)$ are curves whose $\delta$-neighborhoods are disjoint and all intersect $[q_1,q_2]_L$ which is a geodesic segment of length $<3b$. This produces $n$ different points at pairwise distance $\geq \delta$ in $[q_1,q_2]_L$ which is a contradiction since $n > \frac{10b}{\delta}$.

\subsection{Proof of Proposition \ref{p.hypSeifminimal_nonDC}}

We are now ready to prove Proposition \ref{p.hypSeifminimal_nonDC}.

This proof follows the same structure as the one of \cite[Proposition 3.15]{BFFP-prequel} 
and we will continuously refer to it. Recall the standing assumption that $f$ is orientable. 

Consider $\Lambda$ an $f$-minimal non empty subset. We need to show that $\Lambda = M$. We assume by contradiction that $\Lambda \neq M$.

Since $\fbs$ has no closed leaves and $\Lambda$ is $f$-minimal, there cannot be any isolated leaves in $\Lambda$ (for the topology of the stable leaf space).

Now, Proposition \ref{l.analogL53}  allows us to assert that $\ft$ has no fixed points in leaves of $\wt \Lambda$. Then, Corollary \ref{c.L53} implies that each leaf of $\Lambda$ is either a plane or an annulus.

\vskip .1in
We fix an $\epsilon$ small enough and let $\Lambda'$ be the pull back of $\Lambda$ to
the approximating foliation $\fes$. That is, $\Lambda'
= (h^{cs}_{\epsilon})^{-1}(\Lambda)$.
Let $V$ be a connected component of $\mt \setminus
\widetilde  \Lambda'$.

\begin{claim}
	The projection $\pi(V)$ to $M$ has finitely many boundary leaves.
\end{claim}
This is a standard fact in the theory of foliations \cite[Lemma 5.2.5]{CandelConlonI}.

\begin{claim}\label{claim_piLannulus_branching}
Each leaf $L \subset \partial V$ projects to an annulus $\pi(L)$ in $M$.
\end{claim}
\begin{proof}
Suppose that $\pi(L)$ is a plane. Recall (see \cite[Lemma 5.2.14]{CandelConlonI}) that $\pi(V)$ has an octopus decomposition and a compact core. So for any $\delta >0$, the subset of points in $\pi(L)$ that are at distance greater then $\delta$ from another boundary component of $\pi(V)$ is precompact. Since $\pi(L)$ is supposed to be a plane, that subset must be contained in a closed disk $D$. Then $\pi(L)\smallsetminus D$ is an annulus that is $\delta$-close to another boundary component, $\pi(L')$ of $\pi(V)$. Moreover, the subset of $\pi(L')$ that is $\delta$-close to $\pi(L)\smallsetminus D$ then also has to be an annulus. If $\pi_1(L')$ were not a plane it would be an annulus and its non-trivial curve corresponds to a curve homotopic to the boundary of the closed disk $D$ which is homotopically trivial in $M$. Since the leaves of $\fes$ are $\pi_1$-injective, this 
implies that $\pi(L')$ is also a plane.

Since $M$ is irreducible this implies that $\pi(V)$ is homeomorphic to an open  disk times
an interval. So $\pi(V)$ has only two boundary components, both of which are planes. In particular, the isotropy group of $V$ is trivial and $\pi(V)$ is homeomorphic to $V$.

We will now switch to the branching foliation to finish the proof. Let $A = \hs\left(\pi(L)\right)$ and $B= \hs\left(\pi(L')\right)$.
Since we chose $\epsilon$ small enough, up to taking $\delta$ small enough also, the unstable segments through $A\smallsetminus \hs(D)$ intersect $B$, and their length is uniformly bounded. Moreover, no unstable ray of $A$ can stay in $\hs(\pi(V))$. This is because $\pi(V)$ is 
homeomorphic to an open disk times an interval.
 So, since $D$ is compact, the length of every unstable segment between $A$ and $B$ is bounded by a uniform constant. Notice that, since $\fbs$ is a branching foliation, we may have $A\cap B \neq \emptyset$, i.e., some of these unstable segments may be points.

Since $L$ and $L'$ are in $\partial V$, which is a connected component of $\mt \smallsetminus\widetilde\Lambda'$, we have that $A,B \in \partial \left(M \smallsetminus \Lambda\right)$. So in particular, $A$ and $B$ are fixed by $f$. Hence, the set of unstable segments between $A$ and $B$ is also invariant by $f$. Since the length of unstable segments between $A$ and $B$ are bounded above and $f$ expands the unstable length, all the unstable segments must have zero length. i.e., $A=B$. Which implies that $V$ is empty, which contradicts the assumption that $\Lambda \neq M$.
\end{proof}

Thus we showed that every component of $\pi(\partial V)$ is an annulus.
We can then apply without change the (topological) arguments of the proof of
\cite[Proposition 3.15]{BFFP-prequel} to obtain a torus $T$, composed of annuli along leaves of $\fes$, together with annuli transverse to $\fes$, that bounds a solid torus $U'$ in $\pi(V)$.

Now consider $U = \hs(U')$. Because of the collapsing of leaves, $U$ may not be a solid torus. If $U$ is empty for any any such component $U'$,
this would directly contradict the assumption $\Lambda\neq M$.
So for some such complementary component $U'$, the set $U$  is not
empty and it is contained in a solid torus (the $\epsilon$-tubular neighborhood of $U'$ in $M$). We can then use the same ``volume vs.~length'' argument on $U$, exactly as in the end of the proof of \cite[Proposition 3.15]{BFFP-prequel}, to get a final contradiction. This ends the proof of Proposition \ref{p.hypSeifminimal_nonDC}.

\subsection{Some consequences}

An important consequence of Proposition \ref{l.analogL53} is the following: 

\begin{corollary}\label{c.L53}
Suppose that $f$ is a partially hyperbolic 
diffeomorphism in $M$ that is homotopic to the
identity. Let $\ft$ be a good lift of $f$ to $\mt$.
Suppose that $\Lambda$ is a non empty (saturated) $f$-minimal subset 
of $\fbs$ such that
every leaf of the lift $\widetilde \Lambda$ to $\mt$ is fixed by $\ft$.
Then every leaf in the $f$-minimal set $\Lambda$ of $\fbs$, 
is either a plane or an annulus.
\end{corollary}

\begin{proof}
Let $A$ be a leaf of $\Lambda$ and $L$ a lift in $\mt$. By Proposition \ref{l.analogL53}, $L$ does not admit any fixed points of $\ft$. Hence, $\ft$ acts freely on the space of stable leaves in $L$.

Now, recall that $\pi_1(A)$ can be defined as the elements $\gamma \in \pi_1(M)$ that fix $L$ (see section \ref{ss.fundamental_group_branching}). So if $\gamma \in \pi_1(A)$, it must also act freely on the space of stable leaves in $L$. As $\ft$ commutes with every deck transformation, Corollary E.4 of \cite{BFFP-prequel} (which still applies in the context of branching foliation, as does all of \cite[Appendix E]{BFFP-prequel}) implies that $\pi(A)$ is abelian, i.e., $A$ is either a plane or an annulus (again with the understanding that $A$ might actually only be an immersion of one of these manifolds in $M$ and recalling that all bundles were assumed to be orientable in this section, so in particular
the leaves cannot be M\"{o}bius bands).
\end{proof}

As a consequence, we also get the following result that completes the proof of Theorem \ref{thm-nocontractible} as announced. 

\begin{corollary}
Suppose that $f$ is a partially hyperbolic diffeomorphism
homotopic to the identity. Suppose that $f$ is either volume preserving or transitive,
or that $M$ is either hyperbolic or Seifert. Let $\ft$ be a good lift of $f$.
Then $\ft$ has no periodic points. In particular, $f$ has no contractible periodic points.
\end{corollary}

\begin{proof}
Up to finite covers and iterates, we may assume that $f$ 
preserves the branching foliations $\fbs, \fbu$. 

If $\ft$ acts as a translation on either $\fbs$ or $\fbu$, then it does not have periodic points.

Otherwise, since we showed that under our assumptions the branching foliations are $f$-minimal. The result then follows from Theorem \ref{t.good_lifts_have_no_fixed_points}.
\end{proof}

\section{Double invariance implies dynamical coherence}\label{sec-doubleinvndc}


In this section we show that if the center-stable and center-unstable branching foliations are minimal and leafwise fixed by a good lift $\ft \colon \mt \to \mt$, then, $f$ has to be dynamically coherent (i.e., the branching foliations do not branch). Therefore, we will be able to apply the results from the dynamically coherent setting. 

The universal cover $\mt$ of $M$ is homeomorphic to $\mathbb{R}^3$ (since it admits a partially hyperbolic diffeomorphism, see \cite[Appendix B]{BFFP-prequel}).
We do not assume anything further on $M$ in this section.

Recall also that a center leaf is a connected component of the intersection of a leaf of $\wfbs$ and one of $\wfbu$ (cf. Definition \ref{d.center_leaf}).

This section (and the proof of dynamical coherence) is split in three
parts. First, in subsection \ref{ss.center_leaves_are_fixed}, we show that, for an appropriate lift of $M$ and power of $f$, double invariance of the foliations implies that the center leaves are fixed. The lift and power we need to consider here is in order to have everything orientable and coorientable. Then, in section \ref{ss.dcifndc}, we show that if a good lift fixes every center leaf, then it must be dynamically coherent. Finally, in section \ref{dcwithout}, we show that if a lift and power of a partially hyperbolic diffeomorphism is dynamically coherent and fixes the center leaves, then the original diffeomorphism is itself dynamically coherent (and a good lift of a power of it will fix every center leaf).

\subsection{Center leaves are all fixed} \label{ss.center_leaves_are_fixed}
To begin, we would like to show that $\ft$ fixes every center leaf. The results of \S\ref{sec-centerdynamics} already provide at least one fixed center leaf:
\begin{lemma}
	Let $f \colon M \to M$ be an orientable partially hyperbolic diffeomorphism homotopic to the identity with $f$-minimal branching foliations $\fbs,\fbu$.
	If there is a good lift $\ft$ that fixes every leaf of $\wfbs$ and $\wfbu$, then $\ft$ fixes some center leaf.
\end{lemma}
\begin{proof}
	Suppose that $\ft$ fixes no center leaf. Since there is at least one non-planar leaf, Proposition \ref{periodiccenter} provides an $f$-periodic center leaf $c$ in $M$. Applying Proposition \ref{p.alternnonDC} to $\widetilde \cW^{cs}_{bran}$ shows that $c$ is coarsely contracting, but the same result applied to $\widetilde \cW^{cu}_{bran}$ shows that $c$ is coarsely expanding. This is a contradiction, so $\ft$ must fix a center leaf as desired.
\end{proof}

\begin{proposition}\label{p.nondceverycfixed} 
	Let $f \colon M \to M$ be an orientable partially hyperbolic diffeomorphism homotopic to the identity with $f$-minimal branching foliations $\fbs,\fbu$. If a good lift $\ft$ of $f$ fixes every leaf of $\wfbs$ and $\wfbu$, then $\ft$ fixes every center leaf. 
\end{proposition} 
\begin{proof}
	Let 
	\[ \mathrm{Fix}^c_{\ft} := \{ c \ : \ \ft(c)=c \}, \]
	thought of as a subset of the center leaf space.
	
	The set $\mathrm{Fix}^c_{\ft}$ is obviously $\pi_1(M)$-invariant. It is also open, by an argument very similar to the one in \cite[Lemma 6.3]{BFFP-prequel}: If $c$ is a fixed center leaf in a center stable leaf
	$L$ in $\mt$, then for any center leaf $c'$ in $L$ close enough to $c$ (for the topology of the center leaf space in $L$), there is a strong stable leaf that intersects $c$, $c'$ and $\ft(c')$. Now, since $\ft$ fixes the center unstable leaves, $c'$ and $\ft(c')$ are on the same center unstable leaf. Since no transversal can intersect the same leaf twice, it implies that $c' = \ft(c')$. Thus the set of fixed center leaves within each center stable leaf is open (in the center leaf space within that center stable leaf). Similarly, the set of fixed center leaves within each center unstable leaf is open. Together, these facts imply that the set of fixed center leaves is open in the center leaf space.

	Note that since a good lift $\ft$ fixes every leaf of $\wfbs$, then $f$ fixes every leaf of $\fbs$. In particular $f$-minimality of $\fbs$ is equivalent to minimality of $\fbs$. Hence $\fbs$ 
	is minimal. Similarly for $\fbu$.\footnote{Note that $f$-minimality and minimality are in fact always equivalent as long as the branching foliation does not have a compact leaf and without assumptions on $f$, see Lemma \ref{fmin}.}

	To see that $\ft$ fixes every center leaf, we proceed as in \cite[Lemma 6.4]{BFFP-prequel}:
	We show first that every center leaf in a center stable leaf (resp.~center unstable leaf) which projects to an annulus has to be fixed (due to our orientability assumptions, leaves cannot project to a M\"{o}bius band). Then the same argument as in 
	\cite[Lemma 6.4]{BFFP-prequel} applies to show that every center leaf has to be fixed.

	Let $L$ be any center stable leaf that projects to an annulus, and choose a generator $\gamma$ of the isotropy group of $L$.
	
	Since the set of fixed center leaves is open in the center leaf spaces of any center unstable leaf, minimality of $\fbs$ implies that $L$ must have some fixed center leaves.

	We will first prove that if $c$ is a center leaf in $L$ which is in the boundary of the set of fixed center leaves in $L$, then $\pi(c)$ is periodic under $f$. We will then show, as in Proposition \ref{p.graphtransf2}, that any periodic leaf in 
	$\pi(L)$ must be coarsely contracting. The same argument applied to the center-unstable leaves yields that periodic center leaves must also be coarsely expanding, a contradiction.

	Since $\ft$ cannot have fixed points (as $\ft$ fixes all the leaves of $\wfbs$ and $\wfbu$), then $\ft$ acts freely on the space of stable leaves in $L$.
	
	We assume, for a contradiction, that not all center leaves in $L$ 
	are fixed. Let $\mathrm{Fix}_L$ be the set (in, $\cL^c_L$, the center leaf space on $L$) of center leaves fixed by $\ft$. 
	
	The set $\mathrm{Fix}_L$ is open, and assumed not to be the whole of $L$. So let $c_1$ be any leaf in $\partial \mathrm{Fix}_L$. 
	
	Let $(c_n)$ be  any sequence of center leaves in $\mathrm{Fix}_L$ that converge to $c_1$. Then $\ft(c_n) = c_n$ converges to $\ft(c_1)$. As the leaf $c_1$ is not fixed by $\ft$, we deduce that $\ft(c_1)$ is non-separated from $c_1$.
	
	Hence, there exists a (unique) stable leaf $s_1$, which separates $\ft(c_1)$ from $c_1$ and makes a perfect fit with $c_1$ (see section \ref{ss.perfectfits} for the definition of perfect fits in the non dynamically coherent setting). Then $\ft(s_1)$ makes a perfect fit with $\ft(c_1)$. 
	Because $c_1$ and $\ft(c_1)$ are non separated from each other,
	$s_1$ and $\ft(s_1)$ intersect a common transversal to the
	stable foliation.
	It follows that the stable axis of $\ft$ acting on $L$ is a line. Thus, since $\gamma$ commutes with $\ft$, the stable axis of $\gamma$ is that same line. Moreover, both the stable leaves $s_1$ and $\ft(s_1)$ are in the axis of $\ft$.
	
	Since the stable axis of $\ft$ acting on $L$ is a line, the Graph Transform argument \cite[Appendix H]{BFFP-prequel} applies and we obtain a curve $\hat\eta$, tangent to the center direction, that is fixed by both $\gamma$ and $\ft$.
	
	As $s_1$ makes a perfect fit with $c_1$ and $s_1$ intersects $\hat\eta$, we deduce that there exists a stable leaf $s$ that intersects both $c_1$ and $\hat\eta$. Let $x = s \cap \hat\eta$ and $y = s\cap c_1$. We denote by $J$ the segment of $s$ between $x$ and $y$.
	
	Since $\hat\eta$ projects down to a closed curve $\pi(\hat\eta)$, and $\ft$ decreases stable lengths, there exist $n_1,n_2 \in \mathbb{Z}$ and $m_1,m_2\in \mathbb{N}$ as large as we want such that the four points $\gamma^{n_1} \ft^{m_1} (x)$, $\gamma^{n_1} \ft^{m_1} (y)$, $\gamma^{n_2} \ft^{m_2} (x)$ and $\gamma^{n_2} \ft^{m_2} (y)$ are all in a disk of radius as small as we want.

	Suppose now that $\gamma^{n_1} \ft^{m_1}(c_1) \neq \gamma^{n_2} \ft^{m_2} (c_1)$. Then, up to switching $n_1, m_1$ and $n_2,m_2$, we obtain that $\gamma^{n_2} \ft^{m_2} (c_1)$ intersects $\gamma^{n_1} \ft^{m_1} (J)$. This is in contradiction with the fact that $c_1$ is in $\partial \mathrm{Fix}_L$ which is invariant by both $\ft$ and $\gamma$.
	
	Thus $\gamma^{n_1} \ft^{m_1}(c_1) = \gamma^{n_2} \ft^{m_2} (c_1)$. In other words, $c_1$ is fixed by the map $h= \gamma^n \ft^m$ for some $n,m$ integers, $m>0$.
	(Although not useful for the rest of the proof, one can further notice that $\hat\eta$ and $c_1$ intersect, as $h$ decreases the length of $J$ by forward iterations and both $c_1$ and $\hat\eta$ are fixed by $h$.)

	Now recall that we built above a stable leaf $s_1$ making a perfect fit with $c_1$. And, by our choice of $s_1$, the center leaf $c_1$ is in between $s_1$ and $s_2:=\ft^{-1}(s_1)$.
	
	Recall that $s_1$ is the unique leaf making a perfect fit with $c_1$ and separating $c_1$ from $\ft(c_1)$. Thus $h(s_1)$ is the unique leaf making a perfect fit with $h(c_1)= c_1$ and separating $h(c_1)=c_1$ from $h\circ \ft(c_1)= \ft\circ h(c_1) = \ft(c_1)$. That is, $s_1$ is fixed by $h$. Using again that $h$ and $\ft$ commutes, we deduce that $s_2$ is also fixed by $h$.
	
	Now, the leaves $s_1$ and $s_2$ are also a bounded distance apart, so Lemma \ref{l.condition_for_coarse_contraction} holds and we deduce that $c_1$, as well as any other center leaf $c$ that is in between $s_1$ and $s_2$ must be coarsely contracting.
	Note now that any center leaf $c$ in $L$ that is fixed by some $h'=\gamma^{n'}\ft^{m'}$ is separated from $\mathrm{Fix}_L$ by a center leaf $c_1' \subset \partial \mathrm{Fix}_L$ as above. Hence, we proved that every non-fixed periodic leaf in $\pi(L)$ is coarsely contracting.

	Therefore, the same argument applied to the center \emph{unstable} leaf containing $c_1$ shows that $c_1$ must also be coarsely expanding, a contradiction.

	So we obtained that every center stable or center unstable leaf $L$ which is fixed by some non trivial element of $\pi_1(M)$ has all of its center leaves fixed by $\ft$. Since $\mathrm{Fix}^c_{\ft}$ is open (in the center leaf space), minimality of the foliations implies that it contains every center leaf, as in the end of the proof of \cite[Lemma 6.4]{BFFP-prequel}. 
\end{proof}

\subsection{Dynamical coherence}\label{ss.dcifndc} 

We now want to prove dynamical coherence provided that a good lift fixes every center leaf. We do not assume that $f$ is orientable, only that it admits branching foliations.
We start with the following:
\begin{lemma}\label{l.boundedmovement}
	Let $f \colon M \to M$ be a partially hyperbolic diffeomorphism homotopic to the identity
preserving 
 branching foliations $\fbs,\fbu$. Let $\ft$ be a good lift that fixes every center leaf. Then there is a global bound on the length of every center segment between a point $x$ and $\ft(x)$. 
\end{lemma}

In the dynamically coherent case this was very easy as
the center curves form an actual foliation and there is a
local product picture near any compact segment. We have to
be more careful in the non dynamically coherent setting.

\begin{proof}
We assume the conclusion of the lemma fails.
Then there exists a sequence $x_i$ of points in $\mt$ contained in center leaves $c_i$ such that the length in $c_i$ from $x_i$ to
$\ft(x_i)$ divverges to infinity. Notice that this length
depends not only on $x_i$ but also on $c_i$ since there may be many center leaves through $x_i$.
We denote by $e_i$ the segment in $c_i$ from $x_i$ to $\ft(x_i)$.

Up to acting by covering translations we can assume that the $x_i$ converge to a point $x\in \mt$.
Let $L_i$ and $U_i$ be respectively a center stable and center unstable leaves containing $c_i$. 
Up to considering a subsequence, we may assume
that $L_i$ converges to a center stable leaf $L$ containing $x$ (see condition \ref{item.convergence_of_leaves} of Definition \ref{def.branching}). Similarly, we can further assume that $U_i$ converges to some center unstable leaf $U$, with $x\in U$.

For $i$ large enough, all the leaves $L_i$ intersect a small unstable segment
in $u(x)$. The set of center stable leaves intersecting this
segment is a also a segment (even though many different leaves
may intersect a given point in $u(x)$). Hence we may assume
that $L_i$ is weakly monotone, and so is $U_i$.
Let $c$ be the center leaf through $x$ contained in $L \cap U$. Then $\ft(x) \in c$, and we call $e$ the segment in $c$ from $x$ to $\ft(x)$.

Suppose first that $L_i = L$ for all big $i$. So we may
assume $L_i = L$ for all $i$.
Then the center leaves $c_i$ are all in $L$ and, for $i$ big enough, intersect $s(x)$. Hence the leaves $c_i$ are, for $i$ big enough, contained in an interval of the center leaf space in $L$. In addition
they are converging to $c$ which is a center leaf through
$x$ and $\ft(x)$. This implies that the length of 
$e_i$ is converging to the length of $e$ and hence the 
length of $e_i$ is
bounded in $i$. Contradiction.

Suppose now that the $L_i$ are all distinct from $L$.
Notice that the points $x_i$, and $\ft(x_i)$ are all in a compact region
of $\mt$. Since $L_i$ converges to $L$, we have that
$u(x_i)$ intersects $L$ for big enough $i$. We call this nearby intersection $y_i$.
Likewise $u(\ft(x_i))$ intersects $L$ in $\ft(y_i)$.
We want to push the center segments $c_i$ contained
in $U_i \cap L_i$ along unstable segments to center
segments in $U_i \cap L$.

For $i$ big enough, both $x_i$ and $\ft(x_i)$ are very near $L$. Thus, their unstable leaves $u(x_i)$ and $u(\ft(x_i))$ both intersect $L$. Let $y_i$ be the intersection of $u(x_i)$ with $L$ (recall that this intersection is unique as the center stable branching foliation is approximated by a taut foliation). Then $\ft(y_i)$ is the intersection of $u(\ft(x_i))$ with $L$ (since $L$ is fixed by $\ft$). Then the intersection of the unstable saturation of $e_i$ with $L$ is a compact segment inside a center leaf between $y_i$ and $\ft(y_i)$ (since $\ft$ fixes every center leaf).
Let $b_i$ be this segment between $y_i$ and $\ft(y_i)$. The segments $b_i$ also converge to $e$, so the previous paragraph shows that the lengths of the $b_i$ are bounded. Since the distance between $x_i$ and $y_i$ converges to zero, this in turn implies that the lengths of the segments $e_i$ are themselves bounded. Which contradicts our assumption and finishes the proof.
\end{proof}

\begin{lemma}\label{l.nomergefix}
	Let $f \colon M \to M$ be a partially hyperbolic diffeomorphism homotopic to the identity preserving
branching foliations $\fbs,\fbu$. Let $\ft$ be a good lift that fixes every center leaf. If $c_1,c_2$ are different center leaves in a single center stable leaf $L \in \wfbs$, then $c_1 \cap c_2 = \emptyset$. 
\end{lemma}

\begin{proof}
Suppose that there are distinct center leave $c_1, c_2$ that intersect at a point $x \in c_1 \cap c_2$. Then $\ft(x)$ is also in $c_1 \cap c_2$. If $c_1$ coincides
with $c_2$ in their respective segments from $x$ to $\ft(x)$, then
applying iterates of $\ft$ implies that $c_1 = c_2$, contrary
to assumption.

So we may assume that $x$ is a boundary point of an open interval
$I$ in, say, $c_1$ which is disjoint from $c_2$, but such that both endpoints
are in $c_2$. Then $c_1 \cup c_2$ bounds a bigon $B$ with endpoints
$x, y$ and a ``side" in $I$. All center segments in $B$ pass
through $x$ and $y$ and they have bounded length (by Lemma \ref{l.boundedmovement}).
Each stable
segment intersecting $I$ also intersects the other ``boundary"
component of $B$.  See figure \ref{fig.nomergefix}.

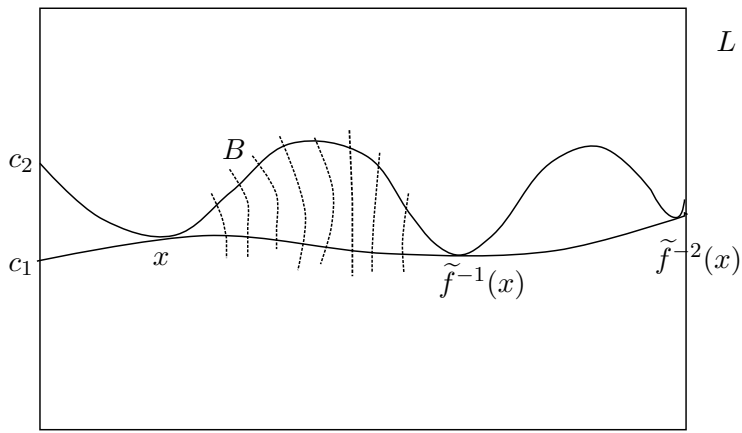
\begin{figure}[ht]
\begin{center}
\includegraphics[scale=0.7]{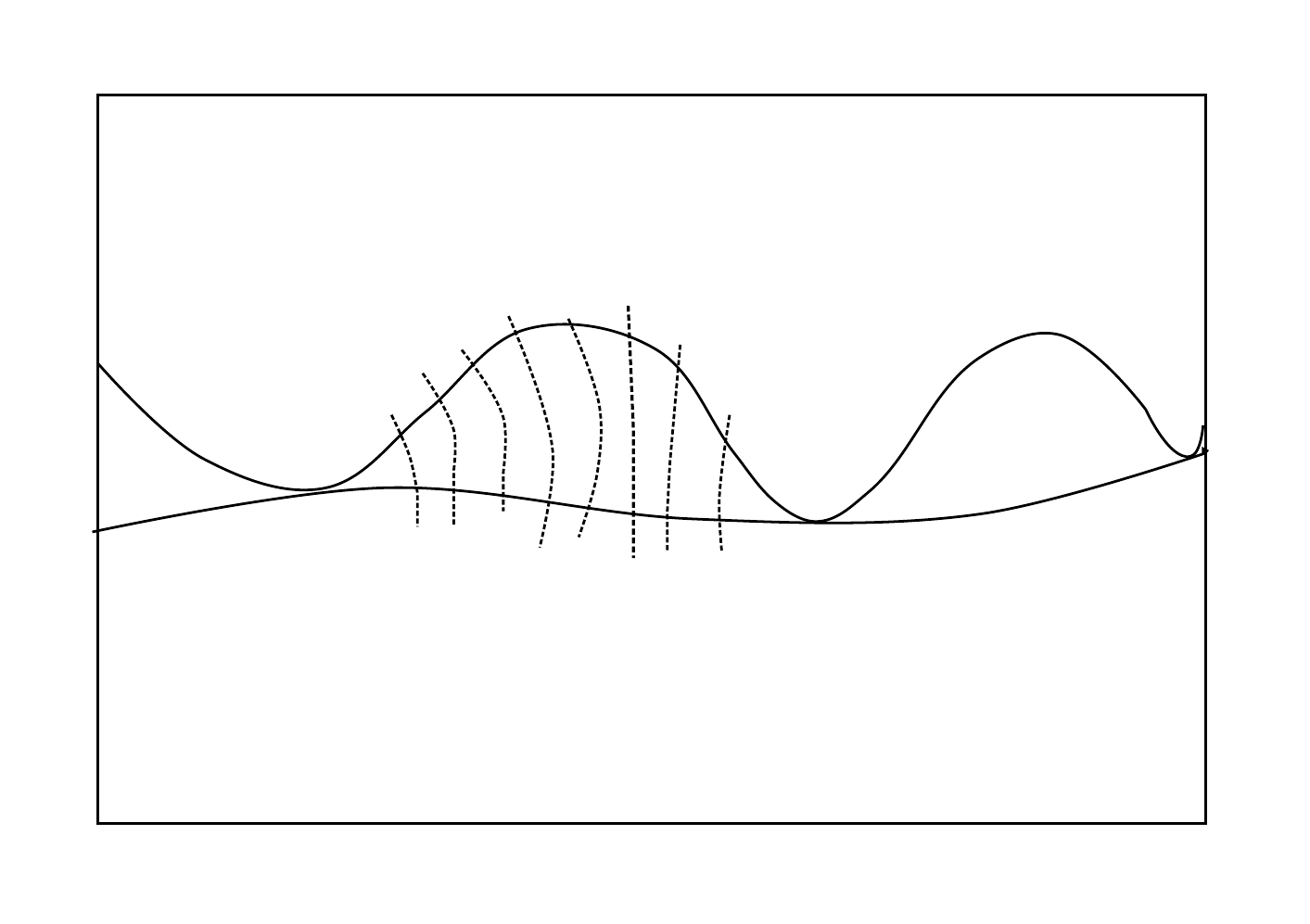}
\begin{picture}(0,0)
\put(-15,165){$L$}
\put(-280,82){$c_1$}
\put(-280,121){$c_2$}
\put(-119,74){$\ft^{-1}(x)$}
\put(-226,84){$x$}
\put(-38,83){$\ft^{-2}(x)$}
\put(-200,124){$B$}
\end{picture}
\end{center}
\vspace{-0.5cm}
\caption{{\small Two centers that merge. The bound on the distance between $x$ and $\ft(x)$ forces a behavior like the figure.}}\label{fig.nomergefix}
\end{figure}

The stable lengths grow without bound under negative iterates of $\ft$. Hence, since a stable segment can intersect a local
foliated disk of the stable foliation in $L$ only
in a bounded length,
it follows that the diameter in $\ft^n(L)$ of $\ft^n(B)$ grows without bound as $n$ goes to $-\infty$. 
But the length of the 
center segments in $\ft^n(B)$ are all bounded according to Lemma \ref{l.boundedmovement}. Moreover, between any two points in $\ft^n(B)$ there exists a path along (at most) two center leaves (one just follows the center leaf to one of the endpoint and then switch to the appropriate other center leaf). Thus the diameter is bounded, which is a contradiction.
\end{proof}

Thus we deduce what we wanted to obtain in this section.
\begin{corollary}\label{coro.fixcentercoh}
	Let $f \colon M \to M$ be a partially hyperbolic diffeomorphism homotopic to the identity preserving branching foliations $\fbs,\fbu$. If some good lift $\tilde f$ fixes every center leaf then $f$ is dynamically coherent. 
\end{corollary}
\begin{proof} By Proposition \ref{p.BWcoh} it is enough to show that the leaves of the branching foliations do not merge.  

Assume that two center unstable leaves $U_1$ and $U_2$
merge. Let $L$ be a center stable leaf intersecting $U_1$ and $U_2$ at the merging, i.e., $L$ is a leaf through a point $x$ such that the unstable leaf through $x$ is a boundary component of $U_1 \cap U_2$.
Then, connected components of 
$U_1 \cap L$ and $U_2 \cap L$ gives two center leaves that intersect but do not coincide.
This contradicts Lemma \ref{l.nomergefix}. A symmetric argument gives that two center stable leaf cannot merge either, proving dynamical coherence of $f$.
\end{proof}

\subsection{Dynamical coherence without taking lifts and iterates}
\label{dcwithout}

 We now want to prove that, if a finite lift and finite power of a partially hyperbolic diffeomorphism is dynamically coherent, then the original diffeomorphism is itself dynamically coherent. Although we do not know how to prove it in this generality, we show it when a good lift of the dynamically coherent lift fixes every center leaf, which is enough for our purposes.
 
Again, in this subsection we do not assume that $f$ is orientable.

We start by showing a uniqueness result for the pairs of the center stable and center unstable foliations under some conditions.
\begin{lemma}\label{lema-uniquefixfol}
Let $g\colon M\to M$ be a dynamically coherent partially hyperbolic diffeomorphism homotopic to the identity. Let $\cW^{cs}$ and $\cW^{cu}$ be $g$-invariant foliations tangent to $E^{cs}$ and $E^{cu}$ respectively. Let $\cW^c$ be the center foliation associated with $\cW^{cs}$ and $\cW^{cu}$ (defined as in Definition \ref{d.center_leaf}), and assume that there exists a good lift $\widetilde g$ which fixes all the leaves of $\wt \cW^c$.

Suppose that $\cW^{cs}_1$ and $\cW^{cu}_1$ are two $g$-invariant foliations tangent respectively to $E^{cs}$ and $E^{cu}$. Suppose that $\wt g$ also fixes all the leaves of the center foliation $\wt\cW_{1}^c$, associated with $\cW^{cs}_1$ and $\cW^{cu}_1$.

Then $\cW^{cs}=\cW^{cs}_1$ and $\cW^{cu}=\cW^{cu}_1$. 
\end{lemma}

\begin{proof}
The argument is similar to the one made in Lemma \ref{l.nomergefix}.

Let $\widetilde \cW^{cs}_1, \widetilde \cW^{cu}_1$ be two $g$-equivariant foliations as in the lemma. We will consider the center foliation $\widetilde \cW^c_1$ is defined by taking the connected components of intersections of leaves of $\widetilde \cW^{cs}_1$
and $\widetilde \cW^{cu}$ to show that $\widetilde \cW^{cs} = \widetilde \cW^{cs}_1$. A symmetric argument shows that $\widetilde \cW^{cu}= \widetilde \cW^{cu}_1$. 

Since every leaf of both $\wt\cW^{c}$ and $\wt\cW^{c}_1$ are fixed by $\wt g$, Lemma \ref{l.boundedmovement} implies that $\widetilde g$ moves points a uniformly bounded amount in both center foliations.

Consider, for a contradiction, a point $x\in \mt$ such that $\widetilde \cW^c(x) \neq \widetilde \cW^{c}_1(x) $ (note that we are dealing here with actual foliations, not branching ones, so this notation make sense). Without loss of generality, we can choose $x$ so that the leaves $L:= \widetilde \cW^{cs}(x)$ and $L_1:=\widetilde \cW^{cs}_1(x)$ do not coincide in any neighborhood of $x$. 

Let $c$ and $c_1$ be the center leaves obtained as the connected components of $L \cap F$ and $L_1 \cap F$ containing $x$ for some $F \in \widetilde \cW^{cu}$.

By assumption, both $c$ and $c_1$ are fixed by $\wt g$, so we are in the exact same set up as in the proof of Lemma \ref{l.nomergefix}. Thus we deduce that $c=c_1$, a contradiction.
\end{proof}

We can now state and prove the aim of this section.

\begin{proposition}\label{coro-DCwithoutlift}
Let $f\colon M \to M$ be a partially hyperbolic diffeomorphism such that $f^k$ is homotopic to the identity for some $k > 0$. Let $\hat M$ be a finite cover of $M$ which makes all bundles orientable. Let $g$ be a lift to $\hat M$ of a homotopy of $f^k$ to the identity that preserves orientation of the bundles.
Suppose that $g$ is dynamically coherent 
and that there exists a good lift $\widetilde{g}$ of $g$ that fixes all the center leaves. Then, $f$ is dynamically coherent and $f^k$ is a discretized Anosov flow. 

\end{proposition} 

\begin{proof} 
First we notice that the assumptions of the proposition will be verified for any further finite cover $\bar M$ of $\hat M$ (because one can take a further lift $\bar g$ of $g$ to $\bar M$, it is dynamically coherent and $\wt g$ is a good lift of $\bar g$ too). Hence, without loss of generality, we may and do assume that $\hat M$ is a normal cover of $M$.

Let $\wt\cW^{cs}$ and $\wt\cW^{cu}$ be the lifts to $\wt M$ of the center stable and center unstable foliations of $g$. Our goal is to show that these foliations are $\pi_1(M)$-invariant, thus decending to foliations in $M$, and that these projected foliations are $f$-invariant. 

Notice that $\wt g$ fixes each leaf of $\wt\cW^{cs}$ and $\wt\cW^{cu}$.

The map $g$ is obtained from a lift of a homotopy of $f^k$ to the identity. Lifting that homotopy further to $\mt$, we get a good lift $\wt{f^k}$ of $f^k$ that is also a lift (and hence a good lift) of $g$ to $\mt$. 
As both $\wt g$ and $\wt{f^k}$ are good lifts of $g$, there exists $\beta\in \pi_1(\hat M) \subset \pi_1(M)$ such that $\widetilde g = \beta \ftk$. (Note however that $\wt g$ is not necessarily a good lift of $f^k$ as $\wt g$ only commutes with elements of $\pi_1(\hat M)$ and not $\pi_1(M)$.) 

Moreover, both $\wt g$ and $\ftk$ move points a bounded distance in $\wt M$, hence so does $\beta = \widetilde g (\ftk)^{-1}$. Lemma \ref{boundcov} then implies that either $\beta$ is the identity or $M$ is Seifert (and $\beta$ is either the identity or a power of a regular fiber).

We split the rest of the proof of dynamical coherence in two cases.

\vskip .1in
\noindent
{\bf {Case 1}} $-$ Suppose that $M$ is \emph{not} a Seifert fibered space.

Then $\beta$ is the identity, which means that $\wt g = \ftk$.

Let $\gamma$ be a deck transformation in $\pi_1(M)$. 
Define the foliations $\cF^{cs}_\gamma := \gamma \widetilde \cW^{cs}$, $\cF^{cu}_\gamma := \gamma \widetilde \cW^{cu}$, and $\cF^{c}_{\gamma} :=\gamma\wt \cW^{c}$. The leaves of these foliations are all fixed by $\widetilde g$ because $\gamma$ commutes with $\ftk = \widetilde g$.
In particular, Lemma \ref{lema-uniquefixfol} then implies that 
$\gamma \wcs = \wcs$ and $\gamma \wcu = \wcu$. Since this is true for any element of $\pi_1(M)$, these foliations descend to foliations $\cW^{cs}_M, \cW^{cu}_M$
in $M$.

Now we need too show that $\cW^{cs}_M, \cW^{cu}_M$ are also $f$-invariant. Equivalently, we need to show that $\wcu$ and $\wcs$ are invariant by any lift $f_1$ of $f$ to $\wt M$.

Let $f_1$ be a lift of $f$ to $\mt$. Notice that
$f$ may not be homotopic to the identity, so $f_1$ is not assumed to be a good lift.
Let $\cF^{cs}_1 := f_1(\wcs)$ and $ \cF^{cu}_1 := f_1(\wcu)$.

We will first show that $f_1$ and $\widetilde g$ commute.
Both $f_1 \widetilde g$ and $\widetilde g f_1$
are lifts of the map $f^{k+1}$ to $\mt$. So
$(\widetilde g)^{-1} (f_1)^{-1} \widetilde g f_1$ is
a deck transformation $\gamma\in \pi_1(M)$.
As $\wt g$ moves points a bounded distance, we have that $d(f_1(y), \widetilde g f_1(y))$ is bounded in $\mt$. In addition, $f_1$ has bounded derivatives
so $d(y, (f_1)^{-1} \widetilde g f_1(y))$ is also bounded
in $\mt$. So using again that $\wt g$ is a good lift, we deduce that $d(y, (\widetilde g)^{-1} (f_1)^{-1} \widetilde g f_1(y))$ is bounded in $\mt$.

Hence $\gamma$ is a deck transformation that moves
points a bounded distance. Applying Lemma \ref{boundcov} again gives that $\beta$ is the identity (since $M$ is not Seifert). Hence $f_1$ and $\widetilde g$
commute.

Since $\widetilde g$ fixes every leaf of $\widetilde \cW^{c}$ (the center foliation in $\mt$) and commutes with $f_1$, we deduce that $\widetilde g$ fixes
every leaf of $f_1(\wt\cW^c)$.
We can again apply Lemma \ref{lema-uniquefixfol} to get that $f_1(\wt\cW^{cs})= \wcs$ and $f_1(\wcu)= \wcu$. That is, the foliations $\widetilde \cW^{cs}$
and $\widetilde \cW^{cu}$ are $f_1$-invariant. Since this holds for any lift of $f$, it implies that $\cW^{cs}_M$ and $\cW^{cu}_M$ are $f$-invariant. 
Hence $f$ is dynamically coherent with foliations 
$\cW^{cs}_M, \cW^{cu}_M$.
This completes the proof that $f$ is dynamically
coherent when $M$ is not Seifert fibered.

\vskip .15in
\noindent
{\bf {Case 2}} $-$ Assume that $M$ is Seifert fibered.

In this case, Lemma \ref{boundcov} implies that $\beta= \widetilde g (\ftk)^{-1}$ is either the identity or represents a power of a regular fiber of the Seifert fibration. In any case, $\beta$ is in a normal subgroup of $\pi_1(M)$ isomorphic to $\ZZ$.
Moreover, as proved earlier, $\beta \in \pi_1(\hat M)$.

Let $\gamma \in \pi_1(M)$ be any deck transformation.  
As before, consider the foliations
$\cF^{cs}_{\gamma} := \gamma \widetilde \cW^{cs}$ and 
$\cF^{cu}_{\gamma}  := \gamma \widetilde \cW^{cu}$.

We first claim that these foliations are $\widetilde g$-invariant. We show this for $\cF^{cs}_{\gamma}$ the other being analogous. Let $L \in \widetilde \cW^{cs}$. We have

$$ \widetilde g (\gamma L) = \beta \ftk (\gamma L) = \beta \gamma \ftk (L) = \gamma \beta^{\pm 1} \ftk (L). $$

Notice that both $\ftk$ (because it is a lift of $g$) and $\beta$ (because it belongs to $\pi_1(\hat M)$ and the foliation $\cW^{cs}$ is defined in $\hat M$) preserve the foliation $\widetilde \cW^{cs}$. 
It follows that $\beta^{\pm 1} \ftk (L)\in \ \widetilde \cW^{cs}$, so

$$ \widetilde g (\gamma L) = \gamma \beta^{\pm 1} \ftk (L)\in \cF^{cs}_{\gamma}.$$ 
Thus $\cF^{cs}_{\gamma}$ is $\widetilde g$-invariant.

We now want to show that the foliations $\cF^{cs}_{\gamma}$, 
$\cF^{cu}_{\gamma}$ and $\cF^c_{\gamma} := \gamma \wt\cW^c$ are all leafwise fixed by $\widetilde g$. 

Since $\hat M$ was chosen to be a normal cover of $M$, any element $\gamma \in \pi_1(M)$ can be thought of as a difeomorphism of $\hat M$. Hence we can consider the foliation $\hat \cF_{\gamma}^{cs} := \gamma \cW^{cs}$ in $\hat M$. Note that $\hat \cF_{\gamma}^{cs}$ is tangent to the center stable distribution $E^{cs} \subset T\hat M$, since $\gamma$ preserves the tangent bundle decomposition, as it is defined by $f$ in $M$.
The argument above shows that $\hat\cF_{\gamma}^{cs}$ is $g$-invariant. 

Thus, we can consider $g$ to be a dynamically coherent diffeomorphism for the pair of transverse foliations $\hat\cF_{\gamma}^{cs}$ and $\cW^{cu}$. Moreover, $g$ is homotopic to the identity and the good lift $\wt g$ fixes every leaf of $\wt\cW^{cu}$. Since $\hat M$ is Seifert, mixed behaviour is excluded 
(cf. \cite[Theorem 5.1]{BFFP-prequel}) and this implies that $\wt g$ must also fix every leaf of $\cF_{\gamma}^{cs}$.

The symmetric argument shows that $\cF^{cu}_{\gamma}$ is also fixed by $\wt g$.
So we can apply Proposition \ref{p.hypSeifminimal_nonDC} to both $\hat\cF_{\gamma}^{cs}$ and $\hat\cF^{cu}_{\gamma}$, implying that they are $g$-minimal. 
To apply the proposition we need that $g$ is orientable.
Hence, the center foliation $\cF^c_{\gamma}$ is fixed by $\wt g$, thanks to Proposition \ref{p.nondceverycfixed} (this also uses that
$g$ is orientable).

Since all the leaves of $\cF^c_{\gamma}$ are fixed by $\wt g$, we can finally apply Lemma \ref{lema-uniquefixfol} to deduce that $\cF^{cs}_{\gamma}= \wt\cW^{cs}$ and $\cF^{cu}_{\gamma}= \wt\cW^{cu}$.
As this is true for any $\gamma$, the foliations $\widetilde \cW^{cs}$ and $\widetilde \cW^{cu}$ descends to foliations $\cW^{cs}_M$ and $ \cW^{cu}_M$ on $M$ in this case too.

We now again have to show that $\cW^{cs}_M$ and $\cW^{cu}_M$ are $f$-invariant. The argument is the same for both foliations, so we only deal with $\cW^{cs}_M$.

We start with a preliminary step. Let $f_*$ be the automorphism of $\pi_1(M)$ induced by $f$. Let

$$A :=\pi_1(\hat M) \cap f_*(\pi_1(\hat M)) \cap 
\dots \cap (f_*)^{k-1}(\pi_1(\hat M)).$$
The set $A$ is a finite index, normal subgroup of $\pi_1(M)$.
Moreover, as $f^k$ is homotopic to the identity, $f_*(A) = A$.

As we remarked at the beginning of the proof, we can without loss of generality prove the result for any further finite cover of $\hat M$. Thus we choose if necessary a further cover so that $\pi_1(\hat M) = A$.
Since $f_*(A) = A$, the map $f$ lifts to a 
homeomorphism $\hat f$ of $\hat M$.

As in the first case, we let $f_1$ be an arbitrary lift of $\hat f$ to $\mt$
and we define $\fol^{cs}_1 := f_1(\wcs)$ and $\fol^{cu}_1 := f_1(\wcu)$. (Note that $f_1$ is in particular also a lift of $f$.)

Note as before that both $\wt g f_1$ and $f_1 \wt g$ are lifts of $f^{k+1}$, and $\widetilde g f_1 (\widetilde g)^{-1} (f_1)^{-1}$ is a bounded distance
from the identity (because $\wt g$ is and $f_1$ has bounded derivatives). So $\delta := \widetilde g f_1 (\widetilde g)^{-1} (f_1)^{-1}$ is an element of $\pi_1(M)$ a bounded distance from identity. By Lemma \ref{boundcov}, $\delta$ represents a power of a regular fiber of the Seifert fibration, so is in the normal ${\ZZ}$ subgroup of $\pi_1(M)$ (note that since $\pi_1(M)$ is not virtually nilpotent, there exists a unique Seifert fibration on $M$, see Appendix \ref{app.3_manifold_topology}).

In addition $\widetilde g f_1$ and $f_1 \widetilde g$ are also
lifts of the homeomorphisms $g \hat f$ and $\hat f g$ in $\hat M$ to
$\widetilde M$. Hence $\delta$ is in $\pi_1(\hat M)$.

Using once more the arguments above, we get that
$(f_1)^{-1} \delta f_1 (\delta)^{-1}$ is a bounded distance from
the identity, and projects to the identity in $M$ (and in $\hat M$), hence it is
a deck transformation 
$\eta$ also contained in the ${\ZZ}$ normal subgroup of $\pi_1(M)$. Thus $\delta$ and $\eta$ commute.
 Moreover, $\eta$ is also in $\pi_1(\hat M)$.

Now we can show that $\widetilde g$ preserves $\fol^{cs}_1$:
Let $L$ in $\wcs$. Then
$$
\widetilde g(f_1(L)) = \delta f_1(\widetilde g(L)) =
\delta f_1(L) = f_1(\eta \delta(L)).
$$
Here $\eta \delta(L)$ is in $\wcs$, because $L$ is 
in $\wcs$ and $\eta \delta$ is in $\pi_1(\hat M)$.
Hence $\widetilde f_1(\eta \delta L)$
is in $f_1(\wcs)$ so $\widetilde g$ preserves $\fol^{cs}_1$.

What we proved implies that $g$ preserves
$\hat f(\cW^{cs})$ in $\hat M$. Now consider the pair of foliations $\hat f(\cW^{cs})$ and $\cW^{cu}$. They are both invariant by $g$, so $g$ is dynamically coherent for this particular pair of foliations, and $\wt g$ fixes the leaves of $\wt \cW^{cu}$. So once again, as $\hat M$ is Seifert,  we get that $\wt g$ must also fix every leaf of $f_1(\wt\cW^{cs})$ (cf. \cite[Theorem 5.1]{BFFP-prequel}).

The symmetric argument implies that $\wt g$ fixes every leaf of $f_1(\wt\cW^{cu})$. Once again, $\hat M$ being Seifert implies that all the foliations are $g$-minimal (Proposition \ref{p.hypSeifminimal_nonDC}). Hence $\wt g$ also fixes the center foliation $f_1(\wt\cW^c)$ (Proposition \ref{p.nondceverycfixed}). So Lemma \ref{lema-uniquefixfol} applies and we deduce that $f_1(\wt\cW^{cs})=\wt\cW^{cs}$ and $f_1(\wt\cW^{cu})=\wt\cW^{cu}$. 

In particular, $f$ preserves the foliations $\cW^{cs}_M$ and $\cW^{cu}_M$ as wanted. So $f$ is dynamically coherent.

This finishes the proof that $f$ is dynamically coherent.
Once that is known, then Proposition 6.5 and Proposition
G.2 of \cite{BFFP-prequel} implies that 
$f^k$ is a discretized Anosov flow.
This finishes the proof of the proposition.
\end{proof}

\section{Proof of Theorem~\ref{thmintro:Seifert}}\label{sec-ThmA}


Fix a partially hyperbolic diffeomorphism $f \colon M\to M$ that is homotopic to the identity on a closed Seifert-fibered $3$-manifold $M$. We make no orientability assumptions. We will show that some iterate of $f$ is a discretized Anosov flow, completing the proof of Theorem \ref{thmintro:Seifert}. 

Fix a finite cover $\hat M$ of $M$ so that the lifted center, stable, and unstable bundles are orientable. Then there is an integer $k > 0$, such a lift of $f^k$ to $\hat M$ will preserve the orientations of the bundles. In addition, we can find such a lift that is homotopic to the identity by lifting a homotopy from $f^k$ to the identity. Fix such a lift $g: \hat M \to \hat M$. 

Applying Theorem \ref{teo-BI}, we have $g$-invariant center stable and center unstable branching foliations $\fbs$ and $\fbu$ on $\hat M$.

\begin{lemma}\label{lem_no_mixed_nonDC}
 There exists a lift $\widetilde g$ of an iterate of $g$ that fixes every leaf of  $\wfbs$ and also fixes every leaf of $\wfbu$.
\end{lemma}
\begin{proof}
	We will use the following result, found in \cite[Proposition~7.1 \& Remark~7.2]{BFFP-prequel}.
	\begin{proposition}\label{p.liftfixleaf}
		Let $g: M \to M$ be a partially hyperbolic diffeomorphism that is homotopic to the identity on a Seifert fibered $3$-manifold $M$ with orientable Seifert fibration. Then some iterate of $g$ has a good lift which fixes every leaf of $\wfbs$.
	\end{proposition}
	
	Since $\hat M$ is orientable, the bundles are orientable, and $\fbs$ is a horizontal foliation (see \cite[Theorem F.3]{BFFP-prequel}), it follows that the Seifert fibration is orientable.  Thus there is an integer $i > 0$ so that the iterate $g^i$ has a good lift $\widetilde{g^i}$ which fixes every leaf of $\wfbs$.

	Suppose that $\widetilde{g^i}$ fixes one leaf of $\wfbu$. Then Proposition \ref{p.hypSeifminimal_nonDC} says that $\fbu$ is $g^i$-minimal and 
$\wt{g^i}$ fixes \emph{every} leaf of $\wfbu$ as desired. 
	
	Suppose, then, that $\wt{g^i}$ fixes no leaf of $\wfbu$. Then $\widetilde g$ fixes no center leaf, and we can apply Proposition \ref{p.alternnonDC} to see that every periodic center leaf of $g$ has to be coarsely contracting. Exchanging roles, and applying Proposition \ref{p.liftfixleaf} to the center unstable branching foliation we deduce that every periodic center leaf for $g$ must be coarsely expanding. Notice that although the lifts may be different, the coarsely expanding and coarsely contracting behavior is for periodic center leaves of the original map $g$.
	
	As there must be at least one such periodic center leaf (cf.~Proposition \ref{periodiccenter}) this gives a contradiction.
\end{proof}

Let $\widetilde{g^i}$ be a good lift of an iterate $g^i$, $i > 0$, that fixes every leaf of both $\wfbs$ and $\wfbu$. Then Proposition \ref{p.nondceverycfixed} implies that $\widetilde{g^i}$ fixes every center leaf, and Corollary \ref{coro.fixcentercoh} says that $g^i$ is dynamically coherent. Then Proposition \ref{coro-DCwithoutlift} tells us that $f$ is dynamically coherent.

Now that we have reduced to the dynamically coherent case, \cite[Theorem A]{BFFP-prequel} says that $f$ has an iterate that is a discretized Anosov flow. This completes the proof of Theorem~\ref{thmintro:Seifert}. 

Note that the arguments in the proof of Lemma~\ref{lem_no_mixed_nonDC} also eliminate mixed behavior for good lifts in Seifert fibered manifolds.

\section{Absolutely partially hyperbolic diffeomorphisms}\label{ss.AbsolutePH} 


In this section, we explain how one can improve the trichotomy
in subsection  \ref{sss-introLeafSpaces} eliminating the mixed case,
if one uses a strong version of partial hyperbolicity. 

\begin{definition}
A partially hyperbolic diffeomorphism $f \colon M \to M$ on a 3-manifold is called \emph{absolutely partially hyperbolic} if there exists constants $\lambda_1<1<\lambda_2$ such that for some $\ell>0$ and every $x\in M$, we have 

$$ \|Df^\ell|_{E^s(x)}\| < \lambda_1 < \|Df^\ell|_{E^c(x)}\| < \lambda_2 < \|Df^\ell|_{E^u(x)}\|. $$ 
\end{definition}

Notice that, although subtle, the difference between being absolutely partially hyperbolic versus just partially hyperbolic is far from trivial. Here, we just show that with this stronger property one can significantly simplify the arguments. However, some previous results have shown significant differences between the two notions, specifically with regard to the integrability of the bundles (see \cite{BBI,HHU-noncoherent,PotT3}).

We will show the following
\begin{theorem}\label{thm_absolute_case}
 Let $f\colon M \to M$ be an absolutely partially hyperbolic diffeomorphism on a $3$-manifold. Suppose that $f$ is homotopic to the identity and preserves two branching foliations $\fbs$ and $\fbu$ that are both $f$-minimal.
 Then either
 \begin{enumerate}[label=(\roman*)]
  \item \label{item.absolute_discretized} $f$ is a discretized Anosov flow, or,
  \item \label{item.absolute_translation} $\fbs$ and $\fbu$ are $\R$-covered and uniform and a good lift $\ft$ of $f$ act as a translation on their leaf spaces.
 \end{enumerate}
\end{theorem}

In order to prove this theorem, the main step will be to show that, using absolute partial hyperbolicity, we have an improvement of Proposition \ref{p.alternnonDC}. 

\begin{proposition}\label{prop.absolute}
Let $f \colon M \to M$ be an absolutely partially hyperbolic diffeomorphism homotopic to the identity and $\ft$ a good lift of $f$ to $\mt$. Assume that every leaf of $\wfbs$ is fixed by $\ft$. Let $L$ be a leaf whose stabilizer is generated  by $\gamma \in \pi_1(M) \setminus \{\mathrm{id}\}$. Then, there is a center leaf in $L$ fixed by $\ft$. 
\end{proposition}

 The proof is essentially the same as the one in \cite[Section 5.4]{HaPS} but we repeat it since the contexts are different. 
\begin{proof}
The proof is by contradiction. Assume that $\ft$ does not fix any center leaf in $L$.

Proposition \ref{periodiccenter} gives that there exists a center leaf periodic by $f$. Call $c$ a lift of this center leaf. Using Proposition \ref{p.alternnonDC2} we get two stable leaves $s_1$ and $s_2$ in $L$ fixed by $h:= \gamma^n \circ \ft^m$, a bounded distance apart in $L$ and such that $c$ separates $s_1$ from $s_2$ in $L$. We denote by $B$ the band bounded by $s_1$ and $s_2$.

Since $\gamma$ is an isometry, the diffeomorphism $h$ is absolutely partially hyperbolic, and we can (modulo taking iterates) assume that there are constants $\lambda_1 < \lambda_2$ such that
$$ 
\|Dh|_{E^s}\| < \lambda_1 <  \lambda_2 < \|Dh|_{E^c}\|.
$$ 
Moreover, there is a constant $R>1$ such that $\|Dh^{-1}\| \leq R$ in all of $L$. 

For simplicity, we will assume that the distance between $s_1$ and $s_2$ is smaller than $1/2$ so that the band $B$ is contained in the neighborhood $\hat B=\bigcup_{x\in S_1} B_1(x)$ of radius $1$ around $s_1$.

For every positive $d$ there is a constant $r(d) > 0$ such that for any set of diameter less than $d$, the length of a stable
leaf contained in this set is at most $r(d)$. This is because in a foliated box only one segment of a stable segment can intersect it.
This implies that stable leaves (and center leaves as well) are quasi-isometrically
embedded in their neighborhoods of a fixed diameter. So there is 
$K > 0$ so that for any stable segment $J$ contained in $\hat B$ with
endpoints $z$ and $w$ we have
$$
\mathrm{length}(J) \leq K d_{\hat B}(z,w). 
$$

Now, choose $n>0$ such that $K^2 \frac{\lambda_1^n}{\lambda_2^n} \ll \frac{1}{2}$ and once $n$ is fixed, choose $D>0$ so that $\frac{D}{2} \gg 2 R^n + \frac{2K}{\lambda_2^n}$. 

We now pick points $z,w \in s_1$ such that $d_{\hat B}(z,w) = D$ and take $J^s$ an arc of $s_1$ joining these points. From the choice of $K$ and $D$ we know that $\mathrm{length}(J^s) \leq KD$. So, it follows that $\mathrm{length}(h^n(J^s)) \leq KD \lambda_1^n$. 

Choose a center curve $J^c$ joining $B_1(h^n(z))$ with $B_1(h^n(w))$ (this can be done because $c$ separates $s_1$ from $s_2$) and call $z_n$ and $w_n$ the endpoints in each ball. It follows that $\mathrm{length}(J^c) \leq K^2 D\lambda_1^n +2K$. 

Since the distance between the endpoints of $J^c$ and $h^n(z)$, $h^n(w)$ is less than $1$, by iterating backwards by $h^{-n}$ we get that $d(h^{-n}(z_n),z)$ and $d(h^{-n}(w_n),w)$ are less than $R^n$. 

This implies that

$$ D \leq d_{\hat B}(z,w) \leq K^2 \frac{\lambda_1^n}{\lambda_2^n} D + 2 R^n +\frac{2K}{\lambda_2^n}, $$
a contradiction with the choices of $n$ and $D$. This completes the proof of the proposition.
\end{proof}

Using this proposition, we can prove Theorem \ref{thm_absolute_case} in the same way as \cite[Theorem 5.1]{BFFP-prequel}.

\begin{proof}[Proof of Theorem \ref{thm_absolute_case}]
 Let $\ft$ be a good lift of $f$. Since $\fbs$ and $\fbu$ are $f$-minimal, by Corollary \ref{c.minimalcase_branching_case}, $\ft$ either fixes each leaf of $\wfbs$ and $\wfbu$, or act as a translation on both leaf space (in which case the foliations are $\R$-covered and uniform and we are in case \ref{item.absolute_translation} of the theorem), or $\ft$ translates one and fixes the other.
 
 If $\ft$ fixes the leaves of both $\wfbs$ and $\wfbu$ then Proposition \ref{p.nondceverycfixed} and Corollary \ref{coro.fixcentercoh} imply that we are in case \ref{item.absolute_discretized} of the theorem.
 
 So we have to show that we cannot be in the mixed case. Suppose that $\ft$ fixes every leaf of $\wfbs$. 
 
 Since $M$ is not $\mathbb{T}^3$, there are leaves of $\fbs$ with non-trivial fundamental group. 
Consider the lift $L$ in
$\wfbs$ of such a leaf, with
$L$ invariant by $\gamma$ in $\pi_1(M) \smallsetminus\{\id\}$.
We can apply Proposition \ref{prop.absolute} to conclude
that there is a center leaf $c$ in $L$ that is fixed by
$\ft$. So, in particular, $\ft$ needs to fix a center unstable leaf containing $c$ (note that there may be an interval of center unstable leaves intersecting $L$ in $c$, but the endpoints of such interval will then be fixed by $\ft$). Thus $\ft$ has to also fix every leaf of $\wfbu$ by Corollary \ref{c.minimalcase_branching_case}.
\end{proof}




\section{Regulating pseudo-Anosov flows and translations} \label{s.pA_and_translations_nonDC}
The rest of the paper is concerned with hyperbolic 3-manifolds. We will get positive results dealing with the non-dynamically coherent case. That is, we want to understand the dynamics of a homeomorphism acting by translation on a branching foliation. In order to be able to do that, we first need to build a regulating pseudo-Anosov flow transverse to the branching foliation. The existence of such a flow is a relatively immediate consequence of the construction of the regulating flow and the fact that the branching foliation is well-approximated by foliations.

\begin{proposition}\label{prop-transversepA} 
Let $M$ be a hyperbolic 3-manifold and $\fol$ a branching foliation well-approximated by foliations $\fole$ such that $\fol$ (and thus also $\fole$ for small $\epsilon$) are $\R$-covered and uniform. Then, there exists a transverse and regulating pseudo-Anosov flow $\Phi$ for $\fol$.
\end{proposition} 

\begin{proof}
By \cite{Thurston2, Calegari00,Fen2002} (see \cite[Theorem D.3]{BFFP-prequel}) for any $\eps$, there exists a pseudo-Anosov flow $\pe$ transverse to and regulating for $\fole$.

Now, as $\eps$ get small, the angle between leaves of $\fole$ and leaves of $\fol$ becomes arbitrarily small.

Then, since both $\fol$ and $\fole$ are $\mathbb{R}$-covered and uniform, for any leaf $L\in \widetilde{\fol}$, there exists two leaves $L_1,L_2 \in \widetilde{\fole}$ such that $L$ is in between $L_1$ and $L_2$ (note that by construction, each leaf of $\fol$ is the image of a leaf of $\fole$ by a continuous map homotopic to identity of $M$, so, given a leaf $L \in \widetilde{\fol}$ there is a leaf $L' \in \widetilde{\fole}$ at a bounded distance $< a_1$ from $L$. Now using the fact that $\fol_\eps$ is uniform, choose $L_1, L_2$ in 
$\widetilde{\fole}$ on different components of $\mt - L'$,
and so that for any $p \in L', \ q \in L_1, \ z \in L_2$, then
$d(p,q) > a_1, \ d(p,z) > a_1$. The leaves $L_1, L_2$ satisfy the required
property).
As $\pe$ is regulating for $\fole$, every orbit of $\widetilde{\pe}$ intersects both $L_1$ and $L_2$, thus it also intersects $L$. So every orbit of $\widetilde{\pe}$ intersect every leaf of $\widetilde{\fol}$, that is, $\pe$ is regulating for $\fol$.

The fact that the flow $\pe$ can be chosen transverse to $\widetilde{\fol}$ follows from the construction of $\pe$ (see \cite{Thurston2, Calegari00,Fen2002}).
The flow $\pe$ is build by blowing down certain laminations transverse to $\fole$. Moreover these laminations are transverse to any foliation that are close enough to $\fole$ for a uniform angle. Since the angle between $\fol$ and $\fole$ gets arbitrarily small, $\pe$ will also be transverse.
For a continuous family of $\mathbb{R}$-covered foliations, this
property is explicitely stated in \cite[Corollary 5.3.22]{Calegari00}.
\end{proof}

Using the regulating pseudo-Anosov flow given by Proposition \ref{prop-transversepA}, all of \cite[Section 8]{BFFP-prequel}
works for a branching foliation without change. Thus we obtain

\begin{proposition}\label{p.homeotranslation_nonDC} 
Let $M$ be a hyperbolic $3$-manifold. Let $f\colon M \to M$ be a homeomorphism homotopic to the identity that preserves a (branching) foliation $\fol$. Suppose that $\fol$ is uniform and $\R$-covered, and that a good lift $\ft$ of $f$ acts as a translation on the leaf space of $\fol$.
Let $\Phi$ be a transverse regulating pseudo-Anosov flow to $\fol$.

Then, for every $\gamma \in \pi_1(M)$ associated with a periodic orbit of $\Phi$, there is a compact $\hat f_{\gamma}$-invariant set $T_\gamma$ in $M_\gamma$ which intersects every leaf of $\hat \fol_{\gamma}$, where $M_{\gamma} = \rquotient{\mt}{\langle\gamma\rangle}$ and $\hat f_{\gamma}\colon M_{\gamma} \to M_{\gamma}$ is the corresponding lift of $f$.

 Moreover, if an iterate $\hat f_{\gamma}^k$ of $\hat f_{\gamma}$ fixes a leaf $L$ of $\hat \fol_{\gamma}$, 
and $\gamma$ fixes all the prongs of this orbit, 
then the fixed set of $\hat f_{\gamma}^k$ in $L$ is contained in $T_\gamma \cap L$ and has negative Lefschetz index. 
\end{proposition}

Almost without any change, we also obtain the corresponding version of  \cite[Proposition 9.1]{BFFP-prequel}

\begin{proposition}\label{fixedleaf_nonDC}
Let $f$ be partially hyperbolic diffeomorphism in a hyperbolic $3$-manifold
which preserves a branching foliation $\fbs$ tangent to $E^{cs}$.
Assume that a good lift $\ft$ of $f$ acts as a translation on the foliation $\fbs$ and let $\Phi^{cs}$ be a transverse regulating pseudo-Anosov flow. Then, for every $\gamma \in \pi_1(M)$ associated to the inverse periodic orbit of $\Phi^{cs}$ there are $n>0, m>0$ such that $h= \gamma^n \circ \ft^m$ fixes a leaf $L$ of $\fbs$. 
\end{proposition}

\begin{proof}
The only difference is that we cannot say that the action of $h$ in the leaf space is expanding since collapsing of leaves may change the behavior. However, the same proof gives the existence of an interval in the leaf space which is mapped inside itself by $h^{-1}$ giving a fixed leaf as desired. 
\end{proof}

\begin{remark}
Note that in the non dynamically coherent situation, the proof of \cite[Theorem B]{BFFP-prequel} does not give a contradiction: it could happen  (and indeed happens in a situation
with similar properties, see e.g.,~\cite{BGHP}) that having a fixed point in a leaf of the foliation, does not force the dynamics on the leaf space to be repelling around the leaf in terms of the action on the
leaf space. This issue has previously appeared, in particular in Proposition \ref{l.analogL53}. 
\end{remark}

Notice that if one assumes the existence of a periodic center leaf, then we can easily prove a version of \cite[Theorem B]{BFFP-prequel} in the non dynamically coherent setting.

\begin{proposition}\label{prop-periodiccenterleaf}
 Let $f\colon M \to M$ be a partially hyperbolic diffeomorphism on a hyperbolic $3$-manifold. Suppose that there exists a closed center leaf $c$ that is periodic under $f$. Then $f$ is a discretized Anosov flow.
\end{proposition}

\begin{proof}
 We start by replacing $f$ by a power, so that $f$ becomes homotopic to the identity.
 
 Let $\ft$ be a good lift of $f$. We will show that $\ft$ fixes every leaf of $\wfbs$ and $\wfbu$. Then, section \ref{sec-doubleinvndc} above shows that the original $f$ (before taking a power) is dynamically coherent, hence the result  follows from  \cite[Theorem B]{BFFP-prequel}. 
 
 Suppose that $\ft$ does not fix every leaf of, say, $\wfbs$. Then Corollary \ref{c.minimalcase_branching_case} implies that the leaf space of $\wfbs$ is $\R$ and that $\ft$ acts as a translation on it.
 
 Let $\wt c$ be a lift of the periodic closed center leaf $c$. Since $c$ is periodic and $\ft$ acts as a translation, there exists $\gamma \in \pi_1(M)$, non-trivial such that $\gamma (\wt c) = \ft^k (\wt c)$ for some $k$. Now $c$ is also closed, so there exists $g \in \pi_1(M) - id$  such that $g(\wt c) = \wt c$.  We have that $g$ is distinct from any power of $\gamma$, since if $L\in \wfbs$ is such that $\wt c \in L$ we have that $g(L)=L \neq \gamma^k(L)$ for every $k \neq 0$. 
 
 On the other hand, $g \circ \gamma(\wt c) = g \circ \ft^k(\wt c) = \ft^k \circ g(\wt c)= \gamma(\wt c)$ which implies that $\gamma^{-1} \circ g \circ \gamma$ and $g$ fix $\wt c$. This is impossible since $M$ is hyperbolic:
if they both fix $\wt c$ then they have they have the same 
axis.
But the geodesic axes of the hyperbolic transformations $g$
and $\gamma^{-1} g \gamma$ cannot share an ideal point since
$g, \gamma$ are not contained in a cyclic group.
\end{proof}

\begin{remark}
The arguments here show that the dynamics of the transverse pseudo-Anosov flow coarsely affects the dynamics of $f$. In particular, if $\ft$ is a translation with respect to a certain $\R$-covered branching foliation, there must be a lower bound on the topological entropy of $f$ depending only on the $\R$-covered branching foliation and the amount of translation of $\ft$. It is possible that in certain hyperbolic 3-manifolds one could control the possible geometries of $\R$-covered foliations, in which case one could find a uniform lower bound on the entropy of partially hyperbolic diffeomorphisms that act as translations on their branching foliations. If such a bound could be obtained, one could deduce that if the entropy of a partially hyperbolic diffeomorphism is sufficiently low, then the system must be a \emph{discretised Anosov flow}. \end{remark}

\section{Translations in hyperbolic $3$-manifolds} \label{sec-transl}
In this section we obtain further consequences of having a partially hyperbolic diffeomorphism act as a translation in a hyperbolic 3-manifold. 

We start by recalling the setting. Let $f\colon M \to M$ be a (not necessarily dynamically coherent) partially hyperbolic diffeomorphism on a hyperbolic $3$-manifold. Up to replacing $f$ by a power, we assume that it is homotopic to the identity. Up to taking a further iterate of $f$ and a lift to a finite cover of $M$, we can assume that $f$ admits branching foliations,
 and that the good lift
$\ft$ acts as a translation on the leaf space of $\wfbs$.

Let $\Phi^{cs}$ be a transverse regulating pseudo-Anosov flow to $\fbs$ given by Proposition \ref{prop-transversepA}.
This flow is fixed throughout the discussion.

Then Proposition \ref{fixedleaf_nonDC} shows that, for any periodic orbit of $\Phi^{cs}$, there exists a center stable leaf periodic by $f$.

\subsection{Periodic center rays}

We will now produce rays in periodic center leaves which are expanding. A \emph{ray} in $L$ is a proper embedding of $[0,\infty)$ into $L$. We say that a ray is a \emph{center ray} if it is contained in a center leaf. So a center ray $c_x$ is the closure in $L$ of a connected component of $c \smallsetminus \{x\}$ where $c$ is a center curve and $x \in c$.

Let $\gamma$ in $\pi_1(M)$ be associated with a
periodic orbit $\delta_0$ of the pseudo-Anosov flow $\Phi^{cs}$. Let $L$ be a leaf (given by Proposition \ref{fixedleaf_nonDC}) of $\wfbs$ fixed by $h:=\gamma^n \circ \ft^m$, with $m > 0$. 

A center ray $c_x$ is \emph{expanding} if $h(c_x) = c_x$ and
$x$ is the unique fixed point of $h$ in $c_x$ and every $y \in c_x \smallsetminus \{x\}$ verifies that $h^{-n}(y) \to x$ as $n\to +\infty$. It is \emph{contracting} if it is expanding for $h^{-1}$.

\begin{proposition}\label{coarseexpand}
Assume that a good lift $\ft$ of $f$ acts as a translation on the (branching) foliation $\wfbs$. Let $\Phi^{cs}$ be a regulating transverse pseudo-Anosov flow. 
 Let $\gamma$ in $\pi_1(M)$ associated with a
periodic orbit $\delta_0$ of $\Phi^{cs}$. Let $L$ be a leaf of $\wfbs$ fixed by $h=\gamma^n \circ \ft^m$,
where $m > 0$. 
Assume that $\gamma$ fixes all prongs of a lift of $\delta_0$ to $\mt$.
Then there are at least two center rays in $L$, fixed by $h$,
which are expanding. 
\end{proposition}

\begin{remark}
We should stress that we cannot guarantee to get 
a single center leaf with both rays expanding.
For example it is very easy to construct an example
such that $h$ has Lefschetz index $-1$ in $L$, it has
exactly 
$3$ fixed center leaves in $L$, and only two fixed 
expanding rays, which are contained in distinct 
center leaves (see Figure \ref{figureCancelation}).
This situation occurs in the examples constructed
in \cite{BGHP} in the unit tangent bundle of a surface.
\end{remark}

We will use Proposition \ref{coarseexpand} and its proof to eliminate the mixed behavior in hyperbolic 3-manifolds. It should be noted that this proposition also gives some relevant information about the structure of the enigmatic double translations examples which are not ruled out by our study.

The key point is to understand how each fixed center leaf contributes to the total Lefschetz index of the map in a center-stable leaf which we can control. Since the dynamics preserves foliations and one of them has a well understood dynamical behavior (i.e., in the center stable foliation, the stable foliation is contracting) we can compute the index just by looking at the dynamics in the center foliation (see Figure~\ref{figureIndex}).

As remarked above, one do have to be careful when computing the index as cancellations might happen with branching foliation (see Figure \ref{figureCancelation}).

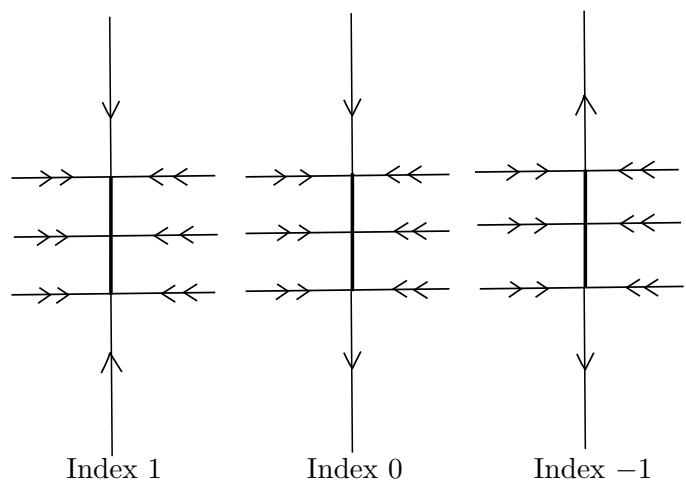
\begin{figure}[ht]
\begin{center}
\includegraphics[scale=0.8]{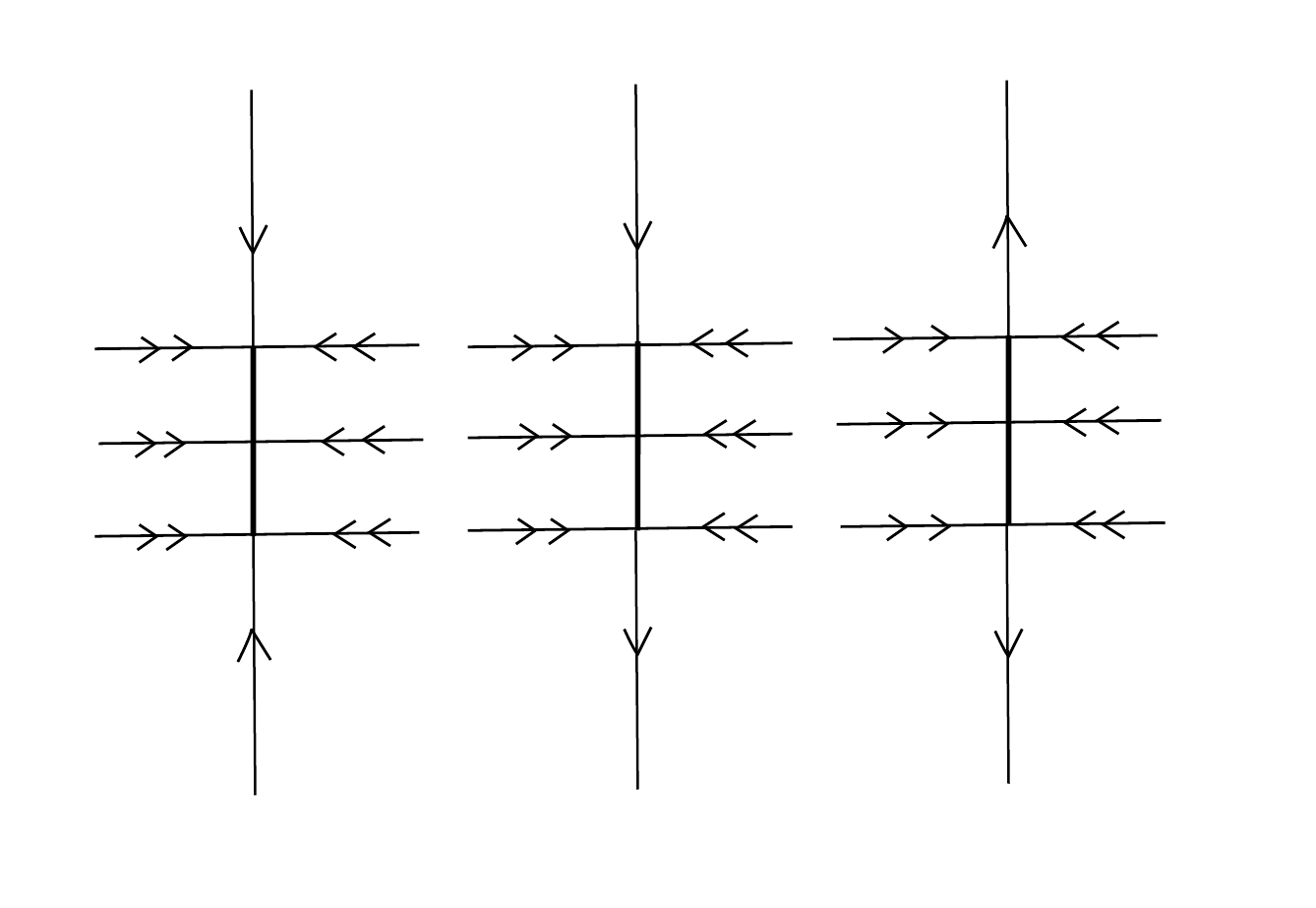}
\begin{picture}(0,0)
\put(-95,18){\text{Index }$-1$}
\put(-180,18){\text{Index }$0$}
\put(-270,18){\text{Index }$1$}
\end{picture}
\end{center}
\vspace{-0.5cm}
\caption{Contribution of index of a center arc depending on the center dynamics}\label{figureIndex}
\end{figure}

\begin{figure}[ht]
\begin{center}
\includegraphics[scale=0.8]{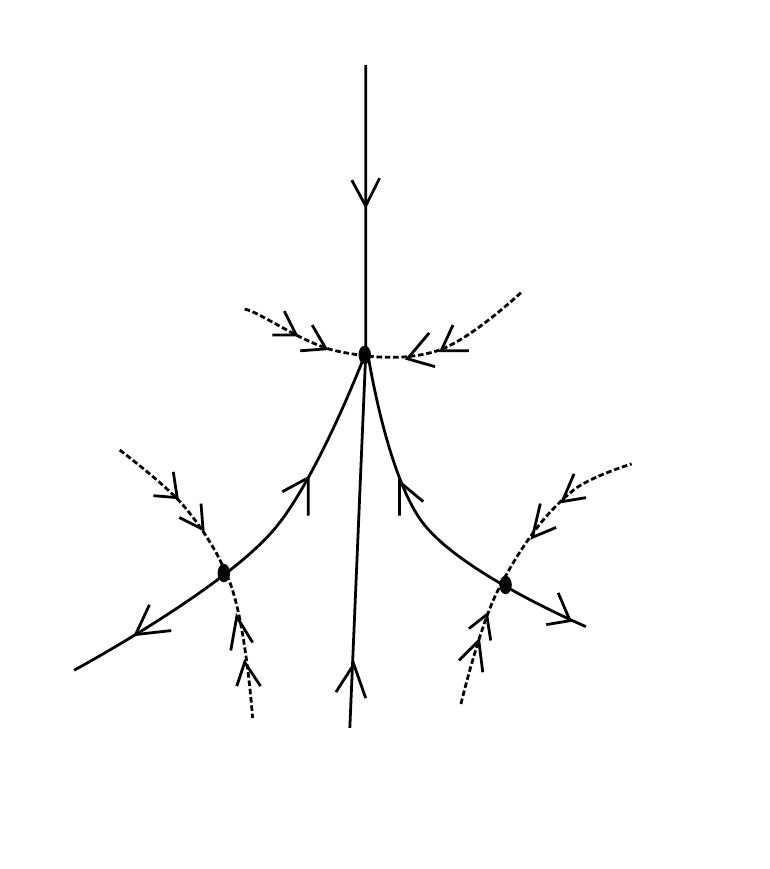}
\end{center}
\vspace{-0.5cm}
\caption{Two segments of zero index merge with a point with index 1 to produce a global -1 index.}\label{figureCancelation}
\end{figure}

We are now ready to give a proof of Proposition \ref{coarseexpand}.

\begin{proof}[Proof of Proposition \ref{coarseexpand}]
By Proposition \ref{p.homeotranslation_nonDC}, we know that the fixed point set of $h$ in $L$ is contained in the lift of $T_\gamma$ to $\mt$ (which intersects $L$ in  a compact set) and has Lefschetz index $1-p$
where $p$ is the number of stable prongs at the fixed 
point.  
In particular $h$ has some fixed points in $L$.

Let $L_2 = \ft^m(L)$. We denote by $\tau_{12}\colon L \to L_2$ the flow along $\wt\Phi^{cs}$ map.

\begin{claim} \label{claim_fixed_centers_match_on_interval}
 Let $c_1$, $c_2$ be two distinct center
leaves in $L$ that have a non-trivial intersection. Suppose that both 
$c_1, c_2$ are fixed by $h$, and there exist two distinct points $z, y \in c_1\cap c_2$ which are fixed by $h$.
Then the center leaves $c_1$ and $c_2$ coincide on the segment between $z$ and $y$.
\end{claim}

\begin{proof}[Proof of Claim \ref{claim_fixed_centers_match_on_interval}]
Let $[y,z]_{c_1}$ and $[y,z]_{c_2}$ be the center segments between $y$ and $z$ in $c_1$ and $c_2$ respectively.

Assume for a contradiction that $[y,z]_{c_1}$ and $[y,z]_{c_2}$ are distinct. Then, up to changing $y$ and $z$, we can assume that the intersection between the open intervals $(y,z)_{c_1}$ and $(y,z)_{c_2}$ is empty.

Thus, by construction, $[y,z]_{c_1}$ and $[y,z]_{c_2}$ intersect only at $z$ and $y$. We let $B$ be the bigon in $L$ bounded by $[y,z]_{c_1}$ and $[y,z]_{c_2}$.

Note that any stable leaf that enters the bigon $B$ must exit it (otherwise it would
limit in a stable leaf entirely contained in $B$, which is impossible). Hence, $B$ is ``product foliated" by stable
leaves. Since $B$ is compact the length of the stable segments
contained in $B$ is bounded.

Since $z,y$ are fixed by $h$ it follows that $B$ is also 
fixed by $h$.
Let $s$ be one such stable segment connecting $(z,y)_{c_1}$ to
$(z,y)_{c_2}$. Then, the images of $s$ under powers of $h^{-1}$ stay in $B$ but must also have unbounded length, contradiction.
\end{proof}

Let $x$ be a fixed point of $h$. Recall from Lemma \ref{l.center_leaf_space_in_L} that the set of center leaves through $x$ in $L$ is a closed interval. In particular $h$ fixes the endpoints of this interval. Hence, $x$ is contained
in a center leaf $c$ such that $h(c) = c$.

\begin{claim}\label{claim_fixed_points_in_finite_union}
 All the fixed points of $h$ in $L$ are contained in the union of finitely many compact segments of center leaves in $L$.
\end{claim}

\begin{proof}[Proof of Claim \ref{claim_fixed_points_in_finite_union}]
Let $c$ be a center leaf fixed by $h$. Since the fixed
points are contained in a compact set $C$ (see \cite[Lemma 8.11]{BFFP-prequel}), there is a 
minimal compact interval $J$ in $c$ which contains
all the fixed points of $h$ in $c$.

Suppose that there exists infinitely many distinct such minimal intervals $J_i$ in center leaves $c_i$.
Since the fixed points of $h$ in $L$ are in a compact set, we can choose $i,j$ large enough, so that $J_i$ is very close in the Hausdorff distance of $L$ to $J_j$. Let $z$ be an endpoint of $J_i$.
Then the stable leaf $s(z)$ through $z$ intersects the center leaf $c_j$. As $z$ is fixed by $h$ and so is $c_j$, contraction of the stable length implies that $z \in c_j$, thus $z\in J_j$.

Hence, both endpoints of $J_i$ are on $J_j$. By Claim \ref{claim_fixed_centers_match_on_interval}, it implies that $J_i \subset J_j$, and minimality of the interval $J_j$ implies $J_j=J_i$ which is a contradiction.
\end{proof}

Let $\{ J_i, 1 \leq i \leq i_0 \}$ be a finite family of compact intervals containing all the fixed point of $h$, as given by Claim \ref{claim_fixed_points_in_finite_union}. Note that we do not necessarily take the minimal intervals as constructed in the proof of Claim \ref{claim_fixed_points_in_finite_union}, as we want the following properties for that family.

\begin{claim}\label{claim_properties_for_J}
 We can choose the collection of intervals
$\{ J_i, 1 \leq i \leq i_0 \}$, each in a center leaf fixed
by $h$, satisfying
the following properties: 
\begin{enumerate}[label=(\arabic*)]
\item \label{item.all_fixed_points} The union $ \bigcup_{1\leq i\leq i_0}J_i $ contains all the fixed points
of $h$.
\item \label{item.fixed_endpoints} The endpoints of each interval $J_i$
are fixed by $h$.
\item\label{item.disjoint} The intervals are pairwise disjoint.
\end{enumerate} 
\end{claim}

\begin{proof}[Proof of Claim \ref{claim_properties_for_J}]
Let $c_1,\dots,c_n$ be a minimal collection of center leaves
that contains all fixed points of $h$ in $L$, as given by Claim \ref{claim_fixed_points_in_finite_union}. 
Let $J_i$ be the minimal compact interval containing all
fixed points of $h$ in $c_i$. 

The family $J_i$ then satisfies conditions \ref{item.all_fixed_points} and \ref{item.fixed_endpoints}. So we only have to show that one can split the intervals $J_i$ further so that conditions \ref{item.disjoint} is also satisfied (while still satisfying the first two conditions).

Notice that $c_i, c_j$ intersect if and only if 
$J_i, J_j$ intersect. Thus, we can restrict our attention to each connected component of the union of the $c_i$'s separately.

Up to renaming, assume that $\cup_{1\leq i\leq k}c_k$ is a connected component of $\cup_{1\leq i\leq n}c_k$.

Now we can consider the union of the $J_1,\dots, J_k$ as a graph, where the vertices are are the endpoints of the segments $J_i$ together with the points where two segments merge, and the edge are the subsegments joining the vertices.
With this convention, the union of the $J_1,\dots, J_k$ is then a tree. Otherwise
there would be a bigon in $L$ enclosed by the union, which is ruled out by Claim \ref{claim_fixed_centers_match_on_interval}.

Let $\mathcal B$ be this tree. Our goal is to remove enough open segments from the $J_i$'s so that no vertex of this associated tree has degree $3$ or more.
Consider a vertex $p$ in $\mathcal B$ with degree $3$ or more. Then there are two edges $e_1$ and $e_2$ abutting at $p$ on the same side of $p$.
We claim that $e_1$ cannot have points fixed by $h$ arbitrarily close
to $p$ (except for $p$ itself). Otherwise one would have a fixed point $y \in e_1$ such that $s(y)$ intersects $e_2$. Since $e_2$ is contained in a fixed leaf, $e_2 \cap s(y)$ is fixed by $h$. This implies (since $h$ decreases stable length) that $y$ is in $e_2$. Thus, by Claim \ref{claim_fixed_centers_match_on_interval}, the intersection of $e_1$ and $e_2$ would contain the segment $[y,p]$, contradicting the fact that they are distinct edges.

Thus, we can remove an open interval $(p,z)$ from, say, $e_1$, where $z$ is fixed by $h$ but $(p,z)$ has no fixed points. In the new tree, $p$ has index one less than before and $z$ has index one.

Doing this recursively on each vertex of index strictly greater than $2$, we will obtain, as sought, a disjoint collection of intervals that also satisfy conditions \ref{item.all_fixed_points} and \ref{item.fixed_endpoints}.
\end{proof}

Now we will look at the index of $h$ on the fixed intervals
$J_i$, $1 \leq i \leq i_0$ produced by Claim \ref{claim_properties_for_J}.
Note that for each such interval $J_i$ there are no other fixed points of 
$h$ nearby in $L$.
Let $c$ be a leaf fixed by $h$ containing $J_i$. 

If $h$ is contracting on $c$ near both endpoints
of $J_i$ on the outside then the index of $J_i$ is $+1$.
This is because the stable foliation is contracting
under $h = \gamma^n \circ \ft^m$ (since $m>0$). 
Hence $h$ is contracting near $J_i$.
If $h$ is expanding on both sides, the index is $-1$. If 
one side is contracting and the other is expanding then
the index is zero.

The global index for $h$ can then be computed by adding the indexes of $h$ on each of the intervals $J_i$, taking care of cancellations.

Let $c_k$, $1 \leq k \leq k_0$, be finitely many
center leaves, fixed by $h$ and containing all the
$J_i$. We choose this collection to have the
minimum possible number of leaves.

Each leaf $c_k$ contains finitely many segments $J_i$, so there are exactly two infinite rays that do not contain any $J_i$. The contribution of $c_k$ to the global index of $h$ (before possible cancellations) will then be $-1$ if both rays are expanding, $0$ if one is expanding while the other contracts and $1$ if both are contracting.

Suppose for a contradiction, that there is at most one expanding ray in $L$. So each $c_k$, considered separately, has index either $0$ or $1$.

If there is an expanding ray, let $c_k$ be a leaf with an expanding ray.
Otherwise let $c_k$ be any leaf. 
Now we need to consider how the other leaves and the possible cancellations impact the global index of $h$.
Let $c_l$ be a leaf that intersect $c_k$. If $c_l$ shares an expanding ray
with $c_k$, then the other ray of $c_l$ is contracting, and eventually
disjoint from the corresponding ray of $c_k$. The fixed set (if
any) of this ray in $c_l$ has index zero. If $c_l$ does not
share an expanding ray with $c_k$, then both rays of $c_l$
are contracting. The ray that is added to the same end
as the expanding ray of $c_k$ contributes index $1$. The
other ray contributes index $0$. In any case the index, starting
at $0$ or $1$, does not decrease.

Now, if $c_m$ is another leaf that is disjoint from the set above, then both rays are contracting and it contributes an index $1$. So again the index does not decrease.

Thus, if there is at most one expanding ray, then the index of $h$ is at least $0$. This contradicts the fact that the index of $h$ is $1-p$ where $p\geq 2$, and thus finishes the proof of Proposition \ref{coarseexpand}.
\end{proof}

\subsection{Periodic rays and boundary dynamics}\label{ss.boundary}

Proposition \ref{coarseexpand} gave the existence of periodic rays that are coarsely expanding. Here we will show that such a ray has a well-defined ideal point on the circle at infinity of the leaf, and that it corresponds to the endpoint of a prong of the transverse regulating pseudo-Anosov flow, $\Phi^{cs}$.

As previously, we assume that we have a center stable leaf $L \in \wfbs$ such that there is a deck transformation $\gamma$ for which $\gamma \circ \ft^m (L)=L$ for some $m>0$.  We let $L_2 = \ft^m(L)$ and define $\tau_{12}\colon L \to L_2$ the flow along $\wt\Phi^{cs}$ map. We also take as before

$$ h := \gamma \circ \ft^m \  \text{ and } \  g := \gamma \circ \tau_{12}. $$

Recall that $h$ and $g$ are maps of $L$ that are a bounded distance from each other. Also $g$ preserves the (singular) foliations $\cG^s$ and $\cG^u$. We again assume that if $g$ has a fixed point $x_0$ in $L$ then $\gamma$ is such that $g$ preserves each of the prongs of $\cG^s(x_0)$ (resp.~$\cG^u(x_0)$).

The action of $g$ on the circle at infinity $S^1(L_1)$ has an even number of fixed
points, which are alternately attracting and repelling. We denote by $P$ the set of attracting fixed points and by $N$ the set of repelling ones. With these notations, we get the following.

\begin{proposition}\label{contract}
Let $\eta\colon [0,\infty) \to L$ be a contracting fixed ray for $h$. Then $\lim_{t \to \infty} \eta(t)$ exists in $S^1(L)$ and it is a (unique) point in $N$. (Symmetrically, if $\eta$ is an expanding fixed ray, its limit point belongs to $P$.)
\end{proposition}

\begin{proof}
Let $y$ in $P$ and $U$ a small neighborhood of $y$ in $L \cup S^1(L)$ as in \cite[\S 8]{BFFP-prequel}. 
If $\eta$ has a point $q$
in $U \cap L$, then $h^n(q)$ converges to $y$ as $n \rightarrow +\infty$,
so $\eta$ could not be a contracting ray, a contradiction.
So $\eta$ cannot limit on any point in $P$. If $z$ is in 
$S^1(L) \smallsetminus \{ N \cup P \}$, then $h^n(z)$ converges to a point
in $P$ under forward iteration. Hence again a small neighborhood
$Z$ of $z$ in $L \cup S^1(L)$ is sent under some iterate
inside a neighborhood $U$ as in the first part of the proof.
So any point in $Z \cap L$ converges to a point in $P$ under
forward iteration. Hence $\eta$ cannot limit
to a point in $S^1(L) \smallsetminus \{ N \cup P \}$ either.
So $\eta$ can only limit on points in $N$. Since $\eta$
is properly embedded in $L$, the set
of accumulations points of $\eta$ is connected, so it has
to be a single point.
\end{proof}

\section{Mixed case in hyperbolic manifolds} \label{sec-mix.hyp}


In this section we show that even in the non-dynamically coherent case, the mixed behavior is impossible for hyperbolic $3$-manifolds. This will be done by using the study of translations in hyperbolic $3$-manifolds developed in sections \ref{s.pA_and_translations_nonDC} and \ref{sec-transl} to provide more information on the dynamics of general partially hyperbolic diffeomorphisms.

The main result of this section is the following.

\begin{theorem}\label{mixednonDC}
Let $f\colon M \to M$ be a partially hyperbolic diffeomorphism homotopic
to the identity on a hyperbolic 3-manifold $M$. Suppose that there exists a finite lift and finite power $\hat f$ of $f$ that preserves two branching foliations $\fbs,\fbu$ and is such that a good lift $\ft$ fixes a leaf of $\wfbu$. Then, $f$ is a discretized Anosov flow. 
\end{theorem}

This, together with Proposition~\ref{p.hypSeifminimal_nonDC}, completes Theorem~\ref{teo-hypseif}.

\subsection{The set up} 
Consider a partially hyperbolic diffeomorphism $f$ as in Theorem \ref{mixednonDC}. 

Our goal is to show that the good lift $\ft$ of $f$ fixes every leaf of $\wfbs, \wfbu$. Indeed, Proposition \ref{p.nondceverycfixed} (and Corollary \ref{coro.fixcentercoh}) then implies that $\hat f$ is dynamically coherent, so we can then use \cite[Theorem B]{BFFP-prequel} to obtain that $\hat f$ is a discretized Anosov flow.
In turns, thanks to Proposition \ref{coro-DCwithoutlift}, we obtain that $f$ itself is dynamically coherent and a discretized Anosov flow.

Since Proposition \ref{coro-DCwithoutlift} allows us to use finite lifts and powers, we assume directly that $f = \hat f$, that $\fbs$ and $\fbu$ are orientable and transversely orientable and that $f$ preserves their orientations. 

Since $\ft$ is assumed to fix one leaf of $\wfbu$, Proposition \ref{p.hypSeifminimal_nonDC} implies that every leaf of $\wfbu$ is fixed. We will prove that every leaf of $\wfbs$ is fixed by $\ft$ by contradiction. So, by Proposition \ref{p.hypSeifminimal_nonDC}, we can assume that $\fbs$ is $\R$-covered and uniform and that $\ft$ acts as a translation on the leaf space of $\wfbs$. In particular, there are no center curves fixed by $\ft$. 

Then, we can apply Proposition \ref{p.alternnonDC} 
to $\fbu$ to deduce that every periodic center leaf is coarsely expanding.

On the other hand, since $\ft$ acts as a translation on $\wfbs$, we can use the results from sections \ref{s.pA_and_translations_nonDC} and \ref{sec-transl}.
Let $\Phi^{cs}$ be a regulating pseudo-Anosov flow transverse to $\fbs$ given by Proposition \ref{prop-transversepA}. 

The flow $\Phi^{cs}$ is a genuine pseudo-Anosov, that is it admits at least one periodic orbit which is a $p$-prong with $p \geq 3$ (see \cite[Proposition D.4]{BFFP-prequel}).  

Now, we choose $\gamma$ in $\pi_1(M)$, associated to this prong, and apply Proposition \ref{fixedleaf_nonDC}: Up to taking powers, we can assume that $h := \gamma \circ \ft^k$ for some $k>0$ fixes a leaf $L$ of $\wfbs$. Moreover, the dynamics in $L$ resembles that of the dynamics of a $p$-prong,
and in particular fixes every prong.

Notice that Proposition \ref{coarseexpand} also provides some center rays which are expanding in $L$ for $h$. We will need to use some of the ideas involved in the proof of that proposition (even though the statement itself will not be used). 

We summarize the discussion above in the following proposition.

\begin{proposition}\label{p.mixedhyp} 
Let $f\colon M \to M$ be a partially hyperbolic diffeomorphism homotopic
to the identity of a hyperbolic 3-manifold $M$ preserving branching foliations $\fbs,\fbu$. Suppose that a good lift $\ft$ fixes a leaf of $\wfbu$ and acts as a translation on $\wfbs$. Then, up to taking finite iterates 
and covers, there exists $\gamma \in \pi_1(M)$ and $k>0$ such that a center stable leaf $L \in \wfbs$ is fixed by $h:= \gamma \circ \ft^k$ and its Lefschetz index is $I_{\mathrm{Fix}(h)}(h) = 1-p$ with $p\geq 3$. Moreover, every center curve fixed by $h$ in $L$ is coarsely expanding. 
\end{proposition}

Let $\gamma$ be as in the proposition.
Let $L$ be a center stable leaf fixed by $h= \gamma \circ \ft^k$ and $L_2=\ft^k(L)$. As previously, we write $\tau_{12} \colon L \to L_2$ for the map obtained by flowing from $L$ to $L_2$ along $\wt\Phi^{cs}$. We set $g := \gamma \circ \tau_{12}$. 

The map $g$ acts on the compactification of $L$ with its ideal circle $L \cup S^1(L)$ the same way as $h$ does (see sections \ref{s.pA_and_translations_nonDC} and \ref{sec-transl}). 

Let $\delta$ be the unique orbit of $\wt\Phi^{cs}$ fixed by $\gamma$ and let $x$ be the (unique) intersection of $\delta$ with $L$. Note that $x$ is the unique fixed point of $g$.
Since we assume that $\gamma$ fixes the prongs of $\delta$, then $h$ has 
exactly $2p$ fixed points in $S^1(L)$. These fixed points are contracting if they correspond to an ideal point of $\cG^u(x)$ and expanding if they are ideal points of $\cG^s(x)$.

\subsection{Proof of Theorem \ref{mixednonDC}} 

To prove Theorem \ref{mixednonDC} we will first show some properties. Recall from Proposition \ref{contract} that every proper ray in $L \in \wfbs$, 
fixed by $h$ has a unique limit point in $S^1(L)$ (notice that the ray must be either expanding or contracting). We will show that the fixed rays associated to the center and stable (branching) foliations have different limit points at infinity. 

\begin{lemma}\label{claim8}
Let $s$ be a stable leaf in $L$ which
is fixed by $h$. Then the two rays of $s$
limit to distinct ideal points of $L$.
The same holds if $c$ is a center leaf in $L$ fixed
by $h$. 
\end{lemma}

\begin{proof}
We do the proof for the center leaf $c$, the one for stable
leaves is analogous, and a little bit easier (since there is no branching).

By hypothesis, $c$ is fixed by $h$, hence it is coarsely expanding 
under $h$. It follows that there are fixed points of
$h$ in $c$. By Proposition \ref{contract} each ray of $c$
can only limit in a point in $P \subset S^1(L)$,
where, as previously, $P$ is the set of attracting fixed points of $h$
in $S^1(L)$.
Let $q_1$, $q_2$ be the ideal points of the rays. What we
have to prove is that $q_1$ and $q_2$ are distinct.

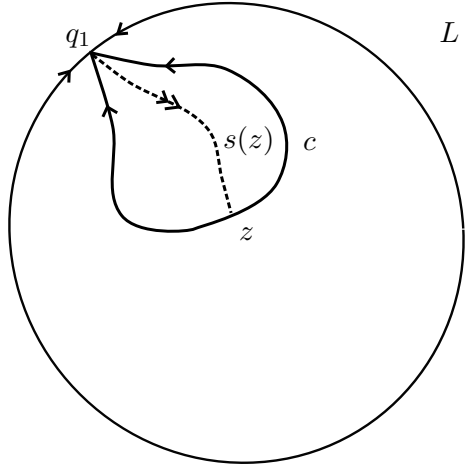
\begin{figure}[ht]
\begin{center}
\includegraphics[scale=1.4]{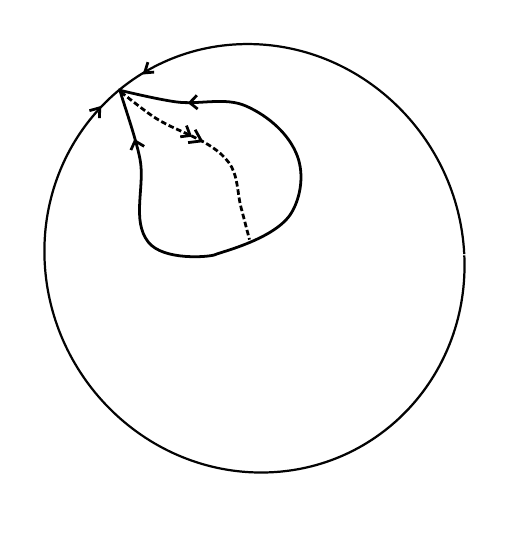}
\begin{picture}(0,0)
\put(-35,185){$L$}
\put(-175,185){$q_1$}
\put(-110,110){$z$}
\put(-86,144){$c$}
\put(-116,145){$s(z)$}
\end{picture}
\end{center}
\vspace{-0.5cm}
\caption{{\small Rays have to land in different points of $S^1(L)$.}}\label{fig-lema153}
\end{figure}

Suppose that $q_1 = q_2$. Then $c$ bounds a unique region
$S$ in $L$ which limits only in $q_1\in S^1(L)$. The other
complementary region of $c$ in $L$ limits to every point in
$S^1(L)$.
Let $z$ be a fixed point
of $h$ in $c$. Then the stable leaf $s(z)$ of $z$ has
a ray $s_1$ entering $S$.  It cannot intersect $c$ again,
and it is properly embedded in $L$. Hence it has to limit
in $q_1$ as well. See Figure~\ref{fig-lema153}.

 But now this ray is contracting for $h$.
This contradicts Proposition \ref{contract} because this ray
should limit in a point of $N$. 
\end{proof}

\begin{remark}
The proof used strongly that periodic center leafs are coarsely expanding, in order to induce a behavior at infinity. In the examples of \cite{BGHP} it does happen that different stable curves land in the same ideal point at infinity in their center stable leaf.  
\end{remark}

Now we show a sort of dynamical coherence for fixed center rays. 

\begin{lemma}\label{claim9} 
Suppose that $c_1, c_2$ are distinct
center leaves in $L$ which are fixed by $h$. Then
$c_1, c_2$ cannot intersect.
\end{lemma} 

Notice that since $f$ is not necessarily dynamically coherent,
the distinct center leaves
$c_1, c_2$ can a priori intersect each other.
The proof will depend very strongly on the fact that 
center rays fixed by $h$ are coarsely expanding.

\begin{proof}
Suppose that $c_1, c_2$ intersect.
Since $c_1, c_2$ are both fixed by $h$, so is their intersection.
Since $h$ is coarsely expanding in each, then $c_1, c_2$
share a fixed point of $h$. 
In the the proof of Claim \ref{claim_fixed_centers_match_on_interval}, 
we showed that $c_1$ and $c_2$ cannot form a bigon $B$. 

It follows that there is a point $x$, fixed by $h$, which is an
endpoint of all intersections of $c_1$ and $c_2$: On
one side $x$ bounds a ray  $e_1$ of $c_1$ and a ray $e_2$
of $c_2$ such that $e_1$ and $e_2$ are disjoint.
For a point $y$ in $e_1$ near enough to $x$, we have that $s(y)$ must intersects $c_2$. Since
stable lengths are contracting under powers of $h$, it implies that $e_1$ is contracting towards $x$ near
$x$ and similarly for $e_2$ (see figure~\ref{fig-lema154}). But $e_1$ is coarsely expanding. Hence there must exist fixed points of $h$ in $e_1$. Let $y\in e_1$ be the closest point to $x$ which is fixed by $h$. Similarly, let $z$ in
$e_2$ closest to $x$ fixed by $h$.

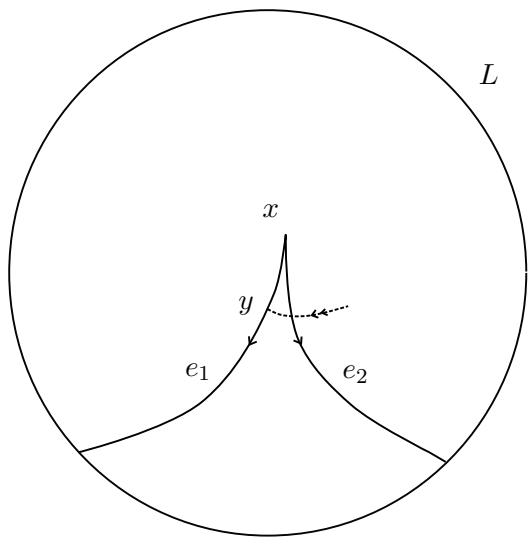
\begin{figure}[ht]
\begin{center}
\includegraphics[scale=0.9]{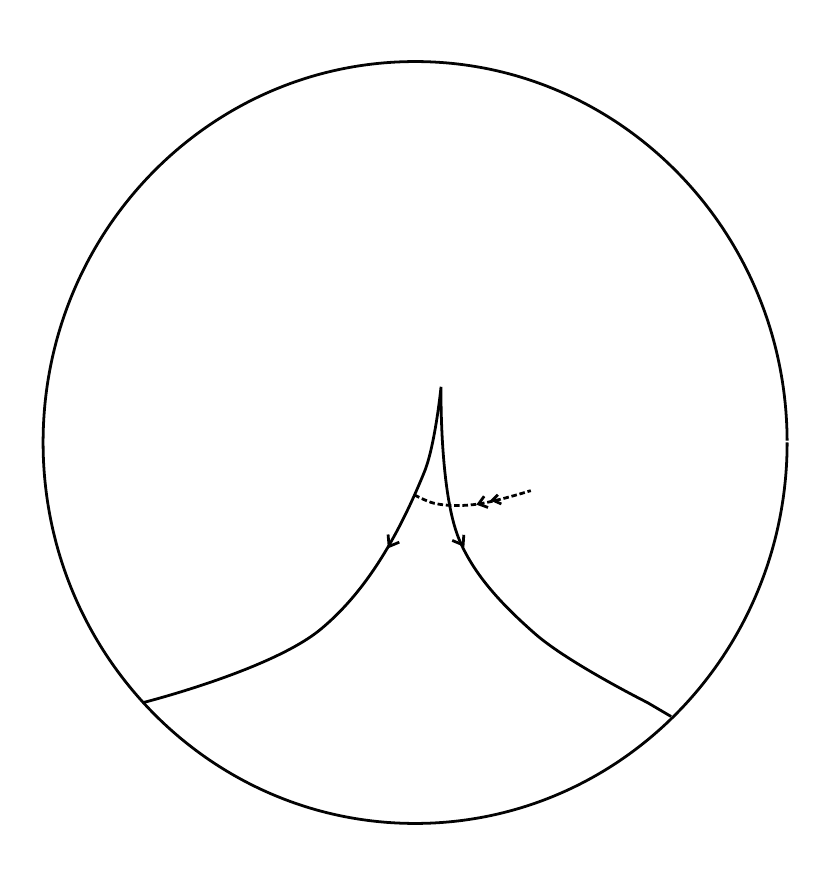}
\begin{picture}(0,0)
\put(-35,185){$L$}
\put(-145,75){$e_1$}
\put(-125,100){$y$}
\put(-86,74){$e_2$}
\put(-116,135){$x$}
\end{picture}
\end{center}
\vspace{-0.5cm}
\caption{{\small Showing the existence of fixed points below $x$ in Lemma \ref{claim9}.}}\label{fig-lema154}
\end{figure}

The leaves $s(y)$, $s(z)$ are not separated from each other
in the stable leaf space in $L$. 

Let now $c$ be a center leaf
through $x$, which is between $c_1$ and $c_2$ and which
is the first center leaf not intersecting $s(y)$. 

Then
$h(c) = c$ since $s(y)$ is fixed and $c$ is
the first leaf through $x$ not intersecting $s(y)$. Consider the ray of
$c$ starting at $x$ and moving in the direction of $y$. This ray is the limit
of compact center segments from $x$ to points in $s(y)$. As such this ray
of $c$ 
can only intersect stable leaves which are between $s(x)$ and
$s(y)$. Because the map $h$ contracts stable lengths it follows that
the map $h$ is contracting in this ray of $c$. This contradicts
Proposition \ref{p.mixedhyp} because this ray is in a center leaf which is
fixed by $h$.
\end{proof}

Thus far, we showed that distinct center leaves
in $L$, which are fixed by $h$ do not intersect.
Then, the proof of Claim~\ref{claim_fixed_points_in_finite_union} also implies that fixed center leaves cannot 
accumulate (as accumulation would imply that some fixed leaves intersect).

We conclude that there are finitely many center leaves in $L$
that are fixed under $h$. Each such center leaf is
coarsely expanding. For each such center leaf $c$, we consider a small enough open topological disk containing 
all the fixed points of $h$ in $c$,  and no other fixed
point of $h$ in $L$.
Then, on such disks, the Lefschetz
index of $h$ is $-1$. 
Since the total Lefschetz number of $h$ in $L$ is
$1 - p$ it follows that:

\begin{lemma}\label{conclusion3}
There are exactly $p  - 1$ center leaves which are
fixed by $h$ in $L$.
\end{lemma}

This together with the following lemma will allow us to make a counting argument to reach a contradiction.  

\begin{lemma}\label{conclusion4}
Let $c_1$, $c_2$ be two distinct center leaves
in $L$ fixed by $h$. Let $y_1\in c_1$ and $y_2\in c_2$ be fixed points of $h$. Then $s(y_1)$ and $s(y_2)$ do not have common ideal points.
\end{lemma}

\begin{proof}
Suppose, for a contradiction, that there are distinct fixed center
leaves $c_1$, $c_2$ satisfying the following: There 
are points $y_1\in c_1$ and $y_2\in c_2$, fixed by $h$,
such that $s_1 = s(y_1)$ and $s_2 = s(y_2)$ share an ideal
point in $S^1(L)$.

Let $q$ be the common ideal point of the corresponding
rays of $s_1$ and $s_2$. Note that  by Proposition \ref{p.mixedhyp} the point $q$ cannot be an endpoint of $c_1$ or $c_2$, because ideal
points of fixed centers are contracting in $S^1(L)$ and
ideal points of fixed stables are repelling in $S^1(L)$. 

Let $e_j$ be the ray in $s_j$ with
endpoint $y_j$ and ideal point $q$.
Suppose first that no center leaf intersecting $e_1$
intersects $e_2$.
Let $c_0$ be a center leaf intersecting $e_1$. Iterate
$c_0$ by powers of $h^{-1}$. It pushes points in $s_1$
away from $y_1$. Since the leaves $h^{-i}(c_0)$ all intersect $s_1$ and none of them intersects $s_2$ or $c_2$, the sequence $(h^{-i}(c_0))$
converges to a collection of center leaves as $i \rightarrow +\infty$. Then there is only one center leaf in this limit,
call it $c$, which separates
all of $h^{-i}(c_0)$ from $s_2$. This $c$ is invariant
under $h$ and it has an ideal point in $q$ because it separates $h^{-i}(c_0)$ (recall that $h^{-i}(c_0) \cap s_1 \to q$ as $i\to \infty$) from $s_2$. Now $q$ is a repelling fixed point in $S^1(L)$,
so $c$ must have an attracting ray, a contradiction with Proposition \ref{p.mixedhyp}.

\begin{figure}[ht]
\begin{center}
\includegraphics[scale=0.6]{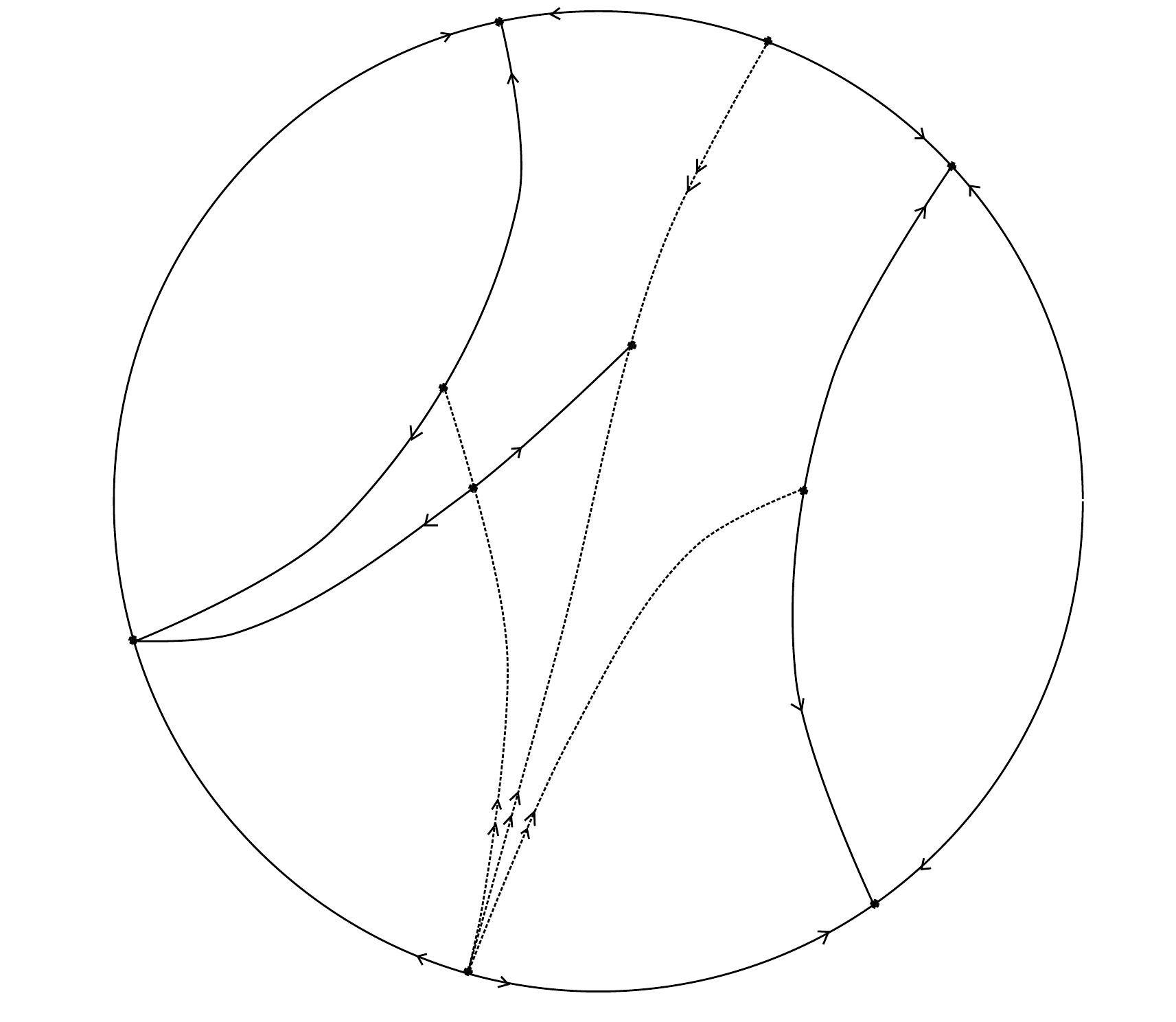}
\begin{picture}(0,0)
\put(-125,199){$s$}
\put(-195,199){$c_1$}
\put(-92,130){$y_2$}
\put(-76,174){$c_2$}
\put(-136,169){$y$}
\put(-199,160){$y_1$}
\put(-179,130){$z \neq y_1,y$}
\put(-146,79){$e_2$}
\put(-186,89){$e_1$}
\put(-186,0){$q$}
\end{picture}
\end{center}
\vspace{-0.5cm}
\caption{{\small A depiction of the main objects in the proof of Lemma~\ref{conclusion4}.}}\label{fig-lema156}
\end{figure}
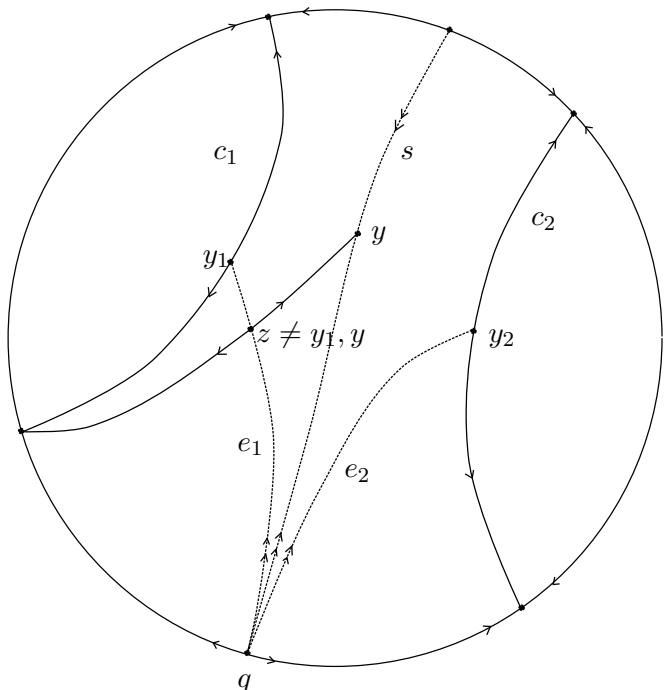

It follows that some center leaf intersecting $e_1$
also intersects $e_2$. Let $c_0$ be one such center leaf.
Now iterate by positive powers of $h$. Then $(h^i(c_0))$ converges
to a fixed center leaf $v_1$ through $y_1$ and a fixed
center leaf $v_2$  through $y_2$. But then $v_1$ and $c_1$ 
are both fixed by $h$ and both contain $y_1$. Lemma \ref{claim9}
implies that $c_1 = v_1$ and $c_2 = v_2$. In particular
$v_1 \not = v_2$, and they are non separated from each 
other. In this case, consider $s$ the unique stable
leaf defined as the first leaf not intersecting $c_1$ that separates $s_1$
from $s_2$.
Then, as above, $h$ fixes $s$ and has a fixed point
$y$ in $s$. But a center leaf $c$ through $y$ fixed by $h$ has
to intersect the interior of the ray $e_1$. This intersection point is 
the intersection of $c$ fixed by $h$, and $s_1$ fixed
by $h$. So this intersection point is fixed by $h$.
But this is a contradiction, because $y_1$ is the
only fixed point of $h$ in $s_1$. So Lemma \ref{conclusion4} is proven.
\end{proof}

We now can complete the proof of Theorem \ref{mixednonDC}.

\begin{proof}[Proof of Theorem \ref{mixednonDC}] 

By Lemma \ref{conclusion3}, there are $p - 1$ center leaves fixed
by $h$ in $L$. We denote them by $c_1, \dots,c_{p-1}$.

Each center leaf has at least 
one fixed point. Let  $y_i$, $1 \leq i \leq p-1$ be a fixed
point in $c_i$.
Then, for each $i$, Lemma \ref{claim8} states that $s(y_i)$ has two distinct ideal points $z^1_i$ and $z^2_i$.

Moreover, for every $i \not = j$, the ideal points of the stable leaves are distinct by Lemma \ref{conclusion4}. It follows that there are at least
$2p -2$ distinct points in $S^1(L)$ which are repelling.

But we also know that there are exactly $p$ points 
in $S^1(L)$ that are repelling under $h$. 
It follows that $2p -2 \leq p$, which implies $p =2$. However, we had that $p\geq 3$, thus obtaining a contradiction.

This finishes the proof of Theorem \ref{mixednonDC}.
\end{proof}

\begin{appendix}


\section{Some $3$-manifold topology} \label{app.3_manifold_topology}

Besides the 3-manifold topology presented in \cite[Appendix A]{BFFP-prequel} we will need an additional result 
important to understand certain particular deck transformations when one lifts to finite covers. 

\begin{lemma}{}\label{boundcov}
Let $M$ be a closed, irreducible $3$-manifold with
fundamental group that is not virtually nilpotent.
Suppose that $\beta$ is a non trivial deck transformation so that
$d(x,\beta(x))$ is bounded above in $\mt$.
Then $M$ is a Seifert fibered space 
and $\beta$ represents a power
of a regular fiber.
\end{lemma}

\begin{proof}{}
First we assume that $M$ is orientable. 
Then, the JSJ decomposition states that $M$ has a canonical 
decomposition into Seifert fibered and geometrically
 atoroidal pieces.
 We lift this to a decomposition of $\mt$ and construct
a tree $\cT$ in the following way: The vertices are the lifts of components of the
torus decomposition of $M$, and we associate an edge if
two components intersect along the lift of a torus.
Such a lift of a torus is called a wall.
There is a minimum separation distance between any two walls.

The deck transformation $\beta$ acts on this tree.
Let $W$ be a wall. Suppose that $\beta(W)$ is distinct 
from $W$. But, as subsets of $\mt$, the walls
$W, \beta(W)$ are a finite Hausdorff distance from
each other. Then $\pi(W)$, $\pi(\beta(W))$ are tori
in $M$, and the region $V$ in $\mt$ between
$W, \beta(W)$ projects to $\pi(V)$ which is $\T^2 \times [0,1]$
in $M$. 
If this happens then $M$ is a torus bundle over a circle.
In that case, use that $\pi_1(M)$ is not virtually nilpotent, so
the monodromy of the fibration is an Anosov map of $\T^2$.
But then no $\beta$ as above could satisfy the bounded
distance property. 
It follows that $\beta(W) = W$ for any wall, and in
particular $\beta(P) = P$ for any vertex of $\cT$.

Now consider a vertex $P$. Suppose first that $\pi(P)$ is 
homotopically atoroidal.
By the Geometrization Theorem, $\pi(P)$ is hyperbolic.
If $\beta$ restricted to $P$ were to satisfy the bounded distance
property, then it would have to be the identity on $P$. Hence $\beta$
itself is the identity, contradiction.

Hence all the pieces of the torus decomposition of $M$
are homotopically toroidal. 
Suppose now that there is one such piece $\pi(P)$ that is geometrically atoroidal (but
not homotopically atoroidal). The proof of the Seifert
fibered conjecture (\cite{Ca-Ju,Gab92}) shows that $\pi(P)$ has no boundary
and $\pi(P)$ is Seifert. In other words, $M=\pi(P)$ is Seifert.
So we can assume that all the pieces of the torus
decomposition are geometrically toroidal. Then they are all
Seifert fibered. Thus $M$ is a graph manifold.

We will show that the torus decomposition of $M$
is in fact trivial, proving that $M$ is Seifert fibered.
Suppose it is not true. Then the tree ${\cT}$ is
infinite. 
Let $P_1, P_2, P_3$ be three consecutive vertices in ${\cT}$.
Let $W_1$ be the wall between $P_1$ and $P_2$. Then 
$\beta(W_1)$ (as a set in $\mt$) is a bounded distance from $W_1$ and
sends the Seifert fibration of $P$ in $W_1$ to lifts
of Seifert fibers. It follows that $\beta = \delta_1^k \alpha_1$
where $\delta_1$ represents a regular fiber in $\pi(P_1)$,
and $\alpha_1$ is a loop in $\pi(W_1)$. Similarly
if $W_2$ is the wall between $P_2$ and $P_3$ then 
$\beta = \delta_3^i \alpha_3$ where $\alpha_3$ is a loop
in $\pi(W_3)$. Then $\alpha_1, \alpha_3$ are both in the
boundary of $\pi(P_2)$. The loops representing 
$\delta_1^k \alpha_1$, $\delta_3^i \alpha_3$ are both
in the boundary of $\pi(P_2)$. They represent the same element
of $\pi_1(M)$ only when $k = i = 0$ and $\alpha_1, \alpha_3$
are freely homotopic. That means that $P_2$ is a torus times
an interval, which is impossible in the torus decomposition
in our situation as explained above.

It follows now that the torus decomposition of $M$ is
trivial, which implies that $M$ is Seifert fibered.
Moreover, if the base is not hyperbolic, then $\pi_1(M)$ is
virtually nilpotent (\cite[Theorem 5.3]{Sco83}).
But this contradicts the hypothesis of the lemma.

It follows that the base is hyperbolic. Also
$\beta$ induces
a transformation in the universal cover of the base
that is a bounded distance from the identity.
This can only happen if this transformation is the
identity. Therefore $\beta$ represents a power
of a regular Seifert fiber in $M$ (notice that non-regular fibers induce a finite symmetry on the base, thus not the identity, and not a bounded
distance from the identity).

So the Lemma is proven when $M$ is orientable. If $M$ is not orientable, then it has a double
cover $M_2$ which is orientable. Now $\beta^2$ lifts to an element of $\pi_1(M_2)$ that satisfies the assumption of the lemma. So we can apply the result to $M_2$ and obtain that $M_2$ is Seifert.
Thus $M$ is doubly covered by a Seifert space, which, by a 
result of Tollefson \cite{Tol}, implies that $M$ itself is Seifert fibered.
It follows that
$\beta$ corresponds to a power
of a regular fiber.
This finishes the proof of the lemma.
\end{proof}

\section{Minimality and $f$-minimality} \label{app.partial_hyperbolicity}

We prove that in certain situations minimality is equivalent to
$f$-minimality. We need the following result which is of interest
in itself.

\begin{lemma}\label{closedsets}
Let $\lcsb$ be the leaf space of $\wfbs$. Let $\cB \subset \lcsb$ be a closed set of leaves.
Suppose that, for all $x\in \mt$, there exists a leaf $L\in \cB$ containing $x$.
Then $\cB = \lcsb$.
\end{lemma}

\begin{proof}
The lemma is obvious when $\fbs$ is a true foliation (and one does not need to require $\cB$ to be closed). However, when $\fbs$ has some branching, one could possibly have a union of leaves that cover all of $\mt$ without using all the leaves of $\wfbs$. For closed sets of leaves we show this is not possible.

Let $L$ be a leaf of $\wfbs$, $x$ a point in $L$ and
$\tau$ an open unstable segment through $x$. 
The set of leaves of $\wfbs$ intersecting $\tau$ is isomorphic
to an open interval. Using the transversal orientation to $\wfbs$, we can put an order on this interval.

By our assumption, every point in $\tau$ intersects a leaf in $\mathcal B$. Let $L'$ be the supremum of leaves in $\cB$, intersecting $\tau$ and smaller than 
or equal to $L$. Since $\cB$ is closed, we have $L'\in \cB$. Notice that $x$ is in both $L$ and  $L'$.

We claim that $L' = L$. 
If $L$ is not equal to $L'$ then they branch out. Let $y$ be 
a boundary point of $L \cap L'$. 
Let $z \in L'$, with $z \notin L$ be close enough to $y$ so that its unstable leaf $u(z)$ intersects $L$. Now take any point $w\in u(z)$ in between $z$ and $L \cap u(z)$. Any leaf $L_1 \in \wfbs$ that contains $w$ must contain $y$. Hence (because leaves do not cross), $L_1$ also contains $x$. By definition, it is above $L'$, thus $L_1$ is not in $\cB$. Since this is true for any leaf through $w$, it contradicts our assumption.
\end{proof}

\begin{lemma}\label{fmin}
When $\fbs$ does not have compact leaves, then $f$-minimality
of $\fbs$ is equivalent to minimality of $\fbs$.
\end{lemma}

\begin{proof}
Note that minimality obviously implies $f$-minimality, so we only need to show the other implication.

Suppose that $\fbs$ is not minimal and let
$C$ be the union of a set of $\fbs$  leaves which is closed
and not $M$. 
Let $\fes$ be an approximating foliation, with
approximating map $\hs$ sending leaves of $\fes$ to
those of $\cW^{cs}$. 
Then $(\hs)^{-1}(C)$ is a set which is a union of $\fes$ 
leaves, which is closed and not $M$.
In particular it contains an exceptional minimal set $D$. 
By \cite[Theorem 4.1.3]{He-Hi}
the actual foliation $\fes$ has finitely many
exceptional
minimal sets $B_1,\dots,B_k$. The union $B$ of these is not
$M$ because $D \not = M$. The set of leaves in $B$ is a closed
set of leaves denoted by $\mathcal B$.  
Then $A = \hs(B)$ is a closed subset of $M$, and
$\mathcal A = \hs(\mathcal B)$ is a closed set of leaves,
being the image by $\hs$ of the leaves in $\mathcal B$. 
Let $\wt{\mathcal A} = \pi^{-1}(\mathcal A)$, we stress that this
is on the leaf space level, not in terms of sets.
This is a closed subset of $\lcsb$.

Let $A_i := \hs(B_i)$. Every leaf of $\fbs$ which is the image
of a leaf in $B_i$ is dense in $A_i$. 
Using this, it is easy to see that $f(A) = A$.
By $f$-minimality it follows that $A = M$.

Since $A = M$ then  $\wt{\mathcal A}$ is a closed subset of $\lcsb$,
whose union of points in all leaves of $\wt{\mathcal A}$ is 
$\mt$ as $A = M$.
Lemma \ref{closedsets} implies that $\wt{\mathcal A} = \lcsb$.
Hence for each leaf $E$ of $\fbs$, it is the image of
a leaf $F$ in some $B_i$.
Conversely every leaf of $\fes$ maps by $\hs$ to a leaf
of $\fbs$. 

For each leaf $E$ of $\fbs$, its preimage $(\hs)^{-1}(E)$ is
a closed interval of leaves of $\fes$. No leaf in
the interior of the interval can be in a $B_i$ as it
is a minimal set. 
It follows that 
 the complementary regions of $B$ in $M$
are $I$-bundles. These can be collapsed to generate another
foliation $\mathcal C$. Since the $B_i$ were minimal sets
of $\fes$, then the collapsing of each of these is a minimal
set of $\mathcal C$. Since the union is all of $M$,
there can be only one such minimal set, so $\fes$ is
minimal. 

But this contradicts the fact that $D$ is an exceptional
minimal set of $\fes$.
\end{proof}

We state the following criteria for dynamical coherence (which in this setting is quite obvious).

\begin{proposition}[Proposition 1.6 and Remark 1.10 in \cite{BW}]\label{p.BWcoh} 
Assume that $f$ is a partially hyperbolic diffeomorphism admitting branching foliations $\fbs$ and $\fbu$. If no two distinct leaves of $\fbs$ or $\fbu$ intersect, then $f$ is dynamically coherent.
\end{proposition}

\section{The Lefschetz index}\label{app.Lefschetz}
 
Here we define the Lefschetz index and give the main property that we used.
We refer to the monograph by Franks \cite[Section 5]{Franks}
for details and other references.

For any space $X$ and subset $A\subset X$, we denote by $H_k(X,A)$ the $k$-th relative homology group with coefficients in $\ZZ$.

\begin{definition}
Let $V \subset \R^k$ be an open set and $F\colon V \subset \R^k
\to \R^k$ be a continuous map such that the
set of fixed point of $F$ is $\Gamma \subset V$, a compact
set. Then the \emph{Lefschetz index} of $F$, denoted by $I_{\Gamma}(F)$ is 
an element in $\ZZ \cong H_k(\R^k, \R^k - \{ 0 \})$, 
defined as follows. It is the image
by $(\mathrm{id} -F)_*\colon H_k(V,V- \Gamma) \to 
H_k(\R^k, \R^k - \{ 0 \})$ of the class $u_{\Gamma}$,
 where 
$u_{\Gamma}$ itself is the image
of the generator $1$ under the composite $H_k(\R^k, \R^k - D) \to
H_k(\R^k, \R^k - \Gamma) \cong 
H_k(V, V - \Gamma)$.
Here $D$ is a ball containing
$\Gamma$.
\end{definition}

It is easy to see that  if $\Gamma = \mathrm{Fix}(F) =
\Gamma_1 \cup \dots \cup \Gamma_j$, where $\Gamma_i$ are compact
and disjoint then 
$I_{\Gamma}(F) = \sum_1^j I_{\Gamma}(F)$. 
Here $I_{\Gamma}(F)$ is the index restricted to an open set
$V_i$ of $V$ which does not intersect the other $\Gamma_m$,
see \cite[Theorem 5.8 (b)]{Franks}.

This technical definition works well with the standard examples. For a single hyperbolic fixed point $q$, the index at $q$ is exactly 
$\mathrm{sgn}(\mathrm{det}(\mathrm{id} - D_qF))$ (see \cite[Proposition 5.7]{Franks}),
where $\mathrm{det}$ is the determinant, and $\mathrm{sgn}$ is the sign of the
determinant. Hence in dimension $2$ the index of a hyperbolic
fixed point when the orientation of the bundles is preserved is $-1$. This can be generalized to a $p$-prong
hyperbolic fixed point to obtain that the index is $1-p$.
This is because the index is invariant by homotopic changes.
A $p$-prong can be easily split into $p-1$ distinct 
hyperbolic points which are differentiable. 
In addition for any fixed set which behaves locally as a hyperbolic
fixed point, the index is the same as the hyperbolic
fixed point.

The main property we use is the following.

\begin{proposition}[Theorem 5.8(c) of \cite{Franks}]\label{p.lefschetz}
Let $P$ be a topological plane equipped with a metric $d$. Let $g,h\colon P \to P$ be two homeomorphisms. Suppose that there exists $R>0$ such that:
\begin{itemize}
\item For every $x\in P$, one has that $d(g(x),h(x))< R$;
\item There is a disk $D$ such that, for every $x \notin D$, one has that $d(x,g(x)) > 2R$. 
\end{itemize}
Then, the total index $I_{\mathrm{Fix}(g)}(g) = I_{\mathrm{Fix}(h)}(h)$. 
\end{proposition}

See also \cite[Section 8.6]{KH} for an alternate
presentation of the Lefschetz index.

\end{appendix}


\bibliographystyle{amsalpha_for_Rafael}
\bibliography{biblio}

\end{document}